\documentclass{article}

\usepackage{amsfonts,amsthm,amssymb,amsmath,mathrsfs} 
\usepackage{pgfplotstable} 
\usepackage{graphicx,multirow} 
\usepackage{booktabs} 
\usepackage[ruled,linesnumbered,commentsnumbered]{algorithm2e}
\usepackage{tikz} 
\usepackage{enumitem}
\usetikzlibrary{shapes, arrows.meta, positioning}
\usepackage{url} 
\usepackage{makecell,hhline}
\usepackage{stmaryrd}
\usepackage{colortbl}

\usepackage{rotating} 
\usepackage{textcomp} 
\usepackage{pgfplots} 
\usepackage{xcolor} 
\usepackage{appendix}  
\usepackage{lipsum} 
\usepackage{subcaption} 
\usepackage{array}
\usepackage{tabu}
 \usepackage{threeparttable}
\usepackage{hyperref}
\newcommand{\thickhline}{%
    \noalign {\ifnum 0=`}\fi \hrule height 1pt
    \futurelet \reserved@a \@xhline
}
\newcolumntype{"}{@{\hskip\tabcolsep\vrule width 1pt\hskip\tabcolsep}}
 
\usepackage{siunitx}
 
\usepackage{tikz}  
\usetikzlibrary{patterns}
\usetikzlibrary{decorations.pathreplacing}
\usepackage{pgfplots} 
 
\definecolor{blues1}{RGB}{198, 219, 239} 
\definecolor{blues2}{RGB}{158, 202, 225} 
\definecolor{blues3}{RGB}{107, 174, 214} 
\definecolor{blues4}{RGB}{49, 130, 189} 
\definecolor{blues5}{RGB}{8, 81, 156} 
\definecolor{blues6}{RGB}{2, 34, 78} 
\definecolor{bg}{RGB}{142, 207, 201} 
\definecolor{mygreen}{HTML}{0D8B43}
\definecolor{green1}{RGB}{0, 150, 0}
\definecolor{MyGreen}{RGB}{054,159,084}  
\definecolor{MyOrange}{RGB}{255,140,0}  
\definecolor{MyPurple}{RGB}{128,0,128}
\definecolor{MyBlue}{RGB}{037,069,128}  
\definecolor{MyYellow}{RGB}{171,110,031}  
\definecolor{workflowblue}{RGB}{202,222,240}

\pgfplotscreateplotcyclelist{mylist}{ 
{blues1}, 
{blues2}, 
{blues3}, 
{blues4}, 
{blues5}, 
{blues6}, 
} 
 
\usepackage[top=1in,bottom=1in,left=1in,right=1in]{geometry} 
\usepackage{natbib}
 
\numberwithin{equation}{section} 
\numberwithin{table}{section} 
\numberwithin{figure}{section} 

\newtheorem{theorem}{Theorem}[] 
\newtheorem{lemma}{Lemma}[]

\newtheorem{remark}{Remark}[section]

\theoremstyle{definition}

\newcommand{\T}{\mathrm{T}} 
 
\newcommand{\vect}[1]{\boldsymbol{#1}} 
\newcommand{\R}{\mathbb{R}} 

\newcommand{\sgn}{\text{sgn}}

\newcommand{\1}{\vect{1}} 

\newcommand{\I}{\mathcal{I}} 

\pgfplotsset{compat=1.18}

\begin{document}

\title{Fast projection onto the top-$k$-sum constraint}

\author{Jianting Pan\textsuperscript{a} \and Ming Yan\textsuperscript{a}}

\date{
  \textsuperscript{a}\,School of Data Science, The Chinese University of Hong Kong, Shenzhen
}


\maketitle
\begin{abstract}
This paper develops an efficient algorithm for computing the Euclidean projection onto the top-$k$-sum constraint, a key operation in financial risk management and matrix optimization problems. Existing projection methods rely on sorting and therefore incur an initial $O(n\log n)$ complexity, which limits their scalability in high-dimensional settings. To address this difficulty, we revisit the Karush-Kuhn-Tucker (KKT) conditions of the projection problem and introduce relaxed conditions that remain sufficient for characterizing the solution. These conditions lead to a simple geometric interpretation: finding the solutions is equivalent to locating the intersection of two monotone piecewise linear functions. Building on this insight, we propose an iterative and highly efficient algorithm that searches directly for the intersection point and completely avoids all sorting procedures. We prove that the algorithm converges globally and reaches the exact solution in a finite number of iterations. Extensive numerical experiments further demonstrate that the proposed algorithm substantially outperforms existing algorithms and exhibits empirical $O(n)$ complexity across a broad range of problem instances. A Julia implementation of the algorithm is available in our public GitHub repository~\footnote{\url{https://github.com/PanT12/top-k-sum}}.
\end{abstract}

\section{Introduction}
\label{sec: intro}
In this paper, we consider the Euclidean projection onto the top-$k$-sum, also known as max-$k$-sum~\cite{maxksum} constraint: 
\begin{equation}
\label{opt}
\begin{aligned}
    & \min_{\vect{x} \in \R^n}\ \frac{1}{2}\| \vect{x} - \vect{a} \|^2 \\
    & \text{subject to}\ x \in \mathcal{T}_{(k)}^r:=\left\{\vect{x} \in \R^n: \T_{(k)}(\vect{x}):=\sum_{i=1}^k \Vec{x}_i \leq r\right\}
\end{aligned}
\end{equation}
where $r \in \R$, $k \in \{1,2,\dots,n\}$ and sequence $\vect{a} \in \R^n$ are given, $\Vec{x}_1 \geq \Vec{x}_2 \geq \dots \geq \Vec{x}_n$ are the components of $\vect{x}$ in nonincreasing order. $\mathrm{T}_{(k)}(\vect{x})$ represents the sum of the largest $k$ elements of $\vect{x}$ and is closely related in many fields. For instance, it is conceptually linked to the $k$-winners-take-all (kWTA) problem, where recent distributed models focus on identifying the top $k$ elements through optimization to avoid centralized sorting~\cite{wang2023distributed,zhang2023single,zhang2022distributed}. In addition, in financial risk measurement~\cite{rockafellar2000optimization,krokhmal2002portfolio}, it arises in limiting a quantile of the distribution, as in the conditional value at risk (CVaR), also known as the average value at risk (AVaR) or superquantile, when $n$ scenarios are generated in this distribution with equal probability. Furthermore, it is relevant to matrix optimization problems involving a matrix’s Ky-Fan $k$-norm, defined as the sum of its $k$ largest singular values~\cite{overton1993optimality,pataki1998rank,pirzada2024ky,doan2022low}.  

Given that problem~\eqref{opt} is a strongly convex quadratic program, it is sufficient to derive the Karush-Kuhn-Tucker (KKT) conditions to uniquely characterize the solution. Based on this, Wu et al.~\cite{wu2014moreau} designed a finite-termination algorithm that performs an exhaustive grid-search to find all possible solutions until the KKT conditions are satisfied. However, this approach suffers from high computational complexity $O(k(n-k))$. In some applications, the value $k$ is usually set to $k = \lfloor (1-\tau)n\rfloor$ for some fixed proportion $\tau \in (0,1)$, leading to $O(n^2)$ complexity~\cite{roth2025n,roth2024fast}, which makes the grid-search process slow and impractical for large-scale problems.

To address this issue, Roth and Cui~\cite{roth2025n} introduced two finite-termination algorithms with $O(n)$ complexity, significantly alleviating the computational burden. The first reformulates the projection problem as a parametric linear complementarity problem (PLCP)~\cite{cottle2009linear} and solves it using a specialized pivoting method. The second method improves upon the grid-search technique by incorporating a refined updating rule that accelerates the search. More recently, Luxenberg et al.~\cite{luxenberg2025operator} proposed another $O(n)$ algorithm that iteratively decreases elements of $\vect{a}$ until the sum of the largest $k$ elements is equal to $r$. 

Despite their linear complexity, these algorithms all require sorting the input vector, which can become a performance bottleneck when $n$ is large. In some of their numerical experiments, the time spent on sorting can dominate the total runtime, and it sometimes exceeds the time taken by the algorithm itself. 
This motivates the need for more efficient methods that avoid sorting altogether.

In this paper, we propose a novel algorithm that eliminates the need for sorting the input vector, thereby significantly reducing the computational complexity associated with the Euclidean projection onto the top-$k$-sum constraint. We begin by examining the Karush-Kuhn-Tucker (KKT) conditions for this projection problem and introducing relaxed KKT conditions that are sufficient to characterize the solution. Leveraging these relaxed KKT conditions, we demonstrate that the projection solution can be interpreted geometrically as the intersection point of two monotone lines and introduce \textbf{E}fficient \textbf{I}ntersection \textbf{P}oint \textbf{S}earching (\textbf{EIPS}) to find this intersection point. In contrast to existing methods, which perform a grid search to identify the two relevant indices, our approach focuses on iteratively searching for the intersection point of both lines. This leads to a more efficient algorithm and avoids the need for sorting. Furthermore, we establish the global convergence of our algorithm, proving that the sequence generated by our iterative search converges to the intersection point in a finite number of steps. Finally, through extensive numerical experiments, we demonstrate that our algorithm is significantly faster than the existing methods and achieves empirical $O(n)$ complexity.

The paper is organized as follows. In Section~\ref{sec: kkt conditions}, we analyze the KKT conditions of the projection problem~\eqref{opt} and provide relaxed KKT conditions to eliminate the sorting requirement. The details of our proposed algorithm, as well as the main theoretical results and some acceleration tricks, are given in Section~\ref{sec: eips}. Finally, we show the numerical performance of EIPS and compare it with existing methods on a range of problems in Section~\ref{sec: experiment}. The paper concludes with a conclusion section.

\textbf{Notation} Through this paper, for any $\vect{x} \in \R^n$, we denote its $i$-th entry by $x_i$. The vector $\Vec{\vect{x}}$ denotes the components of $\vect{x}$ rearranged in descending order, i.e., $\Vec{x}_1 \geq \Vec{x}_2 \geq \cdots \geq \Vec{x}_n$, with the convention that $\Vec{x}_0 := +\infty$ and $\Vec{x}_{n+1} := -\infty$. The notation $|\vect{x}|$ refers to the vector in $\R^n$ whose $i$-th entry is $|x_i|$. The sign vector of $\vect{x}$ is denoted by $\sgn(\vect{x}) \in \R^n$, where $(\sgn(\vect{x}))_i = 1$ if $x_i \geq 0$ and $-1$ otherwise, for $i = 1,\dots,n$. Given two vectors $\vect{a}, \vect{b} \in \R^n$, we define $\min\{\vect{a}, \vect{b}\}$ as the vector in $\R^n$ with the $i$-th entry equal to $\min\{a_i, b_i\}$. For an index set $\I \subseteq [n]:= \{1,2,\dots,n\}$, we denote by $\vect{x}_{\I} \in \R^{|\I|}$ the subvector of $\vect{x}$ corresponding to the indices in $\I$, where $|\I|$ denotes the cardinality of $\I$.

For any thresholds $\theta_1, \theta_2 \in \R$, we define the index set
\[\I\left(\vect{x},\theta_1, \theta_2\right):=\{i\in[n]: \theta_1 \leq x_i < \theta_2\}.\] 
For convenience, we also define the one-sided index sets:
\[\I\left(\vect{x}, \theta_1, +\right):=\{i\in[n]:  \theta_1 \leq x_i\},\quad \I\left(\vect{x},-, \theta_2\right):=\{i\in[n]:  x_i < \theta_2\}.\] 
We denote by $\1_n$ the $n$-dimensional vector of all ones. The Hadamard (element-wise) product of two vectors $\vect{x}, \vect{y} \in \R^n$ is denoted by $\vect{x} \circ \vect{y} \in \R^n$, whose $i$-th entry is $(\vect{x} \circ \vect{y})_i = x_i y_i$.

\section{KKT conditions analysis}
\label{sec: kkt conditions}
In this section, we derive the KKT conditions of the projection problem~\eqref{opt} in Section~\ref{subsec: KKT conditions} and show how to solve it. In addition, we relax the sorting requirement of the KKT conditions in Section~\ref{subsec: relaxation}. 

\subsection{KKT conditions}
\label{subsec: KKT conditions}
The descending order in $\vect{a}$ is preserved in the optimal solution $\vect{x}^{\star}$; i.e., if $a_i \geq a_j$, then $x_i^{\star} \geq x_j^{\star}$. Otherwise, swapping the two values would result in an increase in the objective function value. The optimization problem~\eqref{opt} has a strongly convex objective function and linear constraints. Therefore, its KKT conditions are both necessary and sufficient for uniquely characterizing the solution $\vect{x}^{\star}$.

Its KKT conditions are derived in Appendix~\ref{app: kkt derivation}, and they can be described as follows:
\begin{subequations}
\label{kkt: kkt conditions}
\begin{align}
    \label{kkt: same area}& \sum_{i=1}^{k}\max\{u^{\star} - \vec{a}_i,0\} = \sum_{j=k+1}^n \max\{\vec{a}_j-l^{\star}, 0\},\\
    \label{kkt: top-k-sum constraint} &  \sum_{i=1}^n \max\{a_i - u^{\star}, 0\} + l^{\star}k=\min\left\{\T_{(k)}(\vect{a}), r\right\},
\end{align}
\end{subequations}
where $u^{\star} \geq \Vec{a}_k \geq l^{\star}$. 
Based on $\{a_i\}_{i=1}^n$, we define three sets ($\alpha,\beta,\gamma$) as
\begin{align}
\label{equ: related set} & \alpha := \{ i: a_i > u^{\star}\}, \quad \beta := \{i: u^{\star} \geq a_i \geq l^{\star}\}, \quad \gamma := \{i: a_i < l^{\star}\},
\end{align}
and we have the optimal solution
\begin{align}\label{kkt: solution form}
      &\vect{x}^{\star}_\alpha = \vect{a}_{\alpha} -(u^{\star}-l^{\star})\1_{|\alpha|}, \quad \vect{x}^{\star}_\beta = l^{\star}\1_{|\beta|}, \quad \vect{x}^{\star}_\gamma = \vect{a}_\gamma.
\end{align}

There are two special cases where \( k = 1 \) and \( k = n \), for which we can compute the optimal solution directly. When \( k = 1 \), the solution is given by \( \vect{x}^{\star} = \min\{r\1_n, \vect{a}\} \). When \( k = n \), we have \( a_i\geq u^{\star} \) for all \( i \), and 
\[
n(u^{\star} - l^{\star}) = \1_n^\top \vect{a} - \min\left\{\1_n^\top \vect{a}, r\right\} = \max\left\{\1_n^\top \vect{a}, r\right\} - r.
\]
Thus $\vect{x}^{\star} = \vect{a} +  \left(r - \max\{\1_n^\top \vect{a}, r\}\right)\1_n/n$.


However, finding the optimal solution is not straightforward for other \(k\) values. To solve it, \cite{roth2025n,luxenberg2025operator,wu2014moreau} exploited the properties of the KKT conditions derived above and introduced several grid-search procedures to find $(k_0:=|\alpha|,k_1:=|\alpha|+|\beta|)$ after the elements in $\vect{a}$ are sorted. Even though finding $(k_0, k_1)$ may be very fast, the time taken for sorting is greater than the time required to find $(k_0,k_1)$.

In this paper, we propose a method to find $(u^{\star},l^{\star})$ directly without sorting the elements in $\vect{a}$. To begin with, Figure~\ref{fig: kkt conditions} presents the KKT conditions~\eqref{kkt: kkt conditions} where the input $\Vec{\vect{a}}$ is sorted. In this figure, we identify $u^{\star}$ and $l^{\star}$ such that the summation of the red-shaded area equals $\min\left\{\T_{(k)}(\vect{a}), r\right\}$ (i.e., Equation~\eqref{kkt: top-k-sum constraint}) and the blue-shaded area is equal to the orange-shaded area (i.e., Equation~\eqref{kkt: same area}). After $(u^{\star},l^{\star})$ is found,  the optimal solution $\vect{x}^{\star}$ is obtained in~\eqref{kkt: solution form}.
This comes up with a question: \textbf{\emph{How to identify $(u^{\star}, l^{\star})$}?}



\begin{figure}[!ht]
    \centering
    \begin{minipage}[c]{0.45\linewidth}
    \centering
    \begin{tikzpicture}[scale = 0.7]
        \draw[thick, -latex] (0.25,0) -- (7,0) node[right] {};
        \draw[thick, -latex] (0.25,0) -- (0.25,6) node[above] {};
        \node[below] at (7,0) {Index};
        \node[above] at (0,6) {Value};
    
        \draw[thick, decorate,decoration={brace,amplitude=10pt,mirror,raise=4pt},yshift=-9pt] (0.5,0) -- (2.5,0) node[midway,yshift=-20pt] {$\alpha$};
    
    
        \draw[thick, decorate,decoration={brace,amplitude=10pt,mirror,raise=4pt},yshift=-9pt] (3,0) -- (6,0) node[midway,yshift=-20pt] {$\gamma$};
    
        \foreach \x in {1,...,4} {
            \draw [dashed] (\x/2, {6.2-\x/2})--(\x/2,0);
            \draw[thick] (\x/2-0.25, {6.2-\x/2}) -- (\x/2+0.25, {6.2-\x/2});
        }
    
        \foreach \x/\y in {5/3.3, 6/2.5, 7/2.2, 8/2, 9/1.8} {
            \draw [dashed] (\x/2, \y)--(\x/2,0);
            \draw[thick] (\x/2-0.25, \y) -- (\x/2+0.25, \y);
        }
    
        \foreach \x in {10,11,12} {
            \draw [dashed] (\x/2, 6.3-\x/2)--(\x/2,0);    
            \draw[thick] (\x/2-0.25, {6.3-\x/2}) -- (\x/2+0.25, {6.3-\x/2});
        }
    
        \fill[pattern=north east lines, pattern color=red] (0.25, 5.7) -- (0.75, 5.7) -- (0.75, 5.2) -- (1.25, 5.2) -- (1.25, 4.7) -- (1.75, 4.7) -- (1.75, 4.2) -- (2.25, 4.2) -- (2.25, 3.3) -- (2.75, 3.3) -- (2.75, 2.5) -- (3.25, 2.5) -- (3.25, 0) -- (0.25, 0);
        \node[below] at (1/2 ,0) {$1$};
        \node[below] at (5/2 ,0) {$k_0$};
        \node[below] at (6/2 ,0) {$k$};
        \node[below] at (12/2 ,0) {$n$};
    
        \begin{scope}[shift={(4.5,5.5)}]
            \draw[thick] (0,0) -- (0.5,0) node[right] {$\Vec{\vect{a}}$ value};
            \draw[thick, red] (0,-0.5) -- (0.5,-0.5) node[right, black] {$\Vec{\vect{x}}^{\star}$ value};
            \fill[pattern=north east lines, pattern color=red] (0, -1.5) rectangle (0.5,-1) node[right,black,yshift=-0.25cm] {$\leq r$~\eqref{kkt: top-k-sum constraint}};
        \end{scope}
    \end{tikzpicture}
    \end{minipage}
    \hfill
    \begin{minipage}[c]{0.45\linewidth}
    \centering
    \begin{tikzpicture}[scale = 0.7]
        \draw[thick, -latex] (0.25,0) -- (7,0) node[right] {};
        \draw[thick, -latex] (0.25,0) -- (0.25,6) node[above] {};
        \node[below] at (7,0) {Index};
        \node[above] at (0,6) {Value};
    
        \draw[thick, decorate,decoration={brace,amplitude=10pt,mirror,raise=4pt},yshift=-9pt] (0.5,0) -- (2,0) node[midway,yshift=-20pt] {$\alpha$};
    
        \draw[thick, decorate,decoration={brace,amplitude=10pt,mirror,raise=4pt},yshift=-9pt] (2.5,0) -- (4.5,0) node[midway,yshift=-20pt] {$\beta$};
    
        \draw[thick, decorate,decoration={brace,amplitude=10pt,mirror,raise=4pt},yshift=-9pt] (5,0) -- (6,0) node[midway,yshift=-20pt] {$\gamma$};
    
        \foreach \x in {1,...,4} {
            \draw[thick, -latex] (\x/2, {6.2-\x/2}) -- (\x/2, 4-\x/2);
            \ifnum \x=4
                \draw[dashed, blue] (\x/2, 4-\x/2)--(\x/2,0);
            \else
                \draw [dashed] (\x/2, 4-\x/2)--(\x/2,0);
            \fi
            \draw[thick, red] (\x/2-0.25, 4-\x/2) -- (\x/2+0.25, 4-\x/2);
            \draw[thick] (\x/2-0.25, {6.2-\x/2}) -- (\x/2+0.25, {6.2-\x/2});
        }
    
        \foreach \x/\y in {5/3.3, 6/2.5, 7/2.2, 8/2, 9/1.8} {
            \draw[thick, -latex] (\x/2, \y) -- (\x/2, 1.5);
            \ifnum \x=6
                \draw[dashed, blue] (\x/2, 1.5)--(\x/2,0);
            \else \ifnum \x=9
                \draw[dashed, blue] (\x/2, 1.5)--(\x/2,0);            
            \else
                \draw [dashed] (\x/2, 1.5)--(\x/2,0);
            \fi
            \fi
            \draw[thick, red] (\x/2-0.25, 1.5) -- (\x/2+0.25, 1.5);
            \draw[thick] (\x/2-0.25, \y) -- (\x/2+0.25, \y);
        }
    
        \foreach \x in {10,11,12} {
            \draw [dashed] (\x/2, 6.3-\x/2)--(\x/2,0);    
            \draw[thick] (\x/2-0.25, {6.3-\x/2}) -- (\x/2+0.25, {6.3-\x/2});
            \draw[thick, red] (\x/2-0.25, {6.3-\x/2}) -- (\x/2+0.25, {6.3-\x/2});
        }
    
        \fill[pattern=north east lines, pattern color=red] (0.25, 5.7) -- (0.75, 5.7) -- (0.75, 5.2) -- (1.25, 5.2) -- (1.25, 4.7) -- (1.75, 4.7) -- (1.75, 4.2) -- (2.25, 4.2) -- (2.25, 3.7) -- (0.25, 3.7);
        \fill[pattern=north east lines, pattern color=red] (0.25,0) rectangle (3.25,1.5);

        \fill[pattern=north west lines, pattern color=blue] (2.25, 3.7) -- (3.25, 3.7) -- (3.25, 2.5) -- (2.75, 2.5) -- (2.75, 3.3) -- (2.25, 3.3);
        \fill[pattern=north west lines, pattern color=orange] (3.25, 2.2) -- (3.75, 2.2) -- (3.75, 2) -- (4.25, 2) -- (4.25, 1.8) -- (4.75, 1.8) -- (4.75, 1.5) -- (3.25, 1.5);

        \node[left] at (0.25, 1.5) {$l^{\star}$};
        \node[left] at (0.25, 3.7) {$u^{\star}$};
        \node[below] at (1/2 ,0) {$1$};
        \node[below] at (4/2 ,0) {$k_0$};
        \node[below] at (6/2 ,0) {$k$};
        \node[below] at (9/2 ,0) {$k_1$};
        \node[below] at (12/2 ,0) {$n$};
    
        \begin{scope}[shift={(4.5,5.5)}]
            \draw[thick] (0,0) -- (0.5,0) node[right] {$\Vec{\vect{a}}$ value};
            \draw[thick, red] (0,-0.5) -- (0.5,-0.5) node[right, black] {$\Vec{\vect{x}}^{\star}$ value};
            \fill[pattern=north east lines, pattern color=red] (0, -1.5) rectangle (0.5,-1) node[right,black,yshift=-0.25cm] {$=r$~\eqref{kkt: top-k-sum constraint}};
            \fill[pattern=north west lines, pattern color=blue] (0, -2.5) rectangle (0.5,-2) node[right,black,yshift=-0.25cm] {$=$};
            \fill[pattern=north west lines, pattern color=orange] (1.2, -2.5) rectangle (1.7,-2) node[right,black,yshift=-0.25cm] {\eqref{kkt: same area}};
        \end{scope}
    \end{tikzpicture}
    \end{minipage}
    \caption{Illustration of the KKT conditions~\eqref{kkt: kkt conditions}: Schematic of sorted input $\Vec{\vect{a}}$ and sorted projection $\Vec{\vect{x}}^{\star}$, where $k_0=|\alpha|$ and $k_1=|\alpha|+|\beta|$. \textbf{(Left)}: $\T_{(k)}(\vect{a}) \leq r$. \textbf{(Right)}: $\T_{(k)}(\vect{a}) > r$.}
    \label{fig: kkt conditions}
\end{figure}
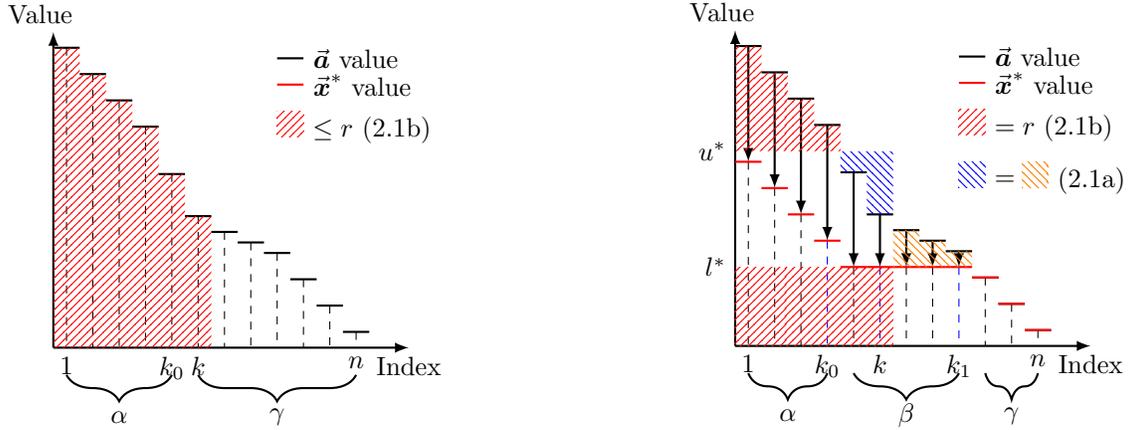


Let's analyze the two KKT conditions~\eqref{kkt: top-k-sum constraint} and \eqref{kkt: same area}. First, we define two equations
\begin{align}
    \label{equ: increasing function}  f(u,l) :=& \sum_{i=1}^n \max\{a_i - u, 0\} + lk - \min\left\{\T_{(k)}(\vect{a}), r\right\}, \\
    \label{equ: decreasing function}  g(u,l) := &\sum_{i=1}^{k}\max\{u -\Vec{a}_i,0\} - \sum_{j=k+1}^n \max\{\Vec{a}_j-l, 0\},
\end{align}
Solving the KKT conditions~\eqref{kkt: kkt conditions} becomes finding $(u^{\star},l^{\star})$ such that $f(u^{\star},l^{\star})=g(u^{\star},l^{\star})=0$ and $l^{\star}\leq \Vec{a}_k\leq u^{\star}$. 

Let
\begin{align}
F(u) = \frac{1}{k} \left(\min\left\{\T_{(k)}(\vect{a}), r\right\} - \sum_{i=1}^n \max\{a_i - u, 0\}\right)
\end{align}
and we have 
$$f(u,F(u))=0.$$

Given $u =\vec{a}_k$, we have $\sum_{i=1}^{k}\max\{u -\Vec{a}_i,0\}=0$. In this case, we have $g(u,l)=0$ for $l\in [\vec{a}_{k+1},\Vec{a}_k]$. We let $G(\vec{a}_k)=[\vec{a}_{k+1},\vec{a}_k]$. Otherwise, if $u>\vec{a}_k$, we have $\sum_{i=1}^{k}\max\{u -\Vec{a}_i,0\}>0$, and there is an unique $l$ such that $g(u,l)=0$. For $u > \vec{a}_k$, we define 
\begin{align}
    &\quad G(u):= \{l: g(u, l)=0\}. 
\end{align}

We then provide the properties of $F(u)$ and $G(u)$. 
\begin{lemma}
\label{thm: f property}
    $F(u)$ is piecewise linear, nondecreasing, and concave. Furthermore, $F(u) \leq u$.
\end{lemma}
\begin{proof}
It is straightforward to observe that $F(u)$ is nondecreasing. Additionally, since $F(u)$ is expressed as the negative sum of the pointwise maximum of affine functions, it follows that $F(u)$ is piecewise linear and concave. 

Moreover, we can verify that $F(u) - u$ is nondecreasing for $u \leq \vec{a}_k$ and nonincreasing for $u \geq \vec{a}_k$. As a result, we have 
\[
F(u) - u \leq F(\vec{a}_k) - \vec{a}_k =\frac{1}{k}\min\left\{r-\T_{(k)}(\vect{a}),0\right\}\leq 0, \quad \text{for all } u.
\]
The lemma is proved.
\end{proof}

\begin{lemma}
\label{thm: g property}
    $G(\vec{a}_k) = [\vec{a}_{k+1}, \vec{a}_k]$ and $G(u)$ is piecewise linear, decreasing, and continuous for $u > \vec{a}_k$. 
\end{lemma}
\begin{proof}
    It is straightforward to observe that $G(u)$ is continuous and piecewise linear because the function $g(u,l)$ is both continuous and piecewise linear. For $u > \vec{a}_k$, as $u$ increases, the term $\sum_{i=1}^{k} \max\{u - \vec{a}_i, 0\}$ also increases. To ensure that $g(u, l) = 0$, the term $\sum_{j=k+1}^{n} \max\{\vec{a}_j - l, 0\}$ must increase, which necessitates that $l \leq \vec{a}_{k+1}$ and the function $G(u)$ decreases.
\end{proof}    

Figure~\ref{fig: gu and fu} summarizes Lemmas~\ref{thm: f property}-\ref{thm: g property} and illustrates the monotonicities of $F(u)$ and $G(u)$. Both figures use the same input $\vect{a}$ and $k$, ensuring that $G(u)$ remains unchanged, while the only difference between the two figures lies in the $r$ values.

In the left panel, where $r \geq \T_{(k)}(\vect{a})$, we have $F(\vec{a}_k) = \vec{a}_k$ and $G(\vec{a}_k) = [\vec{a}_{k+1}, \vec{a}_k]$. In this case, the black point with $l^{\star} = u^{\star}$ satisfies the KKT conditions~\eqref{kkt: kkt conditions}. Note that the function $F(u)$ does not change with $r$ if $r\geq \T_{(k)}(\vect{a})$.

In the right panel, where $r < \T_{(k)}(\vect{a})$, we have $F(u) < u$ for all $u$, which implies that the solution satisfies $l^{\star} < u^{\star}$. Although the figure shows $u^{\star} > \vec{a}_k$, it is possible for the solution to satisfy $l^{\star} < u^{\star} = \vec{a}_k$.


\begin{figure}[!ht]
\begin{minipage}[c]{0.45\linewidth}
    \centering
    \begin{tikzpicture}[scale=0.65]
        \draw[thick, -latex] (-1,0) -- (6.5,0) node[right] {$u$};
        \draw[thick, -latex] (0,-1) -- (0,4.8) node[above] {$l$};
        
        \draw[dashed] (-0.5, -0.5) -- (4.8, 4.8) node[above] {$l = u$};
    
        \def\xvalues{{0.8, 1.3, 1.8, 2, 2.2, 2.5, 3.3, 4.2, 4.7, 5.2, 5.7, 6.0}}
        
        \xdef\ystart{0.1}
        \foreach \i in {0,1,...,10} {
            \pgfmathsetmacro{\value}{\ystart + (11-(\i+1))/6*(\xvalues[\i+1] - \xvalues[\i])}
            \draw[very thick, blue] (\xvalues[\i], \ystart) -- (\xvalues[\i+1], \value);
            \xdef\ystart{\value}
            \ifnum \i=0
                \draw[dashed] (\xvalues[\i], 0.1) -- (\xvalues[\i],0) node[below] {\footnotesize $\Vec{a}_n$};
            \fi
            \ifnum \i=9
                \draw[dashed] (\xvalues[\i+1], \ystart) -- (\xvalues[\i+1], 0) node[below] {\footnotesize $\Vec{a}_1$};
                \draw[dashed] (\xvalues[\i+1], \value) -- (0, \value) node[left] {$r/k$};
            \fi
            \ifnum \i=10
                \node[blue, right] at (\xvalues[\i+1], \ystart) {$F(u)$};
            \fi
    
        }

        \draw[very thick, blue] (0.3, -1) -- (0.8, 0.1);
    
        \draw[very thick, red] (2.5, 2.5) -- (2.5,2.2) -- (3,2) -- (3.3, 1.9) -- (4, 1.2) -- (4.2, 1.1) -- (4.7, 0.5) -- (5.2, -0.4) -- (5.7, -0.9) node[below] {$G(u)$};
        \draw[dashed] (2.5,2.2) -- (2.5,0) node[below] {\footnotesize $u^{\star}$};    
        \draw[dashed] (2.5,2.5) -- (0,2.5) node[left] {$l^{\star}$};
        \fill (2.5,2.5) circle (4pt);

    \end{tikzpicture}
\end{minipage}
\hfill
\begin{minipage}[c]{0.45\linewidth}
    \centering
    \begin{tikzpicture}[scale=0.65]
        \draw[thick, -latex] (-1,0) -- (6.5,0) node[right] {$u$};
        \draw[thick, -latex] (0,-1) -- (0,4.8) node[above] {$l$};
        
        \draw[dashed] (-0.5, -0.5) -- (4.8, 4.8) node[above] {$l = u$};
    
        \def\xvalues{{0.8, 1.3, 1.8, 2, 2.2, 2.5, 3.3, 4.2, 4.7, 5.2, 5.7, 6.0}}
        
        \xdef\ystart{-0.5}
        \foreach \i in {0,1,...,10} {
            \pgfmathsetmacro{\value}{\ystart + (11-(\i+1))/6*(\xvalues[\i+1] - \xvalues[\i])}
            \draw[very thick, blue] (\xvalues[\i], \ystart) -- (\xvalues[\i+1], \value);
            \xdef\ystart{\value}
            \ifnum \i=0
                \draw[dashed] (0.8,-0.5) -- (0.8,0) node[anchor=north east, xshift=5pt] {\footnotesize $\Vec{a}_n$};
            \fi
            \ifnum \i=9
                \draw[dashed] (\xvalues[\i+1], \ystart) -- (\xvalues[\i+1], 0) node[below] {\footnotesize $\Vec{a}_1$};
                \draw[dashed] (\xvalues[\i+1], \value) -- (0, \value) node[left] {$r/k$};
            \fi
            \ifnum \i=5
                \draw[dashed] (\xvalues[\i+1], \value) -- (\xvalues[\i+1], 0) node[below,xshift=3pt] {\footnotesize $\Vec{a}_{k_0}$};
            \fi
            \ifnum \i=10
                \node[blue, right] at (\xvalues[\i+1], \ystart) {$F(u)$};
            \fi
    
        }
        \draw[very thick, blue] (0.4, -1.3) -- (0.8, -0.5);
    
        \draw[very thick, red] (2.5, 2.5) -- (2.5,2.2) -- (3,2) -- (3.3, 1.9) -- (4, 1.2) -- (4.2, 1.1) -- (4.7, 0.5) -- (5.2, -0.4) -- (5.7, -0.9) node[below] {$G(u)$};
        \draw[dashed] (2.5,2.2) -- (2.5,0) node[below, xshift=-3pt] {\footnotesize $\Vec{a}_{k}$};        
              
        \fill (2.75,2.1) circle (4pt);
        \draw[dashed] (2.75,2.1) -- (2.75,0) node[below,xshift=3pt] {\footnotesize $u^{\star}$};
        \draw[dashed] (2.75,2.1) -- (0,2.1) node[left] {$l^{\star}$};
    \end{tikzpicture}
\end{minipage}
\caption{The illustrations of $F(u)$ and $G(u)$ for different $r$ values, where $\vect{a}$ (with $\vec{a}_k > \vec{a}_{k+1}$) and $k$ are fixed. The function $G(u)$ remains unchanged, while $F(u)$ shifts vertically based on the value of $r$. {\bf (Left)}: $r \geq \T_{(k)}(\vect{a})$, the black point is on the line $l=u$. {\bf (Right)}: $r < \T_{(k)}(\vect{a})$.
}
\label{fig: gu and fu}
\end{figure}
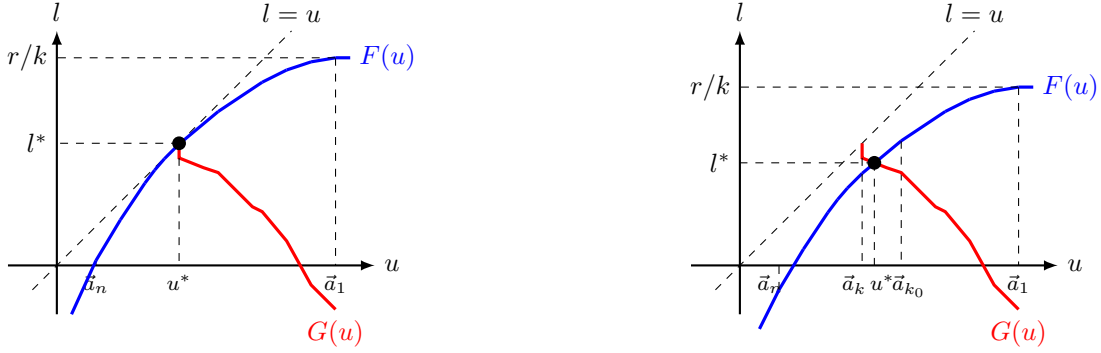

\subsection{Relaxed KKT conditions}
\label{subsec: relaxation}
It is worthwhile to point out that both the KKT conditions~\eqref{kkt: same area} and~\eqref{kkt: top-k-sum constraint} require partial sorting to compute $\T_{(k)}(\vect{a})$ or $\vec{a}_k$, which can be computationally expensive, especially when $k$ is large. To address this issue and avoid the sorting requirement, we introduce relaxed KKT conditions and demonstrate their relationship with the original KKT conditions.

To begin, we redefine $F(u)$ by removing the dependence on $\T_{(k)}(\vect{a})$ and introduce the function:
\begin{align}
    \label{equ: f' increasing function} 
    F_r(u) := \frac{1}{k} \left(r - \sum_{i=1}^n \max\{a_i - u, 0\}\right) = F(u) + \frac{1}{k}\max\left\{r - \T_{(k)}(\vect{a}), 0\right\}\geq F(u).
\end{align}
We observe the following:
\begin{itemize}
    \item If $r < \T_{(k)}(\vect{a})$, then $F_r(u) = F(u)<u$, maintaining equivalence with the original definition.
    \item If $r=\T_{(k)}(\vect{a})$, we have $F_r(u)=F(u)=u$ for $u\in [\vec{a}_{k+1},\vec{a}_k]$ and $F_r(u)=F(u)< u$ otherwise.
    \item If $r > \T_{(k)}(\vect{a})$, we have $F_r(u)>F(u)=u$ for $u\in [\vec{a}_{k+1},\vec{a}_k]$. That is, there exists at least one $u$ such that $F_r(u)>u$.
\end{itemize}

\begin{lemma}
\label{thm: fr property}
    $F_r(u)$ is piecewise linear, nondecreasing, and concave. In addition, we have its one-sided derivatives as
    \[F_r'(u^-)= \frac{|\I\left(\vect{a}, u,+\right)|}{k},\quad F_r'(u^+)= \frac{|\I\left(\vect{a}, u, +\right)|-|\{i: a_i=u\}|}{k}.\]
\end{lemma}
\begin{proof}
    For given $\vect{a}$ and $r$, the difference $F_r(u)-F(u)$ is fixed. So $F_r(u)$ is also piecewise linear, nondecreasing, and concave. The one-sided derivatives can be obtained from the definition of $F_r(u)$.
\end{proof}

Second, we eliminate the sorting requirement in the definition of $g(u,l)$ and instead define the relaxed function $g_r(u,l)$ as follows:
\begin{align}
    g_r(u,l) := \sum_{i=1}^{n}\max\{a_i - u, 0\} - \sum_{i=1}^{n}\max\{a_i - l, 0\} + k(u - l).
\end{align}

It is straightforward to verify that $g_r(u,l) = g(u,l)$ for $(u, l) \in (\vec{a}_{k+1}, +\infty) \times (-\infty, \vec{a}_k]$. Furthermore, we have $g_r(u,l) = g(u,l) = 0$ for $(u, l) \in (\vec{a}_{k+1}, \vec{a}_k] \times (\vec{a}_{k+1}, \vec{a}_k]$. 
In the subsequent discussion, we restrict our attention to the region $u > \vec{a}_{k+1}$.

For a given $u \in (\vec{a}_{k+1}, \vec{a}_k]$, there may be multiple values of $l$ such that $g_r(u,l) = 0$. To preserve the nonincreasing property, we select the smallest such $l$. Based on this, we define the relaxed function as follows:
\begin{align}
    G_r(u) := \min \{l : g_r(u, l) = 0\}.
\end{align}
It is worth noting that $G_r(u) = \vec{a}_{k+1}$ for all $u\in(\vec{a}_{k+1},\vec{a}_k]$, and $G_r(u) < u$ for $u > \vec{a}_{k+1}$.




Therefore, $F_r(u)$ and $G_r(u)$ are unique for a given $u$, with $G_r(u)$ defined over $u > \vec{a}_{k+1}$ and always lying below line $u = l$. However, $F_r(u)$ may have points above the line $u = l$. The main objective is to determine whether an intersection between $F_r(u)$ and $G_r(u)$ exists, and if so, to find the intersection, i.e., solve $F_r(u) = G_r(u)$.
Once such an intersection exists and is successfully identified, we can obtain $(u^{\star}, l^{\star})$ that satisfies the KKT conditions~\eqref{kkt: kkt conditions}. 


    
    


To formalize above mapping, we summarize the relationship between the intersection point and the KKT solution in the following Lemma.

\begin{lemma}
\label{lemma: solution recovery}
Suppose \(u_r^{\star}\) is a solution of $F_r(u)=G_r(u)$ if it exists. Let $l_r^{\star}:=F_r(u_r^{\star})$, $m_r:=|\I(\vect{a},u_r^{\star},+)|$. Then $m_r \leq k$ and the solution \((u^{\star},l^{\star})\) to the KKT conditions~\eqref{kkt: kkt conditions} can be recovered as follows.

\begin{enumerate}[label=(Case \arabic*)]
    \item \label{case1} \textbf{(Inactive Case)} If \(u_r^{\star}\) does not exist, then \(r \geq \T_{(k)}(\vect{a})\), and the optimal solution is $\vect{x}^{\star}=\vect{a}$. In this case, $u^{\star}=l^{\star}$. 

    \item \label{case2} \textbf{(Interior Case)} If \(F_r(u_r^{\star})<u_r^{\star}\) and \(m_r<k\), then \(u_r^{\star}>\vec{a}_k\). In this case, 
    $u^{\star}=u_r^{\star}$, $l^{\star}=F(u^{\star})=F_r(u^{\star})$. Moreover, $u^{\star}-l^{\star}>\vec{a}_k-\vec{a}_{k+1}$.

    \item \label{case3} \textbf{(Boundary Case)} If \(F_r(u_r^{\star})<u_r^{\star}\) and \(m_r=k\), then \(\vec{a}_{k+1}<\vec{a}_k\) and \(u_r^{\star} \in (\vec{a}_{k+1}, \vec{a}_k]\). In this case, 
    $u^{\star}=\vec{a}_k=\min\{a_i: a_i\ge u_r^{\star}\}$, $l^{\star}=F(u^{\star})$. Moreover, $0<u^{\star}-l^{\star}\leq \vec{a}_k-\vec{a}_{k+1}$.
\end{enumerate}
\end{lemma}

\begin{proof}
Given that $G_r(u)$ is defined for $u>\Vec{a}_{k+1}$, we have $m_r \leq k$ if $u_r^{\star}$ exists. We then prove the three cases separately.

(Case 1) If \(u_r^{\star}\) does not exist, then \(F_r(u)>G_r(u)\) for all admissible \(u\).
By the definition of \(F_r\), this can only happen when \(r \geq \T_{(k)}(\vect{a})\). Therefore, the projection is trivial and \(\vect{x}^{\star}=\vect{a}\), which implies \(u^{\star}=l^{\star}\).

(Case 2) Since \(m_r<k\), we have \(u_r^{\star}>\vec{a}_k\). Hence \(F_r(u)=F(u)\) and \(G_r(u)=G(u)\) at \(u=u_r^{\star}\).
Therefore,
\[
F(u_r^{\star})=F_r(u_r^{\star})=G_r(u_r^{\star})=G(u_r^{\star}),
\]
so \(u^{\star}=u_r^{\star}\) and \(l^{\star}=F(u^{\star})=F_r(u^{\star})\) satisfy the KKT conditions. Since
\(u^{\star}>\vec{a}_k\) and \(l^{\star}\leq \vec{a}_{k+1}\), it follows that $u^{\star}-l^{\star}>\vec{a}_k-\vec{a}_{k+1}$.

(Case 3) Since \(m_r=k\), we have \(F_r(u) = F(u)\) and \(u_r^{\star}\in(\vec{a}_{k+1},\vec{a}_k]\),
which implies \(\vec{a}_{k+1}<\vec{a}_k\). In this case, \(G_r(u_r^{\star})=\vec{a}_{k+1}\), while $G(u)$ is only defined for $u \geq \Vec{a}_k$ in Lemma~\ref{thm: g property}. Since $G(\Vec{a}_k) = [\Vec{a}_{k+1}, \Vec{a}_k]$, the correct KKT value is therefore recovered by taking
\[
u^{\star}=\vec{a}_k=\min\{a_i:a_i\ge u_r^{\star}\},
\qquad
l^{\star}=F(u^{\star}).
\]
Then \((u^{\star},l^{\star})\) satisfies the KKT conditions. Moreover, since
\(u^{\star}=\vec{a}_k\) and \(l^{\star}\ge \vec{a}_{k+1}\), we obtain
$0<u^{\star}-l^{\star}\le \vec{a}_k-\vec{a}_{k+1}$.
\end{proof}

The geometric intuition for these three cases is illustrated in Figure~\ref{fig: Gru and Fru}. 

\begin{figure}[!ht]
\begin{minipage}[c]{0.32\linewidth}
    \centering
    \begin{tikzpicture}[scale=0.5]
        \draw[thick, -latex] (-1,0) -- (6.5,0) node[right] {$u$};
        \draw[thick, -latex] (0,-1) -- (0,4.8) node[above] {$l$};
        \draw[thick, densely dashed, gray, fill=bg!50] (1.8, 2) rectangle (2.8,2.5); 
        \draw[dashed] (-0.5, -0.5) -- (4.8, 4.8) node[above] {$l = u$};
    
        \def\xvalues{{0.8, 1.3, 1.8, 2, 2.2, 2.5, 3.3, 4.2, 4.7, 5.2, 5.7, 6.0}}
        
        \xdef\ystart{0.3}
        \foreach \i in {0,1,...,10} {
            \pgfmathsetmacro{\value}{\ystart + (11-(\i+1))/6*(\xvalues[\i+1] - \xvalues[\i])}
            \draw[very thick, blue] (\xvalues[\i], \ystart) -- (\xvalues[\i+1], \value);
            \xdef\ystart{\value}
            \ifnum \i=0
                \draw[dashed] (\xvalues[\i], 0.3) -- (\xvalues[\i],0) node[below] {\footnotesize $\Vec{a}_n$};
            \fi
            \ifnum \i=9
                \draw[dashed] (\xvalues[\i+1], \ystart) -- (\xvalues[\i+1], 0) node[below] {\footnotesize $\Vec{a}_1$};
                \draw[dashed] (\xvalues[\i+1], \value) -- (0, \value) node[left] {$r/k$};
            \fi
    
        }

        \draw[very thick, blue] (0.8, 0.3) -- (0.3, -0.8) node[below, blue] {$F_r(u)$};
        
        \draw[thick, ->,>=latex, gray] (2.9, 2.2) -- (3.4, 2.4);
        \begin{scope}[shift={(3.5,2)}]
            \draw[thick, densely dashed, gray, fill=bg!50] (0.0,0.0) rectangle (2,1); 
            \draw[densely dashed] (0.5,0) -- (1.5, 1);
            \draw[very thick, blue] (0.2, 0) --(0.5, 0.5) -- (1, 1);
            \draw[very thick, red] (0.9, 0.4) -- (1.5, 0.4) -- (2, 0.2);
            \draw[densely dashed] (1.5, 0.4) -- (1.5, 0);     
        \end{scope}
        \draw[very thick, red] (2.2, 2.2) -- (2.5,2.2) -- (3,2) -- (3.3, 1.9) -- (4, 1.2) -- (4.2, 1.1) -- (4.7, 0.5) -- (5.2, -0.4) -- (5.7, -0.9) node[below] {$G_r(u)$};
        \draw[dashed] (2.5,2.2) -- (2.5,0) node[below] {\footnotesize $\Vec{a}_{k}$};        
    \end{tikzpicture}
\end{minipage}
\hfill
\begin{minipage}[c]{0.32\linewidth}
    \centering
    \begin{tikzpicture}[scale=0.5]
        \draw[thick, -latex] (-1,0) -- (6.5,0) node[right] {$u$};
        \draw[thick, -latex] (0,-1) -- (0,4.8) node[above] {$l$};
        \draw[thick, densely dashed, gray, fill=bg!50] (2.5, 1.6) rectangle (3,2.6); 
        \draw[dashed] (-0.5, -0.5) -- (4.8, 4.8) node[above] {$l = u$};
    
        \def\xvalues{{0.8, 1.3, 1.8, 2, 2.2, 2.5, 3.3, 4.2, 4.7, 5.2, 5.7, 6.0}}
        
        \xdef\ystart{-0.5}
        \foreach \i in {0,1,...,10} {
            \pgfmathsetmacro{\value}{\ystart + (11-(\i+1))/6*(\xvalues[\i+1] - \xvalues[\i])}
            \draw[very thick, blue] (\xvalues[\i], \ystart) -- (\xvalues[\i+1], \value);
            \xdef\ystart{\value}
            \ifnum \i=0
                \draw[dashed] (0.8,-0.5) -- (0.8,0) node[anchor=north east, xshift=5pt] {\footnotesize $\Vec{a}_n$};
            \fi
            \ifnum \i=9
                \draw[dashed] (\xvalues[\i+1], \ystart) -- (\xvalues[\i+1], 0) node[below] {\footnotesize $\Vec{a}_1$};
                \draw[dashed] (\xvalues[\i+1], \value) -- (0, \value) node[left] {$r/k$};
            \fi
            \ifnum \i=5
                \draw[dashed] (\xvalues[\i+1], \value) -- (\xvalues[\i+1], 0) node[below,xshift=8pt] {\footnotesize $\Vec{a}_{k_0}$};
            \fi
    
        }
        \draw[very thick, blue] (0.8, -0.5) -- (0.4, -1.3) node[below, blue] {$F_r(u)=F(u)$};
    
        \draw[very thick, red] (2.2, 2.2) -- (2.5,2.2) -- (3,2) -- (3.3, 1.9) -- (4, 1.2) -- (4.2, 1.1) -- (4.7, 0.5) -- (5.2, -0.4) -- (5.7, -0.9) node[below] {$G_r(u)$};
        \draw[dashed] (2.5,2.2) -- (2.5,0) node[below, xshift=-5pt] {\footnotesize $\Vec{a}_{k}$};        

        \draw[thick, ->,>=latex, gray] (3.2, 2) -- (4, 2);
        \begin{scope}[shift={(4.3,1.1)}]
            \draw[thick, densely dashed, gray, fill=bg!50] (0.0,0.0) rectangle (1,2); 
            \draw[densely dashed] (0,1.8) -- (0.2, 2);
            \draw[very thick, blue] (0, 0.6) --(1, 1.4);
            \draw[very thick, red] (0, 1.2) -- (1, 0.9);
            \draw[densely dashed] (0.55, 1.05) -- (0.55, 0);     
            \fill (0.55, 1.05) circle (3pt);
        \end{scope}
              
        \fill (2.75,2.1) circle (3pt);
        \draw[dashed] (2.75,2.1) -- (2.75,0) node[below,xshift=3pt] {\footnotesize $u^{\star}$};
        \draw[dashed] (2.75,2.1) -- (0,2.1) node[left] {$l^{\star}$};
    \end{tikzpicture}
\end{minipage}
\hfill
\begin{minipage}[c]{0.32\linewidth}
    \centering
    \begin{tikzpicture}[scale=0.5]
        \draw[thick, -latex] (-1,0) -- (6.5,0) node[right] {$u$};
        \draw[thick, -latex] (0,-1) -- (0,4.8) node[above] {$l$};
        \draw[thick, densely dashed, gray, fill=bg!50] (2.15,1.8) rectangle (2.65,2.8); 
        
        \draw[dashed] (-0.5, -0.5) -- (4.8, 4.8) node[above] {$l = u$};
        \def\xvalues{{0.8, 1.3, 1.8, 2, 2.2, 2.5, 3.3, 4.2, 4.7, 5.2, 5.7, 6.0}}
        
        \xdef\ystart{0}
        \foreach \i in {0,1,...,10} {
            \pgfmathsetmacro{\value}{\ystart + (11-(\i+1))/6*(\xvalues[\i+1] - \xvalues[\i])}
            \draw[very thick, blue] (\xvalues[\i], \ystart) -- (\xvalues[\i+1], \value);
            \xdef\ystart{\value}
            \ifnum \i=0
                \draw[dashed] (\xvalues[\i], 0) -- (\xvalues[\i],0) node[below, xshift=-4pt] {\footnotesize $\Vec{a}_n$};
            \fi
            \ifnum \i=9
                \draw[dashed] (\xvalues[\i+1], \ystart) -- (\xvalues[\i+1], 0) node[below] {\footnotesize $\Vec{a}_1$};
                \draw[dashed] (\xvalues[\i+1], \value) -- (0, \value) node[left] {$r/k$};
            \fi
            \ifnum \i=4
                \draw[dashed] (\xvalues[\i+1], \value) -- (\xvalues[\i+1], 0) node[below,xshift=6pt] {\footnotesize $\Vec{a}_{k}$};
            \fi
    
        }
        \draw[very thick, blue] (0.8, 0) -- (0.4, -0.8) node[below, blue] {$F_r(u)=F(u)$};
    
        \draw[very thick, red] (2.2, 2.2) -- (2.5,2.2) -- (3,2) -- (3.3, 1.9) -- (4, 1.2) -- (4.2, 1.1) -- (4.7, 0.5) -- (5.2, -0.4) -- (5.7, -0.9) node[below] {$G_r(u)$};
        \fill[orange] (2.5,2.37) circle (3pt);

        \draw[thick, ->,>=latex, gray] (3,2.3) -- (4, 2.3);
        \begin{scope}[shift={(4.3,1.1)}]
            \draw[thick, densely dashed, gray, fill=bg!50] (0.0,0.0) rectangle (1,2); 
            \draw[densely dashed] (0,0.8) -- (1, 1.8);
            \draw[very thick, blue] (0, 0.5) -- (0.2, 0.75) -- (0.8, 1.35) -- (1, 1.5);
            \draw[very thick, red] (0.1, 0.9) -- (0.8, 0.9) -- (1, 0.8);
            \draw[densely dashed] (0.8, 1.35) -- (0.8, 0);
            \fill[orange] (0.8, 1.35) circle (3pt);
            \draw[densely dashed] (0.35, 0.9) -- (0.35, 0);        
            \fill (0.35, 0.9) circle (3pt);
        \end{scope}
        
        \fill (2.33,2.2) circle (3pt);
        \draw[dashed] (2.33,2.2) -- (2.33,0) node[below,xshift=-3pt] {\footnotesize $u_r^{\star}$};
        \draw[dashed] (2.33,2.2) -- (0,2.2) node[left] {$l^{\star}_r$};
    \end{tikzpicture}
\end{minipage}
\caption{The illustrations of $F_r(u)$ and $G_r(u)$ with different $r$ values. \textbf{(Left)}: \ref{case1}. \textbf{(Middle)}: \ref{case2}. \textbf{(Right)}: \ref{case3}.}
\label{fig: Gru and Fru}
\end{figure}
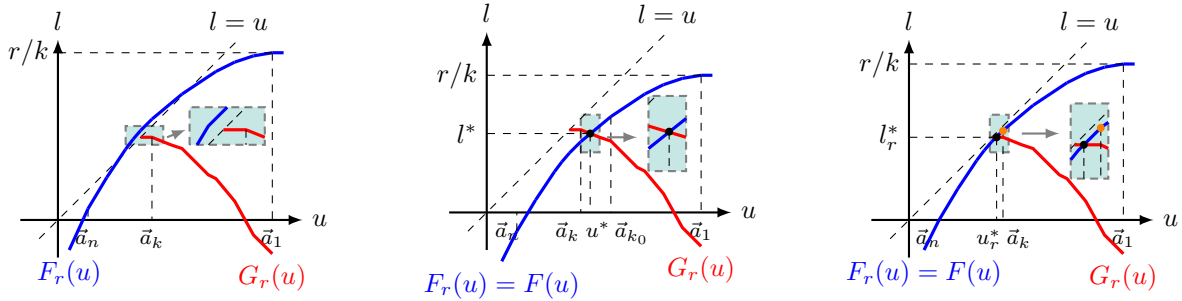

\section{Efficient Intersection Point Searching}
\label{sec: eips}
In this section, we first introduce the computation of $G_r(u)$ in Section~\ref{subsec: solving gu}, a fundamental component of our proposed framework. Subsequently, we present our proposed method, \textbf{E}fficient \textbf{I}ntersection \textbf{P}oint \textbf{S}earching (\textbf{EIPS}) in Section~\ref{subsec: cips implementation}, which is designed to pinpoint the intersection point that characterizes the top-$k$-sum projection solution with two practical accelerating techniques that further reduce the computational effort in Section~\ref{subsec: ACC tricks}. Finally, we show its theoretical guarantee in Section~\ref{subsec: eips analysis}. 

\subsection{Calculating $G_r(u)$}
\label{subsec: solving gu}

The iterative procedure for calculating $G_r(u)$ in Algorithm~\ref{alg: G-Searching} is a variant of the sorting-free projection method developed by Condat~\cite{condat2016fast}. Similarly, we employ a pivot-based partitioning strategy to iteratively update a lower bound $\rho$ that converges to the exact value of $G_r(u)$ without requiring a full sort. 

To implement this strategy, we first reformulate $g_r(u,l)$ into a structure that reveals its dependence on the active elements relative to $u$ and $l$. Specifically, for a given $u \geq l$, $g_r(u,l)$ can be rewritten as follows:

\begin{equation}
\begin{aligned}
    g_r(u,l) & = \sum_{i=1}^{n}\max\{a_i - u, 0\} - \sum_{i=1}^{n}\max\{a_i - l, 0\} + k(u - l) \\
        & = \sum_{j \in \I(\vect{a},u,+)} (a_j - u) - \sum_{i \in \I(\vect{a},u,+) \cup \I(\vect{a}, -, u)} \max\{a_i - l, 0\} + k(u-l) \\
        & = (k-m)(u-l) - \sum_{i \in \I(\vect{a}, l, u)} (a_i - l),
\end{aligned}
\end{equation}
where $m := \left|\I(\vect{a},u,+)\right|$. Since $u > \Vec{a}_{k+1}$, it follows that $m \leq k$.

Consider the case when $m = k$, which means $\Vec{a}_{k+1} < u \leq \Vec{a}_k$, from the definition of $G_r(u)$, we know
\(G_r(u) = \max \{a_i: a_i < u\}=\Vec{a}_{k+1}.\)

Now, let's examine the case when $m < k$, i.e., $u > \Vec{a}_k$. In this case, we know $\left|\I(\vect{a}, G_r(u), u)\right| > \left|\I(\vect{a}, \Vec{a}_{k+1}, u)\right| >  k-m$, $G_r(u) < \Vec{a}_{k+1}$ and $G_r(u)$ is unique, i.e., 
\begin{equation}
    G_r(u) = \frac{\sum_{i \in \I(\vect{a}, G_r(u), u)}a_i - (k-m)u}{|\I(\vect{a}, G_r(u), u)| - (k- m)}.
\end{equation}

In order to find $G_r(u)$, we construct its lower bound and increase the bound. Note that for any sequence $\vect{v}$ such that $\vect{b}_T := \{a_i: i \in \I(\vect{a}, G_r(u), u)\} \subseteq \vect{v} \subseteq \vect{b} := \{a_i : i \in \I(\vect{a}, -, u)\}$, by setting
\begin{equation}
\label{equ: rho calculate}
    \rho = \frac{\sum_{a \in \vect{v}}a - (k-m)u}{|\vect{v}| - (k - m)},
\end{equation}
we have
\begin{equation}
\label{equ: lower bound}
    \rho - G_r(u) = \frac{\sum_{a \in \vect{v}} (a-G_r(u)) - \sum_{a \in \vect{b}_T} (a-G_r(u)) }{|\vect{v}| - (k-m)} \leq 0. 
\end{equation}
Hence, $\rho$ is a lower bound of $G_r(u)$. Consequently, if $a_i \leq \rho$ for $a_i\in\vect{v}$, we can remove $a_i$ from $\vect{v}$ and recalculate $\rho$ using~\eqref{equ: rho calculate}, leading to an increasing to $\rho$. By alternating between the calculation of $\rho$~\eqref{equ: rho calculate} and the update of the sequence $\vect{v}$ by discarding its elements smaller or equal to $\rho$, the algorithm will converge to $G_r(u)$. While the exact number of iterations is inherently data-dependent, governed by the distribution of elements in $\vect{v}$ relative to the rising threshold $\rho$, this process is guaranteed to terminate in at most $|\vect{v}| - (k-m)$ steps, where $|\vect{v}|$ denotes the initial cardinality of the sequence.

However, this introduces another critical consideration: \textbf{how can we guarantee the inclusion of all elements in $\vect{b}_T$?}

To address this issue, note that Inequality~\eqref{equ: lower bound} holds for any subsequence $\vect{v} \subseteq \vect{b}$ with $|\vect{v}| > k-m$. Leveraging this insight, we can first initialize a lower bound $\rho$ of $G_r(u)$ using $\vect{v} \subseteq \vect{b}$ satisfying $|\vect{v}| = k - m + 1$. We then progressively add elements from the remaining elements in $\vect{b}$ to $\vect{v}$ if they have the potential to exceed $G_r(u)$, updating $\rho$ accordingly using~\eqref{equ: rho calculate}. Specifically, when processing $b_j$, if $b_j \leq \rho$, it follows that $b_j \leq G_r(u)$, and therefore, this element can be safely excluded from further consideration. Conversely, if $b_j > \rho$, we incorporate it into $\vect{v}$ and update $\rho$ using~\eqref{equ: rho calculate}, leading to an increasing to $\rho$. We then continue to evaluate the next element $b_{j+1}$ until the whole sequence $\vect{b}$ has been processed. Through this process, all elements that are larger than $G_r(u)$ are collected. 

In conclusion, the procedure used to calculate $G_r(u)$ is summarized as Algorithm~\ref{alg: G-Searching}. The correctness of this algorithm is readily established. After each pass, either at least one element has been removed from $\vect{v}$, or $\vect{v}$ and $\rho$ remain unchanged. In the latter scenario, the elements of $\vect{b}$ which are not present in $\vect{v}$ are smaller than $\rho$ and the elements in $\vect{v}$ are larger than $\rho$. Hence, $\sum_{i \in \vect{v}} \max\{a_i-G_r(u),0\} = \sum_{i\in\vect{b}}\max\{a_i-G_r(u),0\}$, leading to the conclusion that $\rho = G_r(u)$. 

\begin{algorithm}
	\caption{G-Searching} 
    \label{alg: G-Searching}
    \DontPrintSemicolon
    \SetKw{Break}{Break}
    \KwData{
        $u$; 
        $m = \bigl|\I(\vect{a}, u, +)\bigr|$; 
        $H = k - m$; 
        $\vect{b} = \{a_i : i \in \I(\vect{a}, -, u)\}$
    }
    \KwResult{$\rho$}
    \lIf{$H=0$}{\Return $\rho = \max\{a_j: a_j < u\}$}
    Initialize $\vect{v} \gets \{b_1, b_2, \dots, b_{N}\}$ with $N>H$~\quad (e.g., $N=H+1$)\;
    $\rho \gets (\sum_{b \in \vect{v}} b - H u)/(|\vect{v}|-H)$\;
    \For{$n = N+1,...,|\vect{b}|$}{
        \lIf{$b_n > \rho$}{
           append $b_n$ to $\vect{v}$ and set $\rho \gets \rho + \left(b_n - \rho\right)/\left(|\vect{v}| - H\right)$
        }
    }
    \While{$|\vect{v}|$ changes}{
        \ForEach{$b \in \vect{v}$}{
            \lIf{$b \leq \rho$}{
                remove $b$ from $\vect{v}$ and set $\rho \gets \rho + (\rho-b)/(|\vect{v}| - H)$
            }
        }
    }
\end{algorithm}

Note that the complexity of Algorithm~\ref{alg: G-Searching} is strongly affected by the intrinsic ordering of the sequence $\vect{a}$. In the most favorable case, where $\vect{a}$ is sorted in decreasing order, the computational complexity reduces to $O(n)$. In contrast, for an adversarially ordered sequence, the complexity may rise to $O(n^{2})$. To alleviate this issue and improve efficiency, it is advantageous to move larger elements toward the front of $\vect{b}$. Since $F_r(u)$ is easy to compute, we first move all elements greater than or equal to $F_r(u)$ to the front of the sequence $\vect{b}$.

If $|\I(\vect{a}, F_r(u), u)| \le k - m$, then $|\I(\vect{a}, F_r(u), +)| \le k$. This corresponds to the case $F_r(u) > G_r(u)$ and $u > u_r^{\star}$. After the reordering, we select $\vect{v}$ to be the first $k - m + 1$ elements of $\vect{b}$, and the resulting value of $\rho$ computed from $\vect{v}$ is guaranteed to be smaller than $F_r(u)$.

If $|\I(\vect{a}, F_r(u), u)| > k - m$, then the elements moved to the front of $\vect{b}$ coincide with $\vect{v}= \{b_i : i \in \I(\vect{b}, F_r(u), u)\}$. In this case, we initialize $\rho$ as
\begin{equation}
    \rho = \frac{\sum_{i \in \I(\vect{b}, F_r(u), u)}b_i - (k-m)u}{|\I(\vect{b}, F_r(u), u)| - (k-m)},
\end{equation}
which gives a valid initial lower bound.

Moreover, if $\rho \ge F_r(u)$, then $G_r(u) \ge F_r(u)$, and we may skip Steps~4--6 in Algorithm~\ref{alg: G-Searching}.

\subsection{Implementation}
\label{subsec: cips implementation}
The procedure, EIPS, is summarized in Algorithm~\ref{alg: EIPS}, operates in three phases to efficiently identify the intersection point $u_r^{\star}$ that characterizes the top-$k$-sum projection solution. The \textbf{Initialization step} first determines whether $r \geq \T_{(k)}(\vect{a})$; if not, it selects appropriate initial points to initiate the search for $u_r^{\star}$. The \textbf{Pivot step} then carries out an adaptive coarse search, iteratively narrowing the interval that contains $u_r^{\star}$. Finally, the \textbf{Exact step} refines this search to determine the precise value of $u_r^{\star}$.

For simplicity, we denote $D(u) = G_r(u) - F_r(u)$. We define $u_R$ and $u_L$ as the tightest bounds for $u_r^{\star}$ maintained during the iterative process, where $u_R$ denotes the smallest known upper bound such that $u_r^{\star} \leq u_R$, and $u_L$ denotes the largest known lower bound such that $u_r^{\star} \geq u_L$. \\
\\
\textbf{Initialization step} (3-5 in Algorithm~\ref{alg: EIPS}) After excluding the special cases $k=1$ and $k=n$, Algorithm~\ref{alg: initialization} either directly determines whether $u^{\star}-l^{\star}\leq \vec{a}_k-\vec{a}_{k+1}$ (\ref{case1} and \ref{case3}) or $u^{\star}-l^{\star}\geq \vec{a}_1-r/k$ such that the optimal solution can be easily computed, or provides two initial points for $u_r^{\star}$, one of which serves as an upper or lower bound. The update scheme in Algorithm~\ref{alg: initialization} is similar to the root-finding procedure in~\cite[Algorithm 1]{gong2011efficient}.

Note that if there exist a $u$ such that $F_r(u)=u$, then the KKT conditions can be immediately classified as belonging to \ref{case1}. It follows that $r\geq \T_{(k)}(\vect{a})$ and $\vect{x}^{\star}=\vect{a}$ is the optimal solution, as illustrated in the right panel of Figure~\ref{fig: u init update}. The algorithm starts with $u^1=F_r(\vec{a}_1)=r/k$ (implicit $u^0=\vec{a}_1$). If $u^1\geq \vec{a}_1$, we have $F_r(u^1)=u^1$. Conversely, if $u^1< \vec{a}_1$, then $(\vec{a}_1,u^1)$ lies below the line $l=u$, leading to $F_r(u^1)< u^1$ due to the properties of $F_r(u)$. Then, we check the value of $F_r'((u^i)^{-})$.
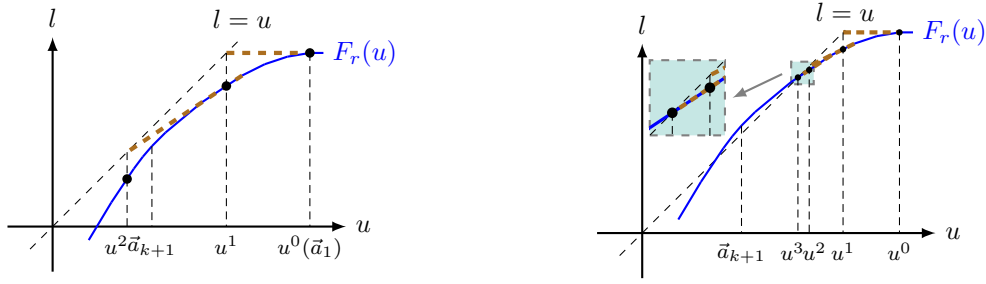
\begin{figure}[!ht]
\centering
\begin{minipage}[c]{0.4\linewidth}
    \centering
    \begin{tikzpicture}[scale=0.6]
        \draw[thick, -latex] (-1,0) -- (6.5,0) node[right] {$u$};
        \draw[thick, -latex] (0,-1) -- (0,4.2) node[above] {$l$};

        \draw[dashed] (-0.5, -0.5) -- (4.2, 4.2) node[above] {$l = u$};
    
        \def\xvalues{{0.8, 1.3, 1.8, 2, 2.2, 2.5, 3.3, 4.2, 4.7, 5.2, 5.7, 6.0}}
        
        \xdef\ystart{-0.3}
        \foreach \i in {0,1,...,10} {
            \pgfmathsetmacro{\value}{\ystart + (11-(\i+1))/6*(\xvalues[\i+1] - \xvalues[\i])}
            \draw[thick, blue] (\xvalues[\i], \ystart) -- (\xvalues[\i+1], \value);
            \ifnum \i=4
                \draw[densely dashed] (\xvalues[\i],\ystart) -- (\xvalues[\i],0) node[below] {\footnotesize $\Vec{a}_{k+1}$}; 
            \fi
            \xdef\ystart{\value}
            \ifnum \i=10
                \node[blue, right] at (\xvalues[\i+1], \ystart) {$F_r(u)$};
            \fi
        }
        \draw[ultra thick, dashed, MyYellow] (3.85, 3.85) -- (5.7, 3.85); 
        \draw[ultra thick, dashed, MyYellow] (4.2, 3.35) -- (1.65, 1.65); 

        \draw[densely dashed] (5.7, 3.85) -- (5.7, 0) node[below] {\footnotesize $u^0$($\vec{a}_1$)}; 
        \draw[densely dashed] (3.85, 3.85) -- (3.85, 0) node[below] {\footnotesize $u^1$}; 
        \draw[densely dashed] (1.65, 1.65) -- (1.65, 0) node[below, xshift=-4pt] {\footnotesize $u^2$}; 
        \fill (5.7, 3.85) circle (3pt);
        \fill (3.85, 3.117) circle (3pt);
        \fill (1.65, 1.055) circle (3pt);

    \end{tikzpicture}
\end{minipage}
\hspace{1cm}
\begin{minipage}[c]{0.4\linewidth}
    \centering
    \begin{tikzpicture}[scale=0.6]
    
      \draw[thick, -latex] (-1,0) -- (6.5,0) node[right] {$u$};
      \draw[thick, -latex] (0,-1) -- (0,4.2) node[above] {$l$};
      
      \draw[gray, dashed, thick, fill=bg!50] (3.30,3.30) rectangle (3.8,3.8);
      \draw[dashed] (-0.5,-0.5) -- (4.5,4.5) node[above] {$l = u$};
      
      \def\xvalues{{0.8,1.3,1.8,2,2.2,2.5,3.3,4.2,4.7,5.2,5.7,6.0}}
          
      \xdef\ystart{0.3}
      \foreach \i in {0,1,...,10} {
          \pgfmathsetmacro{\value}{\ystart + (11-(\i+1))/6*(\xvalues[\i+1]-\xvalues[\i])}
          \draw[thick, blue] (\xvalues[\i], \ystart) -- (\xvalues[\i+1], \value);
          \ifnum \i=4
              \draw[densely dashed] (\xvalues[\i],\ystart) -- (\xvalues[\i],0) node[below] {\footnotesize $\Vec{a}_{k+1}$}; 
          \fi
          \xdef\ystart{\value}
          \ifnum \i=10
              \node[blue, right] at (\xvalues[\i+1], \ystart) {$F_r(u)$};
          \fi  
      }
    
      \draw[ultra thick, densely dashed, MyYellow] (4.45, 4.45) -- (5.7, 4.45); 
      \draw[ultra thick, densely dashed, MyYellow] (4.7,4.2) -- (3.7,3.7); 
      \draw[ultra thick, densely dashed, MyYellow] (4.2, 3.95) -- (3.45, 3.45);
    
      \draw[densely dashed] (5.7, 4.45) -- (5.7, 0) node[below] {\footnotesize $u^0$}; 
      \draw[densely dashed] (4.45, 4.45) -- (4.45, 0) node[below] {\footnotesize $u^1$}; 
      \draw[densely dashed] (3.7, 3.7) -- (3.7, 0) node[below, xshift=2pt] {\footnotesize $u^2$}; 
      \draw[densely dashed] (3.45, 3.45) -- (3.45, 0) node[below, xshift=-2pt] {\footnotesize $u^3$};
    
      \fill (5.7, 4.45)   circle (2pt);
      \fill (4.45, 4.075) circle (2pt);
      \fill (3.7, 3.617) circle (2pt);
      \fill (3.45, 3.45)  circle (2pt);
    
    
      \coordinate (rectLL) at (3.30,3.30);  
      \coordinate (rectUR) at (3.75,3.75);  
    
      \coordinate (mainLeftMid) at (3.30-0.2, {3.30 + (3.75-3.30)/2});
    
      \node[draw=gray, dashed, thick, inner sep=0pt, fill=bg!50] (inset) at (1,3) {
        \begin{tikzpicture}[xscale=2.0, yscale=2]
          \clip (3.30,3.30) rectangle (3.8,3.8);
          
          \draw[thin, -latex] (-1,0) -- (6.5,0) node[right] {$u$};
          \draw[thin, -latex] (0,-1) -- (0,4.2) node[above] {$l$};
    
          \draw[dashed] (-0.5,-0.5) -- (4.5,4.5);
    
          \def\xvalues{{0.8,1.3,1.8,2,2.2,2.5,3.3,4.2,4.7,5.2,5.7,6.0}}
          \xdef\ystart{0.3}
          \foreach \i in {0,1,...,10} {
              \pgfmathsetmacro{\value}{\ystart + (11-(\i+1))/6*(\xvalues[\i+1]-\xvalues[\i])}
              \draw[solid, very thick, blue] (\xvalues[\i],\ystart) -- (\xvalues[\i+1],\value);
              \xdef\ystart{\value}
          }
    
          \draw[ultra thick, densely dashed, MyYellow] (4.7,4.2) -- (3.7,3.7); 
          \draw[ultra thick, densely dashed, MyYellow] (4.2, 3.95) -- (3.45, 3.45);

          \draw[densely dashed] (3.7, 3.7) -- (3.7, 0);
          \draw[densely dashed] (3.45,3.45) -- (3.45,0);    
          \node[anchor=center, inner sep=0pt, minimum size=4pt, circle, fill=black, transform shape=false]
                at (3.7, 3.617) {};
          \node[anchor=center, inner sep=0pt, minimum size=4pt, circle, fill=black, transform shape=false]
                at (3.45,3.45) {};
    
    
        \end{tikzpicture}
      };
    
      \draw[thick, ->,>=latex, gray] (mainLeftMid) -- (2,3);
    
    \end{tikzpicture}
\end{minipage}
\caption{Illustration of the generation of $u^{i}$ (black point) \textbf{(Left)}: the generating $u^3 \geq \Vec{a}_{k+1}$, so we break the for loop and $r < \T_{(k)}(\vect{a})$. \textbf{(Right)}: $F_r(u^3) = u^3$, so $r \geq \T_{(k)}(\vect{a})$. }
\label{fig: u init update}
\end{figure}
\begin{itemize}
    \item If $F_r'((u^i)^{-})<1$, i.e., $u^i > \vec{a}_k$, we compute the intersection of the linear approximation $F_r(u^i)+F_r'((u^i)^{-})(u-u^i)$ with the line $l=u$, which is 
    \[u^{i+1}=\frac{F_r(u^i)-F_r'((u^i)^{-})u^i}{1-F_r'((u^i)^{-})}=\frac{kF_r(u^i)-u^i |\I(\vect{a},u^i,+)| }{k-|\I(\vect{a},u^i,+)| }\leq u^i, \]
    as shown in Figure~\ref{fig: u init update}. It also yields $F_r(u^{i+1})\leq u^{i+1}$ and $F_r'((u^{i+1})^{-})\geq F_r'((u^i)^{-})$ because $F_r(u)$ is below the linear approximation. Moreover, if $F_r'((u^{i+1})^{-})= F_r'((u^i)^{-})$, we obtain that $F_r(u^{i+1})=u^{i+1}$. Otherwise, $F_r'((u^{i+1})^{-})\geq  F_r'((u^i)^{-})+\frac{1}{k}$, which follows from the concavity of $F_r(u)$ in Lemma~\ref{thm: fr property}. This iterative procedure continues until either $F_r(u^{i+1}) = u^{i+1}$ or $F_r'((u^{i+1})^{-}) \geq 1$ is satisfied.
    \item If $F_r'((u^i)^{-})\geq 1$, it implies that $u^i \leq \vec{a}_k$ and $r \leq \T_{(k)}(\vect{a})$. In this case, $u^i \leq u^{\star}$ because $u^{\star}\geq \vec{a}_k$, as shown in the left panel of Figure~\ref{fig: u init update}. Particularly, if $F_r'((u^i)^{-})= 1$, we have that $\vec{a}_k=\min \{a_i: a_i \geq u^i\}$. We subsequently verify whether this scenario corresponds to \ref{case3}.    
\end{itemize}



Following the previous steps, once we obtain $F_r'((u^{i})^{-})\geq 1$, we can conclude that $r < \T_{(k)}(\vect{a})$. Subsequently, our algorithm proceeds to identify the exact intersection point $u_r^{\star}$. To efficiently locate this intersection, we select two suitable initial points that are close to $u_r^{\star}$. To achieve this, we leverage the previously computed iterates $u^i \leq \Vec{a}_k <u^{i-1}$. The detailed selection procedure for these initial points is outlined as follows. 

\begin{algorithm}[!ht]
    \caption{Initialization ($2\leq k\leq n-1$)}
    \label{alg: initialization}
    \KwData{$\vect{a},~k,~r,~i \gets 1,~u^1 \gets r/k,~u_R \gets \infty,~u_L \gets -\infty,~u^{\star} \gets\infty,~l^{\star} \gets -\infty$} 
    \KwResult{$\textsc{Flag}$, $u^{\star}, l^{\star}, u_R, u_L, u_C$}
    \DontPrintSemicolon
    \SetKw{Break}{break}
    \SetKw{True}{True}
    \While{\True}{  
        \lIf{$F_r(u^i) = u^i$}{$\textsc{Flag}\gets0$; \Return \tcp*[f]{Case 1; $\textsc{Flag}=0$}}
        \If{$F_r'((u^i)^{-})= 1$}{$\Vec{a}_k \gets \min\{a_i: a_i \geq u^i\}$\; \If{$g_r(\Vec{a}_k, F_r(\Vec{a}_k)) = 0$}{$\textsc{Flag}\gets-1$, $u^{\star} \gets \Vec{a}_k, l^{\star} \gets F_r(\vec{a}_k)$; \Return \hfill \tcp{Case 3; $\textsc{Flag}=-1$}}}
        \lIf{$F_r'((u^i)^{-}) \geq 1$}{\Break}
        $u^{i+1}\gets\frac{u^i \left|\I(\vect{a},u^i,+)\right| - kF_r(u^i)}{\left|\I(\vect{a},u^i,+)\right| - k}$\hfill\tcp{$u^{i+1}<u^{i}$}
        $i \gets i+1$
    }
    \eIf{$i=1$}{
        \eIf{$g_r(\Vec{a}_1, F_r(\Vec{a}_1))<0$}{$\textsc{Flag}\gets-1$ and $l^{\star}\gets r/k$ \hfill\tcp{$\textsc{Flag}=-1$}}
        {
            $\textsc{Flag}\gets 1$, $u_R$ and $u_C$ using~\eqref{equ: initial points generating case 1}\hfill\tcp{$\textsc{Flag}=1$} 
        }
    }{
        \eIf{$g_r(u^{i-1}, F_r(u^{i-1}))<0$}{$\textsc{Flag}\gets 2$, $u_L$ and $u_C$ using~\eqref{equ: initial points generating case 2} \hfill\tcp{$\textsc{Flag}=2$} }
        {
            $\textsc{Flag}\gets 3$, $u_R$ and $u_C$ using~\eqref{equ: initial points generating case 3} \hfill\tcp{$\textsc{Flag}=3$}
        }
    }
    
\end{algorithm}

Initially, by examining the sign of $g_r(u^{i-1}, F_r(u^{i-1}))$, we determine the relationship between $u^{i-1}$ and $u_r^{\star}$. Specifically, if $g_r(u^{i-1}, F_r(u^{i-1}))>0$, then $F_r(u^{i-1})>G_r(u^{i-1})$, implying $u^{i-1}>u_r^{\star}$. Conversely, if $g_r(u^{i-1}, F_r(u^{i-1}))<0$, then $u^{i-1}<u_r^{\star}$. Guided by these observations and the iteration index $i$, we adaptively select two refined initial points for the subsequent computational stage, tailored to each of the following scenarios:

\begin{itemize}
    \item \label{item: initial case -1}$\textsc{Flag}=-1$ ($i=1$ and $g_r(\Vec{a}_1, F_r(\Vec{a}_1))<0$). We have $\Vec{a}_1 < u_r^{\star}$ and $l^{\star} = F_r(\Vec{a}_1) = r/k$. According to the solution given in~\eqref{kkt: solution form}, the optimal solution is directly obtainable as $\vect{x}^{\star} = \vect{a}$, except for $x^{\star}_{\I_1} = \frac{r}{k}\1_{|\I_1|}$, where $\I_1 := \I(\vect{a},r/k, +)$. 
    
    \item \label{item: initial case 1}$\textsc{Flag}=1$ ($i=1$ and $g_r(\Vec{a}_1, F_r(\Vec{a}_1))>0$). We have $\Vec{a}_1 > u_r^{\star}$. Due to the limited information available (only one iteration thus far), we select our points as:
    \begin{equation}
    \label{equ: initial points generating case 1}
    u_R = \Vec{a}_1, \quad u_C = \frac{D(u_R) }{k/\left|\I(\vect{a}, G_r(u_R), u_R)\right|} + u_R,
    \end{equation}
    Here, $u_C(<u_R)$ is horizontal coordinate of the intersection between the line \( l = r/k \) and the tangent line to \( G_r(u) \) at the point \( (u_R, G_r(u_R)) \), as illustrated in the left panel of Figure~\ref{fig: initial points generating}.
    \item \label{item: initial case 2}$\textsc{Flag}=2$ ($i>1$ and $g_r(u^{i-1}, F_r(u^{i-1}))<0$). We have $u^{i-1}<u_r^{\star}$ and select the following initial points:
    \begin{equation}
    \label{equ: initial points generating case 2}
    u_L = u^{i-1}, \quad u_C = (u_L+u')/2, \quad u'=\frac{kD(u^{i-1})}{\left|\I(\vect{a}, u^{i-1},+)\right|} + u^{i-1},
    \end{equation}
    i.e., $u'(>u_L)$ is horizontal coordinate of the intersection between the line \( l = G_r(u^{i-1}) \) and the tangent line to \( F_r(u) \) at the point \( (u_L, F_r(u_L)) \), as shown in the middle panel of Figure~\ref{fig: initial points generating}.

    \item \label{item: initial case 3}$\textsc{Flag}=3$ ($i>1$ and $g_r(u^{i-1}, F_r(u^{i-1}))>0$): We have $u^{i-1}>u_r^{\star}$. We accordingly define the points:
    \begin{equation}
    \label{equ: initial points generating case 3}
    u_R = u^{i-1}, \quad u_C = \left(u_R+ \max\{u'', u^i\}\right)/2, \quad u''= \frac{kD(u^{i-1})}{\left|\I(\vect{a}, u^{i-1},+)\right|} + u^{i-1},
    \end{equation}
    i.e., $u''(<u_R)$ is the horizontal coordinate of the intersection between the line \( l = G_r(u^{i-1}) \) and the tangent line to \( F_r(u) \) at the point \( (u_R, F_r(u_R)) \), as depicted in the right panel of Figure~\ref{fig: initial points generating}.
\end{itemize}

\begin{figure}[!ht]
\centering
\begin{minipage}[c]{0.32\linewidth}
    \centering
    \begin{tikzpicture}[scale=0.5]
        \draw[thick, -latex] (-0.5,0) -- (6,0) node[right] {$u$};
        \draw[thick, -latex] (0,-1) -- (0,3.5) node[above] {$l$};

        \draw[dashed] (-0.5, -0.5) -- (3.5, 3.5);
    
        \def\xvalues{{0.8, 1.3, 1.8, 2, 2.2, 2.5, 3.3, 4.2, 4.7, 5.2, 5.7, 6.0}}
        
        \xdef\ystart{-2.2}
        \foreach \i in {0,1,...,10} {
            \pgfmathsetmacro{\value}{\ystart + (11-(\i+1))/6*(\xvalues[\i+1] - \xvalues[\i])}
            \draw[very thick, blue] (\xvalues[\i], \ystart) -- (\xvalues[\i+1], \value);
            \xdef\ystart{\value}
            \ifnum \i=10
                \node[blue, right] at (\xvalues[\i+1], \ystart) {$F_r(u)$};
            \fi
    
        }
        \draw[very thick, red] (2.2, 2.2) -- (2.5,2.2) -- (3,2) -- (3.3, 1.9) -- (4, 1.2) -- (4.2, 1.1) -- (4.7, 0.5) -- (5.2, -0.4) -- (5.7, -0.9) -- (6.0, -1.15) node[below] {$G_r(u)$};

        \draw[ultra thick, dashed, MyYellow] (5.7, 1.95) -- (1.95, 1.95); 
        \draw[densely dashed] (5.7, -0.9) -- (2.28, 1.95); 
        \draw[densely dashed] (5.7, 1.95) -- (5.7, -0.9) node[above, yshift=1pt] {\footnotesize $u_R$}; 
        \draw[densely dashed] (2.28, 1.95) -- (2.28, 0) node[below, xshift=3pt] {\footnotesize $u_C$}; 
        \fill[orange] (5.7, 1.95) circle (3pt);
        \fill[orange] (2.28, 1.95) circle (3pt);
        \fill (5.7, -0.9) circle (3pt);
        
    \end{tikzpicture}
\end{minipage}
\hfill
\begin{minipage}[c]{0.32\linewidth}
    \centering
    \begin{tikzpicture}[scale=0.5]
        \draw[thick, -latex] (-0.5,0) -- (6,0) node[right] {$u$};
        \draw[thick, -latex] (0,-1) -- (0,3.7) node[above] {$l$};

        \draw[dashed] (-0.5, -0.5) -- (3.7, 3.7);
    
        \def\xvalues{{0.8, 1.3, 1.8, 2, 2.2, 2.5, 3.3, 4.2, 4.7, 5.2, 5.7, 6.0}}
        
        \xdef\ystart{-1.3}
        \foreach \i in {0,1,...,10} {
            \pgfmathsetmacro{\value}{\ystart + (11-(\i+1))/6*(\xvalues[\i+1] - \xvalues[\i])}
            \draw[very thick, blue] (\xvalues[\i], \ystart) -- (\xvalues[\i+1], \value);
            \xdef\ystart{\value}
            \ifnum \i=10
                \node[blue, right] at (\xvalues[\i+1], \ystart) {$F_r(u)$};
            \fi
    
        }
        \draw[very thick, red] (2.2, 2.2) -- (2.5,2.2) -- (3,2) -- (3.3, 1.9) -- (4, 1.2) -- (4.2, 1.1) -- (4.7, 0.5) -- (5.2, -0.4) -- (5.7, -0.9) node[below] {$G_r(u)$};

        \draw[ultra thick, dashed, MyYellow] (5.7, 2.85) -- (2.85, 2.85); 
        \draw[densely dashed] (2.85, 2.85) -- (2.85, 0) node[below, xshift=-5pt] {\footnotesize $u_L$};
        \draw[ultra thick, dashed, MyYellow] (4.85, 3.06) -- (0.85, -0.23);
        \draw[densely dashed] (3.65, 2.05) -- (3.65, 0) node[below,xshift=5pt] {\footnotesize $u'$};
        \draw[densely dashed] (3.25, 2.05) -- (3.25, 0) node[below] {\footnotesize $u_C$};  
        \draw[densely dashed] (2.5, 2.05) -- (4.1, 2.05);
        \fill (2.85, 2.05) circle (3pt);
        \fill (3.65, 2.05) circle (3pt);
        \fill[orange] (3.25, 2.05) circle (3pt);
        \fill[orange] (2.85, 1.4167) circle (3pt);   


    \end{tikzpicture}
\end{minipage}
\hfill
\begin{minipage}[c]{0.32\linewidth}
    \centering
    \begin{tikzpicture}[scale=0.5]
        \draw[thick, -latex] (-0.5,0) -- (6,0) node[right] {$u$};
        \draw[thick, -latex] (0,-1) -- (0,4.2) node[above] {$l$};

        \draw[dashed] (-0.6, -0.6) -- (4.2, 4.2);
    
        \def\xvalues{{0.8, 1.3, 1.8, 2, 2.2, 2.5, 3.3, 4.2, 4.7, 5.2, 5.7, 6.0}}
        
        \xdef\ystart{-0.6}
        \foreach \i in {0,1,...,10} {
            \pgfmathsetmacro{\value}{\ystart + (11-(\i+1))/6*(\xvalues[\i+1] - \xvalues[\i])}
            \draw[very thick, blue] (\xvalues[\i], \ystart) -- (\xvalues[\i+1], \value);
            \xdef\ystart{\value}
            \ifnum \i=10
                \node[blue, right] at (\xvalues[\i+1], \ystart) {$F_r(u)$};
            \fi
    
        }
        \draw[very thick, red] (2.2, 2.2) -- (2.5,2.2) -- (3,2) -- (3.3, 1.9) -- (4, 1.2) -- (4.2, 1.1) -- (4.7, 0.5) -- (5.2, -0.4) -- (5.7, -0.9) node[below] {$G_r(u)$};

        \draw[ultra thick, dashed, MyYellow] (5.7, 3.55) -- (3.55, 3.55); 
        \draw[ultra thick, dashed, MyYellow] (4.55, 3.283) -- (0.75, 0.75); 
        \draw[densely dashed] (3.55, 3.55) -- (3.55, 0) node[below, xshift=2pt] {\footnotesize $u_R$};
        \draw[densely dashed] (4, 1.65) -- (2, 1.65);
        \draw[densely dashed] (2.825, 2.13) -- (2.825, 0) node[below] {\footnotesize $u_C$};
        \draw[densely dashed] (2.1, 1.65) -- (2.1, 0) node[below] {\footnotesize $u''$};
        \fill[orange] (3.55, 2.6167) circle (3pt);
        \fill (3.55, 1.65) circle (3pt);
        \fill (2.1, 1.65) circle (3pt);
        \fill[orange] (2.825, 2.13) circle (3pt);
        
        
                
    \end{tikzpicture}
\end{minipage}
\caption{Illustration of the generation of the initial points (orange point). \textbf{(Left): }$\textsc{Flag}=1$. \textbf{(Middle): }$\textsc{Flag}=2$. \textbf{(Right): }$\textsc{Flag}=3$.}
\label{fig: initial points generating}
\end{figure}
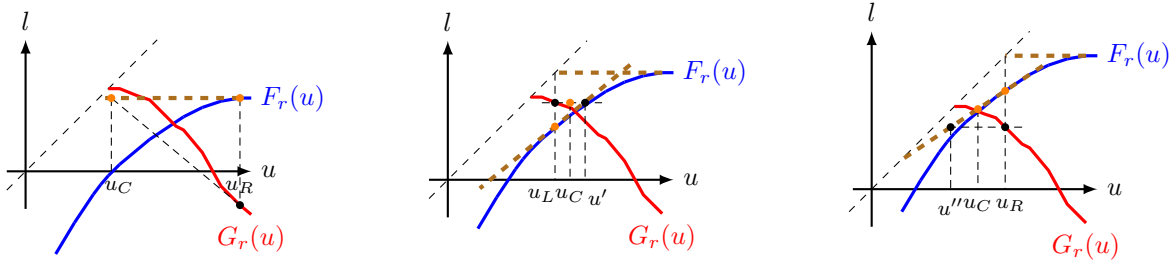


In summary, the classification described above is based upon the relative position between $u^{i-1}$ and $u_r^{\star}$. Our aim in this categorization is to select initial points that are as close as possible, thereby improving convergence efficiency. Additionally, it is noteworthy that the case $\textsc{Flag}=1$ and the case $\textsc{Flag}=3$ indicate the same scenario $u^{i-1} > u_r^{\star}$. However, in the case $\textsc{Flag}=1$, due to the limited amount of information available from the initial iteration stage, it is impossible to construct $u''$ in the same manner as done in the scenario $\textsc{Flag}=3$. Consequently, we explicitly distinguish it from $\textsc{Flag}=3$ to adequately handle this special case.\\
\\

\begin{algorithm}[!ht]
    \caption{Efficient Intersection Point Searching (EIPS)}
    \label{alg: EIPS}
    \DontPrintSemicolon
    \SetKwFunction{GSearching}{G-Searching}
    \SetKwFunction{Initialization}{Initialization}
    \SetKw{Continue}{Continue}
    \SetKw{break}{break}
    \SetKw{return}{return}
    \SetKw{true}{true}
    \SetKw{llif}{if}
    \SetKw{llelse}{else}
    \SetKw{then}{then}
    \SetKw{repeat}{repeat}
    \SetKw{until}{until}
    \KwData{$\vect{a}, k, r, \epsilon$.}
    \KwResult{$\vect{x}^{\star}$.}
    \lIf{$k=1$}{\Return $\vect{x}^{\star} = \min\{r\1_n, \vect{a}\}$}
    \lIf{$k=n$}{\return $\vect{x}^{\star} = \vect{a} +  \left(r - \max\{r, \1_n^\top \vect{a}\}\right)\1_n/n$}
    Obtain $\textsc{Flag}$, $u^{\star}, l^{\star}, u_R, u_L, u_C$ from Algorithm~\ref{alg: initialization} ($\texttt{Initialization}$)\;
    \lIf{$\textsc{Flag}=0$}{\Return $\vect{x}^{\star} = \vect{a}$}
    \lIf{$\textsc{Flag}=-1$}{go to line 28}
    \While{\true}{
        \lWhile{$F_r'((u_C)^{-}) > 1$}{$u_C \gets (u_C + u_R)/2$ }
        \lIf{$\left|D(u_C)\right| \leq \epsilon$}{$u_L \gets u_C$ and \break}
        \uIf{$F_r'((u_C)^{-}) = F_r'((u_L)^{-})$}{
                \lIf{$D(u_C) > 0$}{$u_L \gets u_C$, $u_C \gets \min\{a_i: a_i > u_C\}$}
                \lIf{$D(u_C) \le 0$}{\break}}
        \ElseIf{$F_r'((u_C)^{-}) = F_r'((u_R)^{-})$}{
                \lIf{$D(u_C) < 0$}{$u_R \gets u_C$, $u_C \gets \max\{a_i: a_i < u_C\}$}
                \lIf{$D(u_C) \geq 0$}{$u_L \gets u_C$ and \break} }
        \uIf{$u_R = +\infty$}{
                \llif $D(u_C) < 0$, \then $u_R \gets u_C$; \llelse $u_{LB} \gets u_L, u_L \gets u_C$
        }
        \uElseIf{$u_L = -\infty$}{
            \llif $D(u_C) > 0$, \then $u_L \gets u_C$; \llelse $u_{RB} \gets u_R, u_R \gets u_C$
        }
        \Else{
            {$u_R \gets u_C \ \text{if}\ D(u_C) < 0;\ u_L \gets u_C \ \text{if}\ D(u_C) > 0$}
        }
        Update $u_C$ by~\eqref{equ: u piv update}\;
    }
    \lWhile{$\left|D(u_L)\right| > \epsilon$}{update $u_L$ by~\eqref{equ: exact step update}} 
    \lIf{$m=k$}{
        $u^{\star} \gets \min\{a_i: a_i \geq u_L\}$, $l^{\star} \gets F_r(u^{\star})$
    }
    \lElse{
        $u^{\star} \gets u_L$, $l^{\star} \gets F_r(u^{\star})$
    }
    Construct $\alpha$, $\beta$, $\gamma$ through~\eqref{equ: related set} and \return $\vect{x}^{\star}$ through~\eqref{kkt: solution form}\;
\end{algorithm}

\textbf{Pivot step} (6-24 in Algorithm~\ref{alg: EIPS}) Given the two initial points obtained in the Initialization step, we now begin searching for a narrow interval that contains $u_r^{\star}$. The main idea is to identify both an upper and a lower bound for $u_r^{\star}$ and iteratively shrink the interval until we obtain $\vec{a}_{j+1}\leq u_L\leq u_r^{\star}\leq \vec{a}_j$.

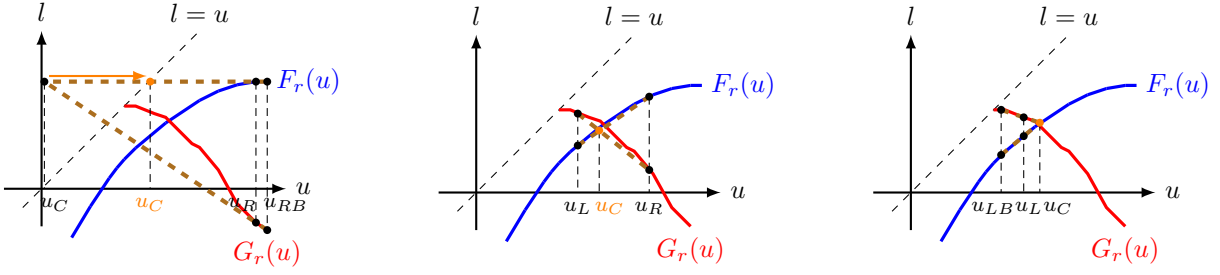
\begin{figure}[!ht]
\centering
\begin{minipage}[c]{0.3\linewidth}
    \centering
    \begin{tikzpicture}[scale=0.5]
        \draw[thick, -latex] (-1,0) -- (6.5,0) node[right] {$u$};
        \draw[thick, -latex] (0,-1) -- (0,4.2) node[above] {$l$};

        \draw[dashed] (-0.5, -0.5) -- (4.2, 4.2) node[above] {$l = u$};
    
        \def\xvalues{{0.8, 1.3, 1.8, 2, 2.2, 2.5, 3.3, 4.2, 4.7, 5.2, 5.7, 6.0}}
        
        \xdef\ystart{-1.3}
        \foreach \i in {0,1,...,10} {
            \pgfmathsetmacro{\value}{\ystart + (11-(\i+1))/6*(\xvalues[\i+1] - \xvalues[\i])}
            \draw[very thick, blue] (\xvalues[\i], \ystart) -- (\xvalues[\i+1], \value);
            \xdef\ystart{\value}
            \ifnum \i=10
                \node[blue, right] at (\xvalues[\i+1], \ystart) {$F_r(u)$};
            \fi
    
        }
        \draw[very thick, red] (2.2, 2.2) -- (2.5,2.2) -- (3,2) -- (3.3, 1.9) -- (4, 1.2) -- (4.2, 1.1) -- (4.7, 0.5) -- (5.2, -0.4) -- (5.7, -0.9) -- (6.0, -1.1) node[below] {$G_r(u)$};

        \draw[ultra thick, dashed, MyYellow] (6.0, 2.85) -- (0.075,2.85); 
        \draw[ultra thick, dashed, MyYellow] (6, -1.1) -- (0.075,2.85); 
        \draw[densely dashed] (6.0, 2.85) -- (6, 0) node[below, xshift=7pt] {\footnotesize $u_{RB}$};
        \draw[densely dashed] (6.0, 0) -- (6, -1.1);
        \fill (6.0, 2.85) circle (3pt);
        \fill (6, -1.1) circle (3pt);
        \draw[densely dashed] (5.7, 2.85) -- (5.7, 0) node[below,xshift=-5pt] {\footnotesize $u_R$}; 
        \draw[densely dashed] (5.7, 0) -- (5.7, -0.9); 
        \fill (5.7, 2.85) circle (3pt);
        \fill (5.7, -0.9) circle (3pt);   
        \draw[densely dashed] (0.075,2.85) -- (0.075,0) node[below, xshift=4pt] {\footnotesize $u_C$}; 
        \draw[densely dashed] (2.8875,2.85) -- (2.8875,0) node[below, orange] {\footnotesize $u_C$};         
        \fill (0.075,2.85) circle (3pt);
        \fill[orange] (2.8875,2.85) circle (3pt);
        \draw[thick, -latex, orange] (0.2,3) -- (2.8,3);
    \end{tikzpicture}
\end{minipage}
\hfill
\begin{minipage}[c]{0.3\linewidth}
    \centering
    \begin{tikzpicture}[scale=0.5]
        \draw[thick, -latex] (-1,0) -- (6.5,0) node[right] {$u$};
        \draw[thick, -latex] (0,-1) -- (0,4.2) node[above] {$l$};

        \draw[dashed] (-0.5, -0.5) -- (4.2, 4.2) node[above] {$l = u$};
    
        \def\xvalues{{0.8, 1.3, 1.8, 2, 2.2, 2.5, 3.3, 4.2, 4.7, 5.2, 5.7, 6.0}}
        
        \xdef\ystart{-1.3}
        \foreach \i in {0,1,...,10} {
            \pgfmathsetmacro{\value}{\ystart + (11-(\i+1))/6*(\xvalues[\i+1] - \xvalues[\i])}
            \draw[very thick, blue] (\xvalues[\i], \ystart) -- (\xvalues[\i+1], \value);
            \xdef\ystart{\value}
            \ifnum \i=10
                \node[blue, right] at (\xvalues[\i+1], \ystart) {$F_r(u)$};
            \fi
    
        }
        \draw[very thick, red] (2.2, 2.2) -- (2.5,2.2) -- (3,2) -- (3.3, 1.9) -- (4, 1.2) -- (4.2, 1.1) -- (4.7, 0.5) -- (5.2, -0.4) -- (5.7, -0.9) node[below] {$G_r(u)$};

        \draw[ultra thick, dashed, MyYellow] (2.7, 2.1) -- (4.6, 0.6); 
        \draw[ultra thick, dashed, MyYellow] (2.7, 1.25) -- (4.6, 2.55); 
        \draw[densely dashed] (2.7, 2.1) -- (2.7, 0) node[below] {\footnotesize $u_L$};
        \fill (2.7, 2.1) circle (3pt);
        \fill (2.7, 1.25) circle (3pt);
        \draw[densely dashed] (4.6, 2.55) -- (4.6, 0) node[below] {\footnotesize $u_R$}; 
        \fill (4.6, 2.55) circle (3pt);
        \fill (4.6, 0.6) circle (3pt);   
        \draw[densely dashed] (3.27,1.651285) -- (3.27,0) node[below, xshift=4pt, orange] {\footnotesize $u_C$}; 
        \fill[orange] (3.27,1.651285) circle (3pt);
    \end{tikzpicture}
\end{minipage}
\hfill
\begin{minipage}[c]{0.3\linewidth}
    \centering
    \begin{tikzpicture}[scale=0.5]
        \draw[thick, -latex] (-1,0) -- (6.5,0) node[right] {$u$};
        \draw[thick, -latex] (0,-1) -- (0,4.2) node[above] {$l$};

        \draw[dashed] (-0.5, -0.5) -- (4.2, 4.2) node[above] {$l = u$};
    
        \def\xvalues{{0.8, 1.3, 1.8, 2, 2.2, 2.5, 3.3, 4.2, 4.7, 5.2, 5.7, 6.0}}
        
        \xdef\ystart{-1.3}
        \foreach \i in {0,1,...,10} {
            \pgfmathsetmacro{\value}{\ystart + (11-(\i+1))/6*(\xvalues[\i+1] - \xvalues[\i])}
            \draw[very thick, blue] (\xvalues[\i], \ystart) -- (\xvalues[\i+1], \value);
            \xdef\ystart{\value}
            \ifnum \i=10
                \node[blue, right] at (\xvalues[\i+1], \ystart) {$F_r(u)$};
            \fi
    
        }
        \draw[very thick, red] (2.2, 2.2) -- (2.5,2.2) -- (3,2) -- (3.3, 1.9) -- (4, 1.2) -- (4.2, 1.1) -- (4.7, 0.5) -- (5.2, -0.4) -- (5.7, -0.9) node[below] {$G_r(u)$};

        \draw[ultra thick, dashed, MyYellow] (2.4, 2.2) -- (24/7, 13/7); 
        \draw[ultra thick, dashed, MyYellow] (2.4, 1) -- (24/7, 13/7); 
        \draw[densely dashed] (3, 2) -- (3, 0) node[below, xshift=2pt] {\footnotesize $u_L$};
        \fill (3, 2) circle (3pt);
        \fill (3, 1.5) circle (3pt);
        \draw[densely dashed] (24/7, 13/7) -- (24/7, 0) node[below, xshift=7pt] {\footnotesize $u_C$}; 
        \fill (2.4, 2.2) circle (3pt);
        \fill (2.4, 1) circle (3pt);   
        \draw[densely dashed] (2.4, 2.2) -- (2.4, 0) node[below, xshift=-3pt] {\footnotesize $u_{LB}$}; 
        \fill[orange] (24/7, 13/7) circle (3pt);
    \end{tikzpicture}
\end{minipage}
\caption{Illustration of the update of $u_C$ (orange point). \textbf{(Left)}: $u_L = -\infty$. the generating (black point) $u_C \leq \Vec{a}_{k+1} (m > k)$, we then generate (orange point) $u_C$ by setting $u_C \gets (u_C + u_R)/2$, now the new generating $u_C$ is valid. \textbf{(Middle)}: $u_L > -\infty, u_R < \infty$. \textbf{(Right)}: $u_R = \infty$.}
\label{fig: ui+1 update}
\end{figure}

Suppose we are currently processing $u_C$. Before evaluating $G_r(u)$ via Algorithm~\ref{alg: G-Searching}, we must verify that $F_r'((u_C)^{-}) \leq 1$, as $G_r(u)$ is undefined for $u \leq \vec{a}_{k+1}$. If $F_r'((u_C)^{-}) > 1$ (implying $u_C < \vec{a}_{k+1}$), we iteratively update $u_C$ by averaging it with $u_R$ until the condition $F_r'((u_C)^{-}) \leq 1$ is satisfied (see Figure~\ref{fig: ui+1 update}, left panel). Once valid, if the convergence criterion $|D(u_C)| \leq \epsilon$ is met, we designate $u_r^{\star} = u_C$ and terminate the procedure. Otherwise, we update the search interval for $u_r^{\star}$ as follows:
\begin{itemize}
    \item \textbf{Lines 9–15}: Check whether the vector $\vec{a}$ contains any elements in the interval $(u_L, u_C)$ (or $(u_C, u_R)$ for lines 12-14). If no such elements exist, then $F_r(u)$ is linear over $(u_L, u_C)$. If $D(u_C) > 0$, we obtain $u_r^{\star} \in (u_L, u_C)$; otherwise, we shift the interval $(u_L, u_C)$ along the same linear segment or move to the next linear segment. We then recheck whether $u_r^{\star} \in (u_L, u_C)$ in the updated region.
    \item \textbf{Lines 16-19}: This stage handles the scenario where only one finite bound exists. For instance, if $u_R=+\infty$, we rely solely on the lower bound $u_L\leq u_r^{\star}$, with the current pivot satisfying $u_C>u_L$. We first determine whether $u_C>u_r^{\star}$. 
    If so, we treat the current $u_C$ as the new upper bound by setting $u_R\gets u_C$, and subsequently update $u_C$ to a point within the new interval $(u_L,u_R)$. 
    Conversely, if $u_C\leq u_r^{\star}$, we update the lower bound by setting $u_L\gets u_C$, and increase $u_C$ according to line 23. This iterative process continues until both bounds are identified or the convergence condition $|D(u_C)|<\epsilon$ is satisfied.
    \item \textbf{Line 21}: When both $u_L$ and $u_R$ are established, we iteratively shrink the search interval using $u_C$ and update $u_C$ according to line 23.

\end{itemize}

A new point $u_C$ is generated using a modified secant method applied to the function $D(u)$. This update step is defined by:
\begin{equation}
\label{equ: u piv update} 
u_C = 
\begin{cases}
    \texttt{Generate}(u_{LB}, u_L), \ & \text{if}\ u_R = +\infty\\
    \texttt{Generate}(u_{RB}, u_R), \ & \text{if}\ u_L = +\infty \\
    \texttt{Generate}(u_R, u_L), \ & \text{otherwise.}
\end{cases},
\end{equation}
where $z = \texttt{Generate}(x, y) = \frac{-xD(y) + yD(x)}{-D(y) + D(x)}$. This $\texttt{Generate}$ function provides the root estimate based on a linear interpolation of $D(u)$ between points $x$ and $y$. This is equivalent to finding the horizontal coordinate of the intersection between the line passing through points $(x, F_r(x))$ and $(y, F_r(y))$ and the line passing through points $(x, G_r(x))$ and $(y, G_r(x))$, as shown in Figure~\ref{fig: ui+1 update}.

In conclusion, the iterative loop in the \textbf{Pivot step} terminates exclusively when one of the following three conditions is met:
\begin{itemize}
    \item Convergence achieved: $\left|D(u_C)\right| \leq \epsilon$.
    \item A narrow interval that contains $u_r^{\star}$ found: either $\Vec{a}_{j+1} \leq u_L < u_r^{\star} \leq u_C \leq \Vec{a}_j$ or $\Vec{a}_{j+1} \leq u_C \leq u_r^{\star} < u_R \leq \Vec{a}_j$ for some $j$.
\end{itemize}

    

\textbf{Exact step} (26-29 in Algorithm~\ref{alg: EIPS}) Given the narrow interval containing $u_r^{\star}$ obtained in the Pivot step, we now proceed to determine the exact value of $u_r^{\star}$. 

Since $D(u)$ is piecewise linear and convex over this interval, we can apply a Newton-type iteration to locate the root $u_r^{\star}$. Specifically, we begin the procedure at the lower bound $u_L$. If $\left|D(u_L)\right| \leq \epsilon$, the exact value of $u_r^{\star}$ has been located, and the procedure terminates. Otherwise, we construct a linear approximation of $D(u_L)$ at $u_L$ with its right-hand derivative $D'((u_L)^+)$ and compute the intersection between the linear approximation and the horizontal axis (i.e., where the function value equals zero). This yields the update rule
\begin{equation}
\label{equ: exact step update}
u_L \gets - \frac{D(u_L)}{D'((u_L)^+)} + u_L.
\end{equation}
By repeatedly applying this update rule, the iteration converges and yields the desired solution $u_r^{\star}$. 

Upon accurately determining the intersection point $u_r^{\star}$, we classify it into either \ref{case2} or \ref{case3}, depending on the relationship between $m$ and $k$. Finally, we construct the solution vector $\vect{x}^{\star}$ using~\eqref{kkt: solution form}.

\begin{remark}
    The proposed algorithm, EIPS, can be interpreted as an adaptive variant of ESGS~\cite{roth2025n} discussed in Section~\ref{sec: intro}, which constructs a path from an initial candidate pair $(k_0', k_1') = (k-1, k)$ toward the true indices $(k_0, k_1)$. This path consists of incremental adjustments of $k_0' \gets k_0'-1$ and $k_1' \gets k_1'+1$ in each iteration and stops until certain termination conditions are satisfied. By contrast, in EIPS, the Pivot step initially employs relatively larger adjustments, i.e., effectively a larger `stepsize' for $k_0'$ in one iteration, to quickly identify a narrow interval containing $u_r^{\star}$. Once such a interval is located, corresponding to identifying $k_0$, EIPS then proceeds to the Exact step, where the adjustments become progressively smaller, and EIPS finally locates $k_1$. Thus, compared with ESGS, the stepsize in EIPS is adaptively controlled, allowing EIPS to serve as an accelerated search strategy for efficiently pinpointing the pair $(k_0, k_1)$. 
\end{remark}

\subsection{Accelerating tricks}
\label{subsec: ACC tricks}

The efficiency of EIPS is closely tied to the computations of $F_r(u)$ and $G_r(u)$. 
Repeatedly evaluating these functions at each iteration can be computationally expensive, especially when the length of $\vect{a}$ is large. To reduce this cost, we exploit the fact the interval $(u_L, u_R)$ shrinks while always containing the optimal solution $u_r^{\star}$. This ensures that elements of $\vect{a}$ lying outside the interval become irrelevant or reusable. By filtering out these unnecessary elements and reusing intermediate sums, we eliminate redundant computations and significantly speed up the evaluation.

Specifically, we begin with two initial filtering operations on $\vect{a}$ to construct the filtered sets $\vect{a}_f$ and $\vect{a}_g$, which are designed to facilitate the computation of $F_r(u)$ and $G_r(u)$, respectively. We define $\vect{a}_f$ contain all elements in $[u_L,u_R)$, and $\vect{a}_g$ to contain all elements in $[G_r(u_R),u_R)$, as these subsets are sufficient for computing $F_r(u)$ and $G_r(u)$.
The construction depends on the scenarios determined during the Initialization step:
\begin{itemize}
    \item $\textsc{Flag}=1$ ($u_R=\vec{a}_1$, $u^i\leq \vec{a}_{k+1}\leq u_r^{\star}$): set $\vect{a}_f = \left\{a_j: j \in \I(\vect{a}, u^i, u_R)\right\}$ with $B = 0$, $N = 0$; set $\vect{a}_g = \left\{a_j: j \in \I(\vect{a}, G_r(u_R), u_R)\right\}$. 
    \item $\textsc{Flag}=2$ ($u_R=+\infty$): set $\vect{a}_f = \left\{a_j: j \in \I(\vect{a}, u_L, +)\right\}$ with $B = 0$, $N = 0$; set $\vect{a}_g = \vect{a}$. 
    \item $\textsc{Flag}=3$ ($u^i\leq \vec{a}_{k+1}\leq u_r^{\star}$): set $\vect{a}_f = \left\{a_j: j \in \I(\vect{a}, u^i ,u_R)\right\}$ with $B = \sum_{i \in \I(\vect{a}, u_R, +)}a_i$, $N = \left|\I(\vect{a}, u_R, +)\right|$; set $\vect{a}_g = \left\{a_j: j \in \I(\vect{a}, G_r(u_R), u_R)\right\}$. 
\end{itemize}
Here, $B$ denotes the cumulative sum of elements in $\vect{a}$ exceeding $u_R$, and $N$ is their count. 

Given $u_C$, we compare $F_r(u_C)$ and $G_r(u_C)$ to determine whether $u_C > u_r^{\star}$. 
If $F_r(u_C) > G_r(u_C)$, then $u_C > u_r^{\star}$, and we update $u_R \gets u_C$. 
Let $\I^{1} := \I(\vect{a}_f, u_C,+)$. We then update $\vec{a}_f$ by removing all elements greater than or equal to $u_C$. We recompute $B$ and $N$, and update
$F_r(u_C)$ through:
$$
\begin{aligned}
& N \gets \left|\I^{1}\right| + N,\quad B\gets \sum_{i \in \I^{1}}a_i + B \\
& F_r(u_C) = \frac{1}{k}\left(r - B + N\cdot u_C\right),
\end{aligned}
$$
Additionally, we update $\vect{a}_g$ by containing only the elements in $[G_r(u_C), u_C)$.

Otherwise, if $u_C \le u_r^{\star}$, we set $u_L\gets u_C$ and update $\vec{a}_f$ by removing all elements smaller than $u_C$, while leaving $\vect{a}_g$ unchanged.

Through these filtering and partitioning techniques, the size of $\vect{a}_f$ continuously decreases across iterations, significantly improving computational efficiency. Figure~\ref{fig: Fru accelerating} illustrates how partitioning $\vect{a}$ and recording intermediate computations effectively accelerate the calculation of $F_r(u)$.

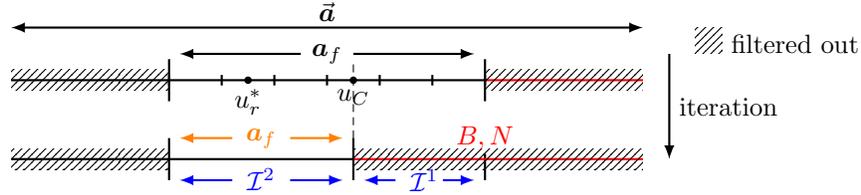
\begin{figure}[!ht]
\centering
    \begin{tikzpicture}[scale=0.7]
        \draw[thick, latex-latex] (-4,1) -- (8,1);
        \draw[thick, -latex] (8.5, 0.5) -- (8.5, -1.5) node[pos=0.5, right] {iteration};
        \node at (2, 1.3) {$\Vec{\vect{a}}$};

        \draw[thick] (-4,0) -- (5,0);       
        \draw[thick, red] (5,0) -- (8,0); 
        \fill[pattern=north east lines, pattern color=black] (5, -0.2) rectangle (8,0.2);
        \fill[pattern=north east lines, pattern color=black] (-4, -0.2) rectangle (-1,0.2);
        
        \def\xvalues{{-1, 0, 1, 2, 3, 4, 5}};
        \foreach \i in {0,1,...,6}{
            \ifnum \i=0
                \draw[thick] (\xvalues[\i], 0.4) -- (\xvalues[\i], -0.4);
            \fi
            \ifnum \i=6
                \draw[thick] (\xvalues[\i], 0.4) -- (\xvalues[\i], -0.4);
            \fi
            \draw[thick] (\xvalues[\i], 0.1) -- (\xvalues[\i], -0.1);
        }
        \draw[thick, latex-] (-0.8, 0.5) -- (1.5, 0.5);        
        \draw[thick, -latex] (2.5, 0.5) -- (4.8, 0.5);  
        \node at (2, 0.5) {$\vect{a}_f$};

        \fill (0.5, 0) circle (2pt) node[below] {$u^{\star}_r$};
        \fill (2.5, 0) circle (2pt) node[below] {$u_C$};

        \draw[thick] (-4, -1.5) -- (2.5, -1.5);
        \fill[pattern=north east lines, pattern color=black] (-4, -1.3) rectangle (-1,-1.7);
        \draw[thick] (-1, -1.1) -- (-1, -1.9);
        \draw[thick] (5, -1.4) -- (5, -1.9);
        \draw[thick, red] (2.5, -1.5) -- (8, -1.5);  
        \fill[pattern=north east lines, pattern color=black] (2.5, -1.3) rectangle (8,-1.7);
        \node[red] at (5, -1.1) {$B, N$};        
        \draw[densely dashed] (2.5, 0.3) -- (2.5, -1.4);
        \draw[thick] (2.5, -1.1) -- (2.5, -1.9);
        \draw[thick, latex-, orange] (-0.8, -1.1) -- (0.2, -1.1);    
        \node[orange] at (0.75, -1.1) {$\vect{a}_f$};
        \draw[thick, -latex, orange] (1.3, -1.1) -- (2.3, -1.1);  
        \node[blue, below] at (0.75, -1.5) {$\I^{2}$};
        \draw[thick, latex-, blue] (-0.8, -1.9) -- (0.2, -1.9);    
        \draw[thick, -latex, blue] (1.3, -1.9) -- (2.3, -1.9);  
       
        \draw[thick, latex-, blue] (2.7, -1.9) -- (3.3, -1.9);    
        \node[blue] at (3.85, -1.9) {$\I^{1}$};
        \draw[thick, -latex, blue] (4.2, -1.9) -- (4.8, -1.9);
        \begin{scope}[shift={(9,2)}]
            \fill[pattern=north east lines, pattern color=black] (0, -1.5) rectangle (0.5,-1) node[right,black, yshift=-6pt] {filtered out};
        \end{scope}
    \end{tikzpicture}
\caption{Illustration of how to accelerate calculating $F_r(u)$, where $\vect{a}$ is sorted. Given $\vect{a}_f$, if $u_C > u^{\star}_r$, the elements in $\I^{1}$ contribute to the calculation of $F_r(u^{\star}_r)$. Update cumulative sums ($B, N$) using the information in $\I^1$ and replace $\vect{a}_f$ with the element in $\I^2$, leading to a smaller size.}
\label{fig: Fru accelerating}
\end{figure}

\subsection{Convergence Analysis of Algorithm~\ref{alg: EIPS}}
\label{subsec: eips analysis}
After presenting our proposed Algorithm~\ref{alg: EIPS} and some accelerating tricks, we now establish its convergence properties. Before analyzing the convergence of the full algorithm, we begin with the following lemma, which ensures that the potential issue of undefined behavior ($F_r'((u_C)^-) > 1$) is resolved within a finite number of iterations. 

\begin{lemma}
\label{lem: finite}
    In the Pivot step (line 7 in Algorithm~\ref{alg: EIPS}), if $F_r'((u_C)^-) > 1$, then $u_R < \infty$, and after a finite number of iterations we obtain an updated point  $u_C'$ such that $F_r'((u_C')^-) \leq 1$. Moreover, $u_R - u_C' \geq \frac{1}{2}(u_R - u_r^{\star})$. 
\end{lemma}

\begin{proof}
For \textsc{Flag}=2, we have $u_L>-\infty$ and $F_r'((u_C)^-) \leq 1$ is always satisfied. Thus, if $F_r'((u_C)^-) > 1$, the case corresponds to \textsc{Flag}=1 or \textsc{Flag}=3, and we have $u_R < \infty$.
Since $u_C \leq \vec{a}_{k+1} < u_R$, the interval $[u_C,u_R]$ is repeatedly halved until its length becomes shorter than that of $[\vec{a}_{k+1},u_R]$. Therefore, the maximum number of iterations is $\left\lceil\log_2((u_R-u_C)/(u_R-\vec{a}_{k+1})) \right\rceil + 1$. Moreover, we have $u_R - u_C' \geq \frac{1}{2}(u_R - u_r^{\star})$ because otherwise we would already have $F_r'((u_C')^-) \leq 1$ in the previous iteration.


\end{proof}

Lemma~\ref{lem: finite} indicates that the undefined issue only occurs when $u_R = \infty$ and that it is resolved in finite iterations. We now proceed to establish the convergence of Algorithm~\ref{alg: EIPS}. 

\begin{theorem}
\label{thm: convergence}
    Algorithm~\ref{alg: EIPS} terminates in a finite number of iterations. Furthermore, its worst-case complexity for Algorithm~\ref{alg: EIPS} is $O((n+k)n^2)$. 
\end{theorem}

\begin{proof}
Algorithm~\ref{alg: EIPS} consists of three steps: the Initialization step, the Pivot step, and the Exact step. We prove that each step terminates in a finite number of iterations.

\textbf{Initialization step.} Recall that in Algorithm~\ref{alg: initialization}, if $F_r'((u^{i+1})^{-})= F_r'((u^i)^{-})$, we obtain $F_r(u^{i+1})=u^{i+1}$. Otherwise, by construction, $F_r'((u^{i+1})^{-}) \geq F_r'((u^i)^{-})+\frac{1}{k}$. Since $F_r'((u^0)^{-}) = 0$ initially and Algorithm~\ref{alg: initialization} terminates either $F_r'((u^{i+1})^{-}) \geq 1$ or $F_r(u^{i+1}) = u^{i+1}$, it follows that the Initialization step terminates within at most $k$ iterations.

\textbf{Pivot step.} 
The Pivot step seeks an index $j$ such that 
\(\vec{a}_{j+1} \le u_r^{\star} \le \vec{a}_j.\)
Define 
\[i_L=\min\!\left(|\mathcal{I}(\vect{a},u_L,+)|,\,k\right), \quad
i_C=\min\!\left(|\mathcal{I}(\vect{a},u_C,+)|,\,k\right), \quad
i_R=|\mathcal{I}(\vect{a},u_R,+)|.\]
When $u_R<+\infty$ and $u_L>\vec{a}_{k+1}$, we have 
\(u_r^{\star} \in [u_L,u_R] \subset (\vec{a}_{i_L+1},\,\vec{a}_{i_R}]\) and
\(0 \le i_R < i_L+1 \le k+1.\)
The Pivot step repeatedly updates $i_R$ and $i_L$ until they stabilize.
\begin{itemize}
    \item \textbf{Lines 9–15.} 
    If $F_r'((u_C)^{-}) = F_r'((u_L)^{-})$ (equivalently $i_C = i_L \le k$), the step either terminates or returns $i_C \le i_L$. 
    In the case $i_C=i_L$, we obtain $u_C = \vec{a}_{i_L}$.  
    
    If $F_r'((u_C)^{-}) = F_r'((u_R)^{-})$ (i.e., $i_R = i_C < k$), the algorithm again either terminates or returns $i_R < i_C$. 
    
    If neither condition is met, then the relation 
    \(i_R < i_C < i_L\)
    persists without change.
    \item \textbf{Lines 16–22.} 
    Since $u_C$ is reassigned to either $u_L$ or $u_R$, one of the indices $i_L$ or $i_R$ must change whenever 
    \(i_R < i_C < i_L.\)
    The only unresolved situation is $i_C = i_L$, which triggers Line~10 and forces $u_C = \vec{a}_{i_L}$. 
    After Line~23, this leads to 
    \(i_C < i_L\) and either $i_L$ or $i_R$ changes in the next iteration.
\end{itemize}

In conclusion, in each iteration of the Pivot step, we obtain either $i_R < i_C$ or $i_C < i_L$. After reassigning $u_C$, either $i_L$ increases or $i_R$ decreases. Since $0 \leq i_R < i_L \leq k$, this step terminates within at most $k$ iterations.

\textbf{Exact step.} 
Define
$
i_G(u) = |\I(\vect{a}, G_r(u), u)|.
$
After the Pivot step, we have $\Vec{a}_{j+1} \leq u_L \leq u_r^{\star} \leq \Vec{a}_j$ and $i_G(u_L) \leq i_G(u_r^{\star})$. Since $D(u)$ is piecewise linear, convex on $[\Vec{a}_{j+1}, \Vec{a}_{j}]$, and positive on $[\Vec{a}_{j+1}, u_r^{\star})$, the updating rule ensures that $u_L$ moves towards $u_r^{\star}$. Consequently, this step either terminates when $i_G(u_L) = i_G(u_r^{\star})$ or leads to a strict increase in $i_G(u_L)$. Since $i_G(u) \leq n$, the Exact step terminates within at most $n$ iterations. 

\textbf{Overall complexity. }Combining the bounds from the three steps, Algorithm~\ref{alg: EIPS} terminates within at most $(n+k)$ iterations. In addition, since the worst-case complexity for calculating $G_r(u)$ is $O(n^2)$ in both the Pivot and Exact steps, the worst-case total complexity for Algorithm~\ref{alg: EIPS} is $O((n+k)n^2)$. 
\end{proof}


Although the theoretical worst-case complexity of $O((n+k)n^2)$ established above is higher than that of previously developed methods, such as~\cite[Algorithms 1 and 3]{roth2025n} and the algorithm proposed in~\cite{luxenberg2025operator}, this bound is highly conservative and primarily reflects pathological instances. In practice, the updating rules introduced in the Pivot step, combined with the sequence filtering techniques introduced in Section~\ref{subsec: ACC tricks}, dynamically and rapidly reduces the active problem size. Because these acceleration tricks successfully filter out unnecessary elements and eliminate redundant computations, the actual number of iterations observed in our numerical tests is significantly smaller than the theoretical maximum derived here. Our numerical evaluations in Section~\ref{sec: experiment} confirm this behavior: despite the pessimistic theoretical bound, Algorithm~\ref{alg: EIPS} ultimately operates with an empirical complexity of $O(n)$.

\section{Numerical experiment}
\label{sec: experiment}
In this section, we evaluate the numerical performance of EIPS for solving the projection problem~\eqref{opt}. All numerical experiments are implemented by running Julia 1.10.2 on a MacBook Pro laptop under OS 13.0 with Apple M2 Pro and 16GB RAM.

For consistency and comparability, we adopt the experimental setup introduced in~\cite{roth2025n}. Specifically, the projection parameters are defined as $k = \tau_k \cdot n$ and $r = \tau_r \cdot \T_{(k)}(\vect{a})$, where the input sequence $\vect{a}$ are generated uniformly from $[0,1]^n$. The parameter grids are chosen as follows:
\begin{itemize}
    \item $\tau_k \in \left\{\frac{1}{10000}, \frac{1}{1000}, \frac{1}{100}, \frac{1}{20}, \frac{1}{10}, \frac{2}{10}, \frac{3}{10}, \frac{4}{10}, \frac{5}{10}, \frac{6}{10}, \frac{7}{10}, \frac{8}{10}, \frac{9}{10}\right\}$. \\[0.5pt]
    \item $\tau_r \in \left\{0, \frac{1}{10}, \frac{2}{10}, \frac{3}{10}, \frac{4}{10}, \frac{5}{10}, \frac{6}{10}, \frac{7}{10}, \frac{8}{10}, \frac{9}{10}, \frac{99}{100},\frac{999}{1000}, \frac{10}{10}, \frac{101}{100}, \frac{11}{10}, \frac{12}{10}, \frac{15}{10}, \frac{20}{10}\right\}$. 
\end{itemize}
The stopping criterion for EIPS is set to $\epsilon=10^{-8}$. We generated 100 random input vectors unless stated otherwise and independently executed the algorithm on each instance. The reported performance metrics are based on the mean computation time across these trials. 

\subsection{Initial points selecting} 

\begin{figure}
\centering
{
\resizebox*{0.48 \textwidth}{!}{\includegraphics{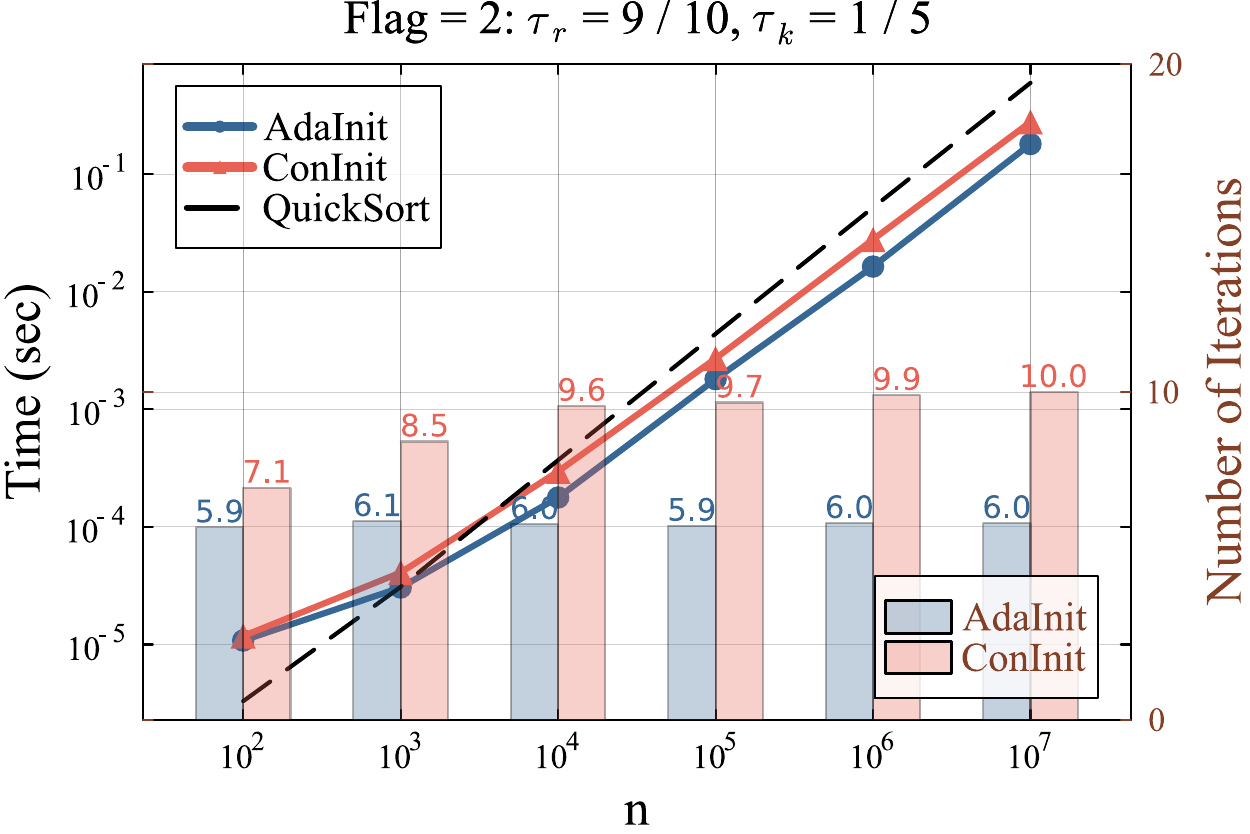}}}\hfill
{
\resizebox*{0.48 \textwidth}{!}{\includegraphics{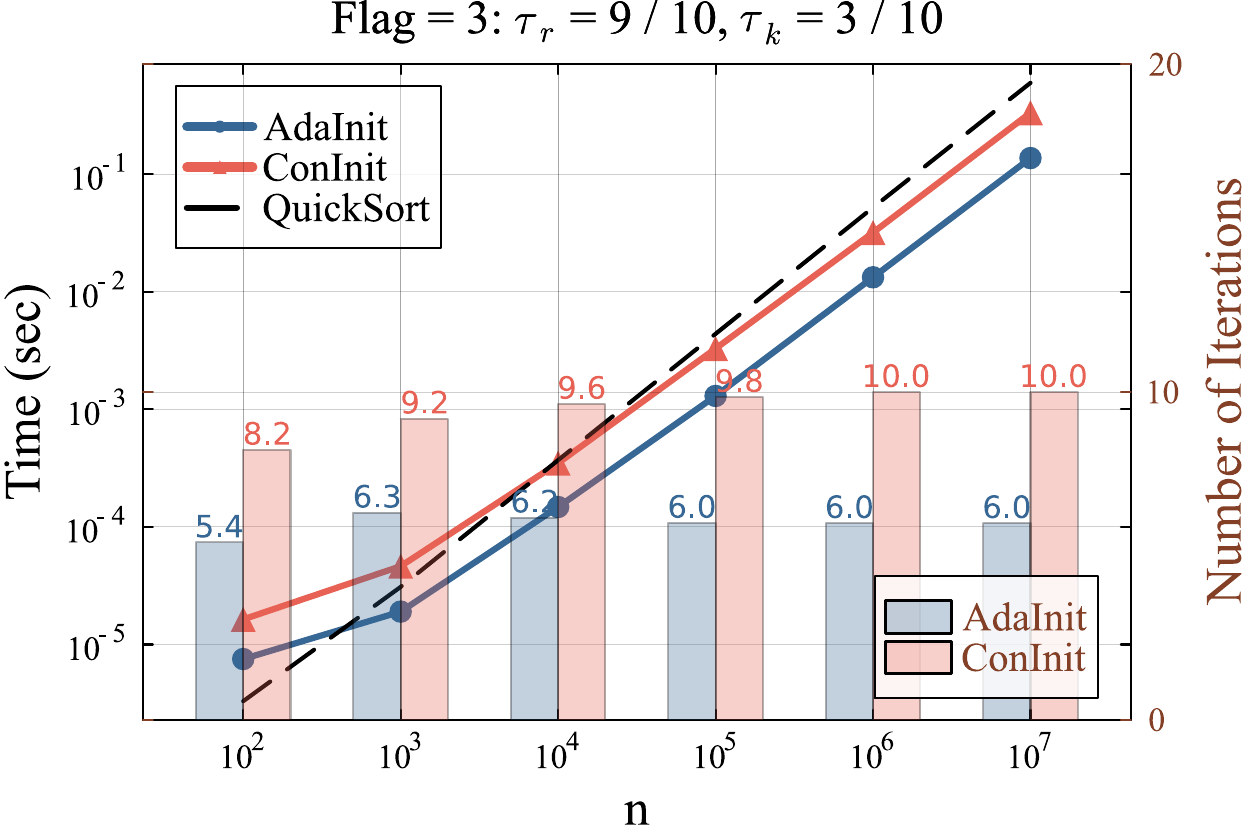}}}
\caption{Mean computation time and the number of iterations under different $\textsc{Flag}$ averaged over 100 instances, where the values of $\textsc{Flag}$ in the title are computed over 100 instances under the same $\tau_r$, $\tau_k$ on the input size from $10^2$ to $10^7$. The line chart corresponds to the left $y$-axis labeled `Time (sec)', while the bar chart corresponds to the right $y$-axis labeled `Number of Iterations'. The `Number of Iterations' represents the total iterations executed across all three steps of Algorithm~\ref{alg: EIPS}.} \label{fig: initial points comparing}
\end{figure}

We begin by examining the impact of initial point selection in the Initialization step of EIPS under the regime $r<\T_{(k)}(\vect{a})$ for \textsc{Flag} = 1, 2, 3. Specifically, we compare the following two strategies for selecting the initial points: 
\begin{itemize}
    \item $\texttt{AdaInit}$: Initial points are selected adaptively according to the \textsc{Flag} values in the Initialization step.
    \item $\texttt{ConInit}$: Initial points are selected consistently across all settings. In particular, we adopt the construction method in~\eqref{equ: initial points generating case 1} used for \textsc{Flag} = 1. 
\end{itemize}
All other components of EIPS remain unchanged, including the acceleration techniques introduced in Section~\ref{subsec: ACC tricks}.

Figure~\ref{fig: initial points comparing} shows the mean computation time and the number of iterations between these two strategies across three different \textsc{Flag} values. The results show that selecting proper initial points significantly improves performance. In particular, for $\textsc{Flag} = 2,3$, the \texttt{AdaInit} strategy consistently outperforms the \texttt{ConInit} strategy in a shorter time and a smaller number of iterations, highlighting the benefit of adaptive initialization. Moreover, we also observe that the total numbers of iterations for both strategies are much smaller than the theoretical worst-case bound of $2k+n$ shown in Theorem~\ref{thm: convergence}. This provides further evidence of the practical efficiency of our algorithm. Finally, when compared with the time required by quicksort, both initialization strategies are faster for $n \geq 10^4$, reflecting the benefit of avoiding sorting in EIPS. 

Figure~\ref{fig: time distribution} further illustrates how the \texttt{AdaInit} strategy reduces the computation time based on the time distribution across the three steps for three \textsc{Flag} values at $n=10^7$. Specifically, notable reductions in computation time are observed within the Pivot step and the Initialization step for $\textsc{Flag} = 2, 3$. Two primary factors contribute to these improvements. First, the $\texttt{AdaInit}$ strategy provides initial points that are closer to $u_r^{\star}$, allowing for a better initial lower bound $\rho$ when calculating $G_r(u)$, thus reducing computational effort. Second, a closer initial point to $u_r^{\star}$ facilitates improved filtering of the sequences $\vect{a}_f$ and $\vect{a}_g$, further decreasing the subsequent computational overhead involved in evaluating $F_r(u)$ and $G_r(u)$. 


\begin{figure}
\centering
{
\resizebox*{0.31 \textwidth}{!}{\includegraphics{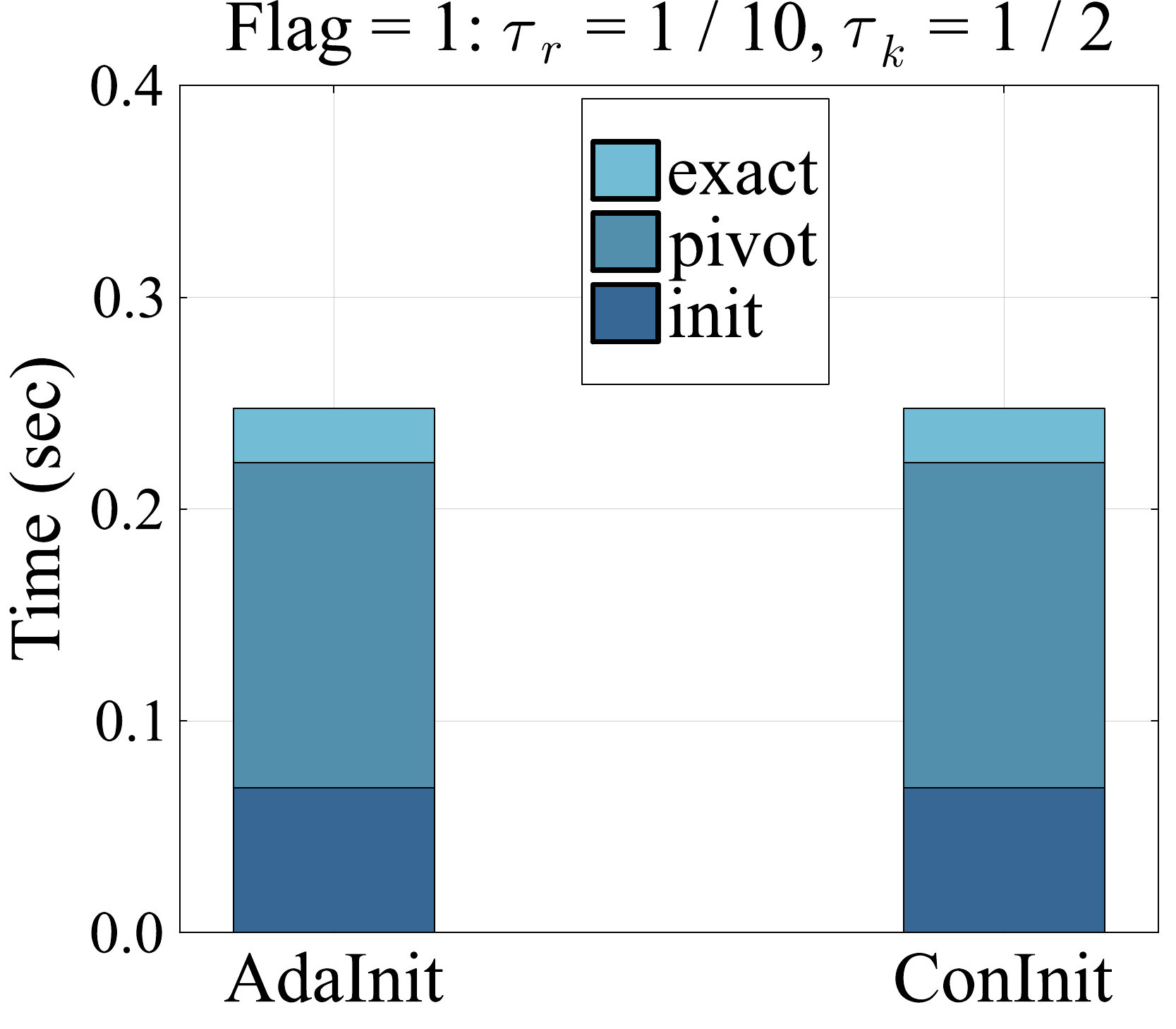}}}\hfill
{
\resizebox*{0.31 \textwidth}{!}{\includegraphics{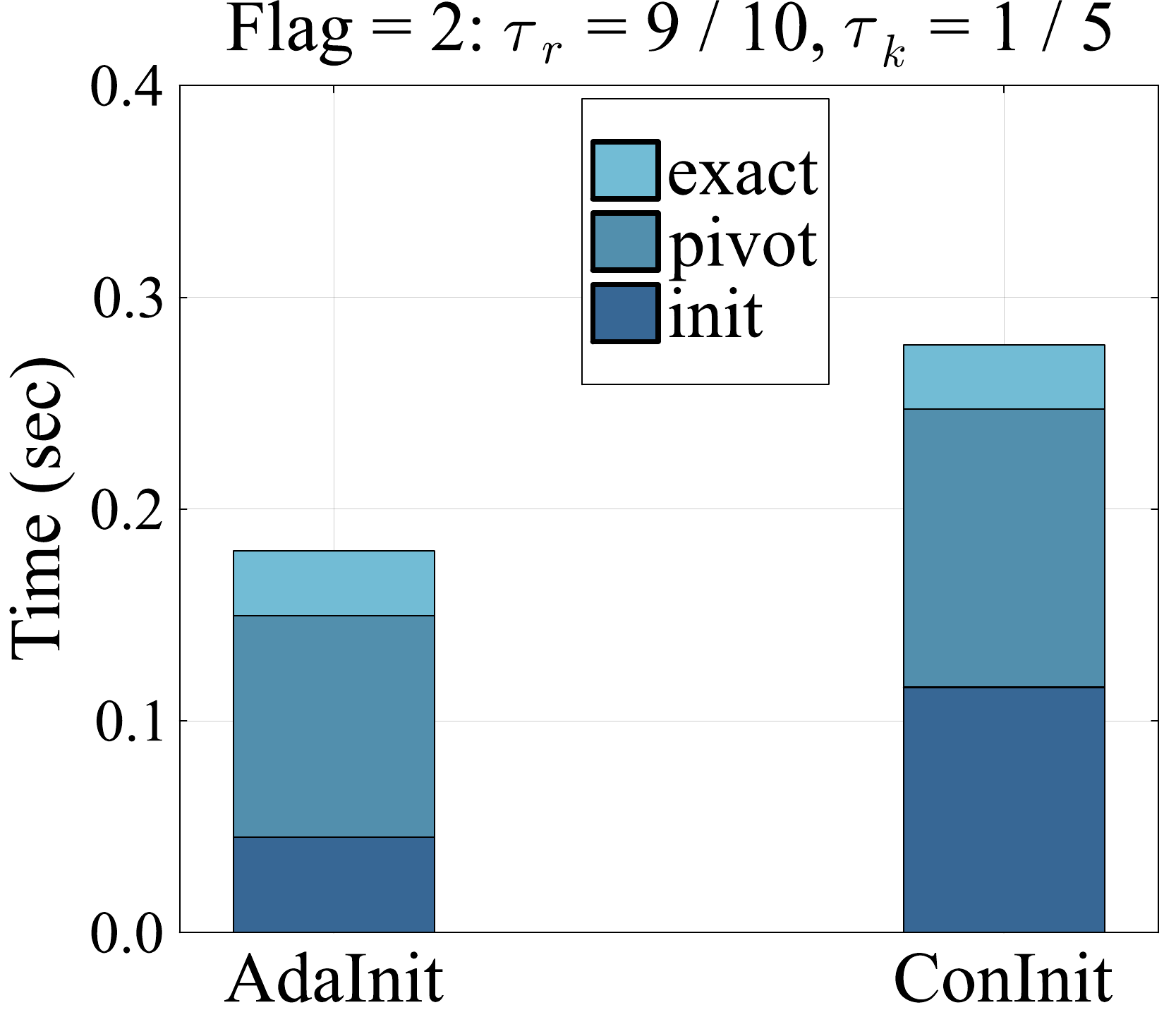}}}\hfill
{
\resizebox*{0.31 \textwidth}{!}{\includegraphics{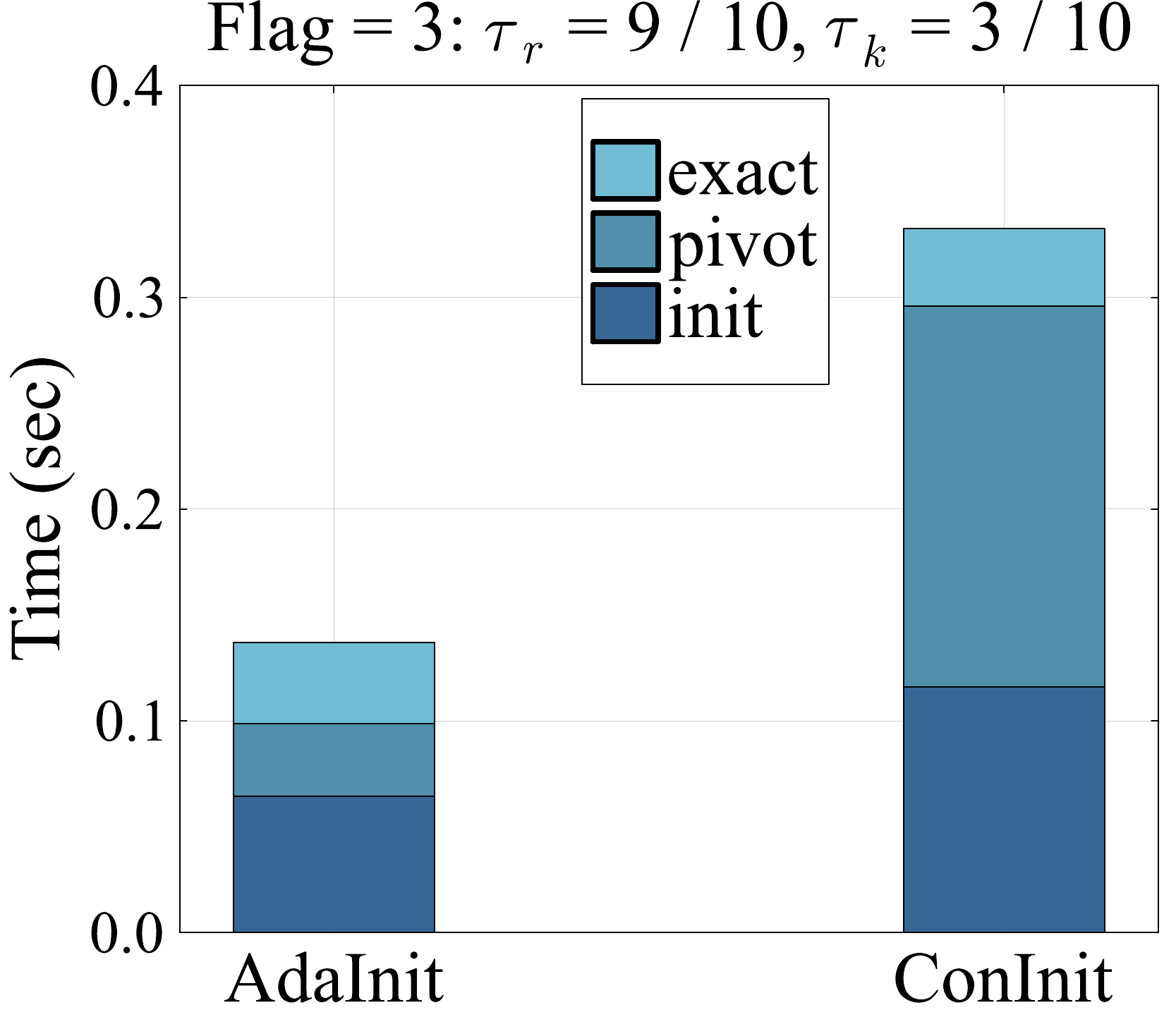}}}
\caption{Mean computation time spent in each step averaged over 100 instances at $n=10^7$.} \label{fig: time distribution}
\end{figure}

\begin{figure}[!ht]
\centering
{
\resizebox*{0.48 \textwidth}{!}{\includegraphics{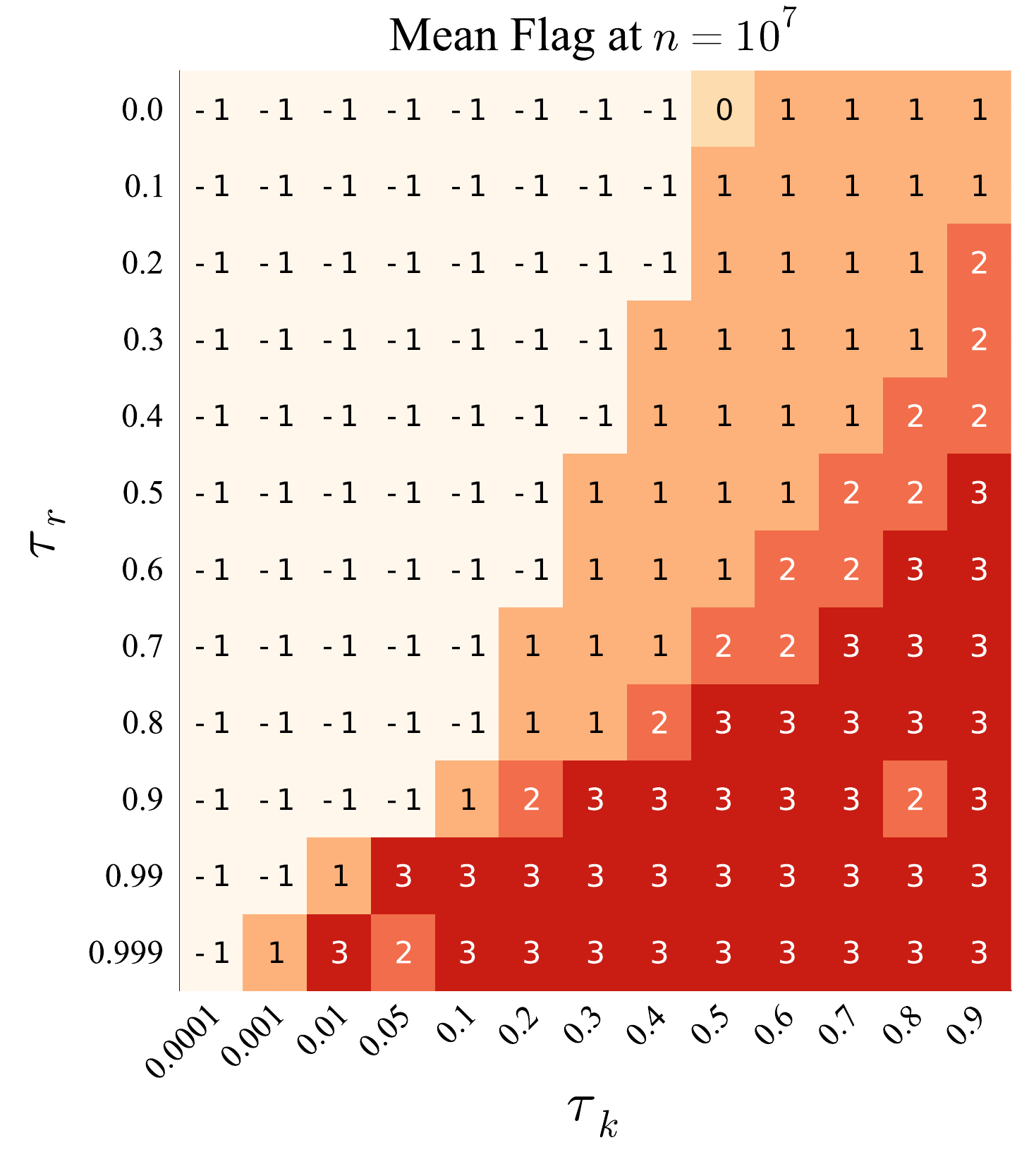}}}\hfill
{
\resizebox*{0.48 \textwidth}{!}{\includegraphics{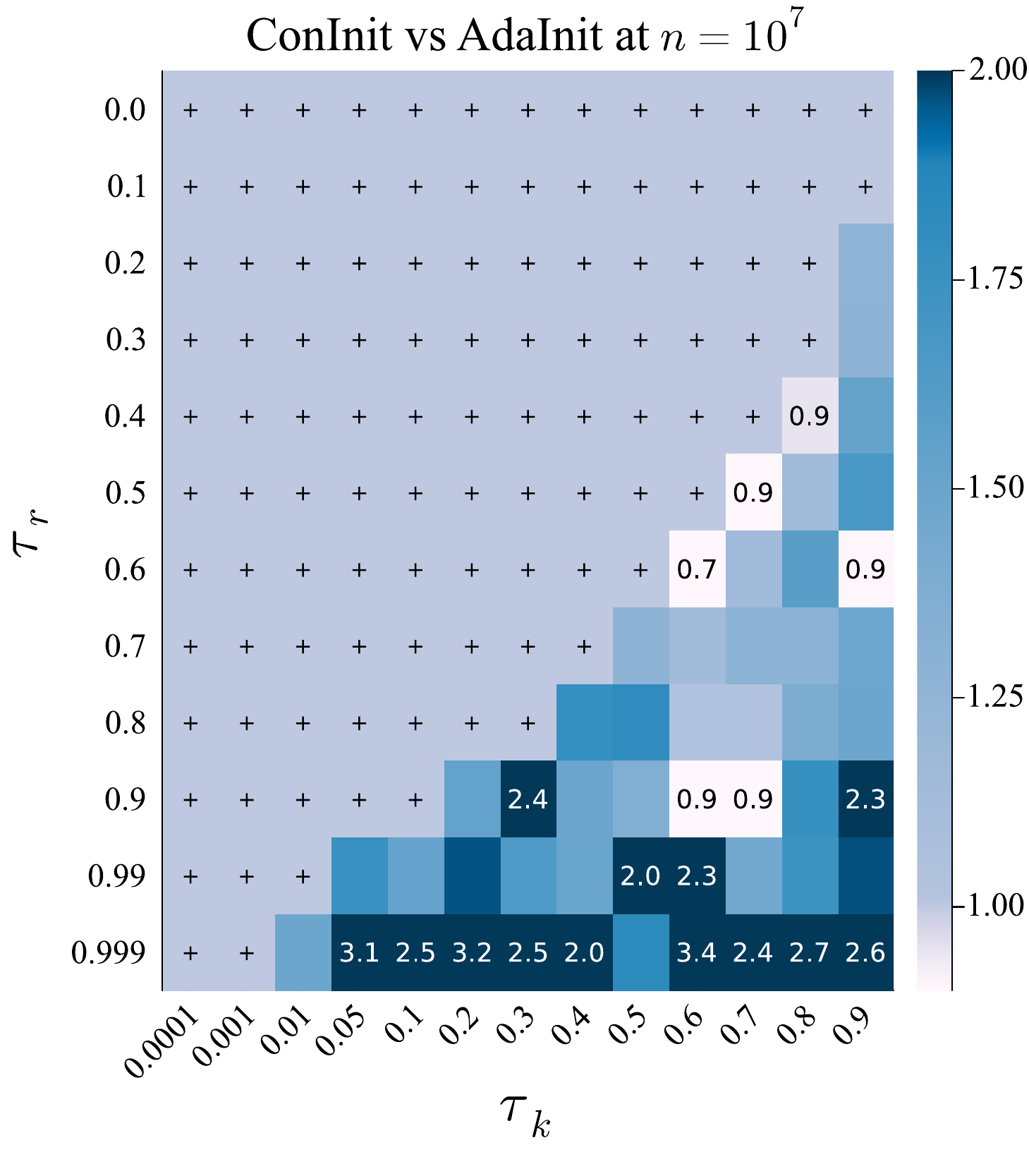}}}
\caption{(\textbf{Left}): Mean $\textsc{Flag}$ values for $\texttt{AdaInit}$ strategy at $n=10^7$ averaged over 100 instances. (\textbf{Right}): Relative computation time between $\texttt{ConInit}$ and $\texttt{AdaInit}$ strategy at $n=10^7$ averaged over 100 instances. The `+' markers indicate parameter pairs for which the corresponding mean \textsc{Flag} value equals -1, 0, or 1. A value of $c>1$ indicates $\texttt{AdaInit}$ is $c$ times faster than $\texttt{ConInit}$ strategy, while a ratio $c < 1$ indicates that $\texttt{AdaInit}$ is slower, taking $1/c$ times as long. For visualization purposes, values of $c>2$ and $c<1$ are capped at 2 and 1 respectively, and these values are shown in the heatmap.} \label{fig: initial points comparing (heatmap)}
\end{figure}

Figure~\ref{fig: initial points comparing (heatmap)} compares the relative performance between these two strategies and shows corresponding mean \textsc{Flag} values across the entire spectrum of the parameters $r$ and $k$ at $n=10^7$. It highlights that $\texttt{AdaInit}$ strategy is faster than the other one in most cases when $\textsc{Flag}=2,3$.

Given these observations, we choose \texttt{AdaInit} strategy to select initial points for our proposed Algorithm~\ref{alg: EIPS} for later experiments unless stated otherwise. 

\subsection{Practical complexity}
\label{subsec: practical complexity}

In this experiment, we evaluate the practical complexity of EIPS. The numerical results are summarized in Figure~\ref{fig: complexity}, where slopes corresponding to the complexity are estimated using the \texttt{GLM} package in Julia. In addition, an $n \log n$ reference curve is included for comparison. These results consistently suggest an empirical complexity of $O(n)$ for both scenarios $r \geq \T_{(k)}(\vect{a})$ and $r < \T_{(k)}(\vect{a})$.

In addition to the problem dimension $n$, the parameters $\tau_r$ and $\tau_k$ also affect the runtime, although their influence is much weaker. Specifically, the upper panels of Figure~\ref{fig: complexity}, corresponding to scenarios where $r < \T_{(k)}(\vect{a})$, demonstrate that $\tau_k$ and $\tau_r$ also influence computation time. Similarly, the bottom panels ($r \geq \T_{(k)}(\vect{a})$) illustrate that while decreasing the value of $k$ leads to reduced computation times for a fixed problem dimension $n$. Additionally, an increase in $\tau_r$ reduces the computational effort, since an increasing $\tau_r$ make it easier to find the intersection between the function $F_r(u)$ and $l=u$.

Notably, in the scenario $\tau_r = 1.0$ and $\tau_k = 0.0001$, the initially estimated slope (approximately 1.1) appears relatively high. However, upon focusing on larger input sizes ranging from $10^6$ to $10^8$, the slope estimate adjusts to approximately 1.02, indicating $O(n)$ complexity. Further details on this particular case $\tau_r=1.0$ are provided in Appendix~\ref{app: tau_r=1}. 

\begin{figure}[t]
    \centering
    \includegraphics[width=\linewidth]{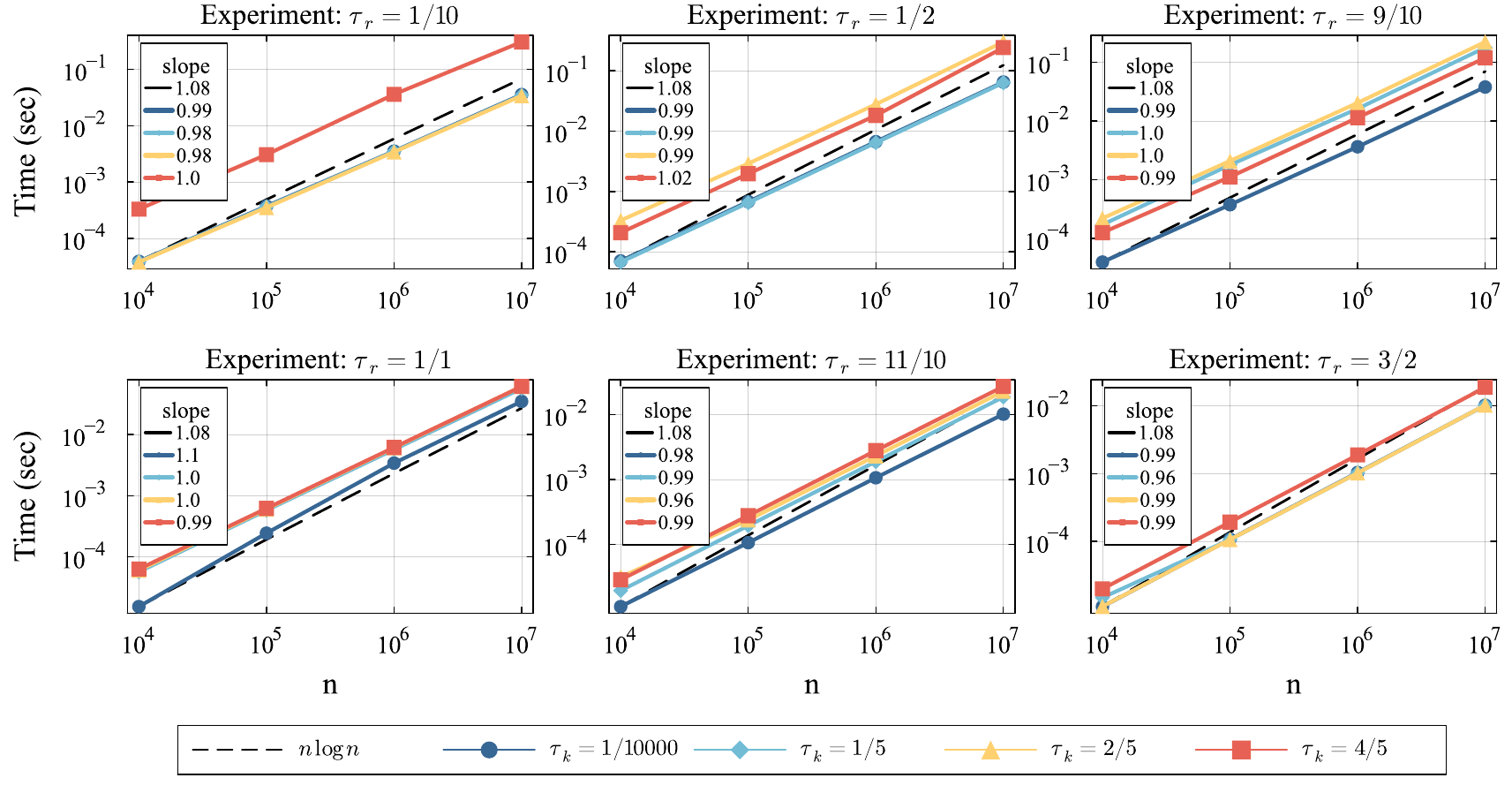}   
    \caption{Mean computation time under different $\tau_r$ and $\tau_k$ averaged over 100 instances. The upper figures indicate $r<\T_{(k)}(\vect{a})$ and the bottom figures indicate $r\geq\T_{(k)}(\vect{a})$. The legend $\emph{slope}$ represents the slope estimated using $\texttt{GLM}$ in Julia, where an approximate value of $1$ indicates $O(n)$ complexity. }
    \label{fig: complexity}
\end{figure}

\subsection{Time comparing} 
We compare the performance of EIPS against several baseline algorithms: MOVE~\cite{luxenberg2025operator}, ESGS~\cite{roth2025n}, PCLP~\cite{roth2025n}, GRID~\cite{wu2014moreau}, and GURO, the latter of which solves the sorted formulation using Gurobi 11.0.0 with feasibility and optimality tolerances set to $10^{-9}$. A maximum computation time of 3000 seconds is enforced for each instance across all algorithms.


Except for EIPS, all baseline methods rely on partial sorting to initially determine whether $r > \T_{(k)}(\vect{a})$, and on full sorting for solution computation. To provide a more conservative reference point and to tighten the empirical comparison, we adopt the following benchmarking protocol:
\begin{itemize}
    \item When $r \geq \T_{(k)}(\vect{a})$, the optimal solution $\vect{x}^{\star} = \vect{a}$ is identified immediately following the initial partial sort. 
    \item When $r < \T_{(k)}(\vect{a})$, we exclude the partial sorting time by assuming the feasibility status is known a priori. Therefore, The reported time only account for the full sort and all subsequent operations required to recover $\vect{x}^{\star}$.
\end{itemize}
In our experiments, this complete sorting step is implemented using the QuickSort algorithm in Julia. The following results show the performance for both $r \geq \T_{(k)}(\vect{a})$ and $r < \T_{(k)}(\vect{a})$.

{\bf Regime $r \geq \T_{(k)}(\vect{a})$.} 
Since all baseline algorithms employ partial sorting to determine feasibility, we focus our comparison on EIPS and one representative baseline, ESGS. The comparative results are summarized in Figure~\ref{fig: time compare nonactive} and Table~\ref{tab: time compare nonactive}.

\begin{figure}[b]
    \centering
    \includegraphics[width=\linewidth]{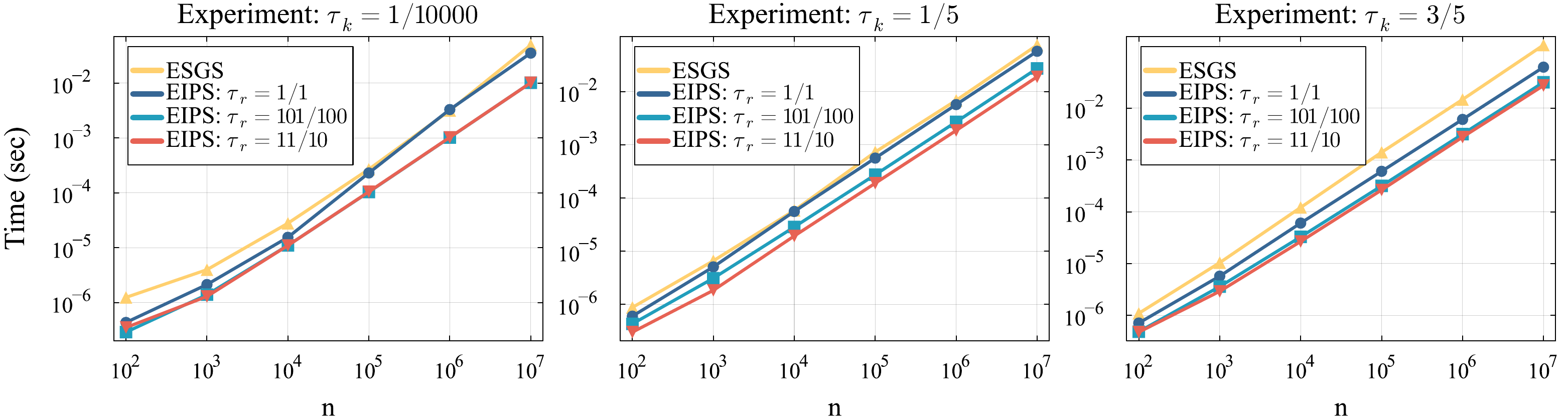}   
    \caption{Mean computation time under $r \geq \T_{(k)}(\vect{a})$ averaged over 100 instances. Since all algorithms except for EIPS use $\texttt{partialsort}$ to determine $r \geq \T_{(k)}(\vect{a})$, here we only show the comparison between ESPS and EIPS.}
    \label{fig: time compare nonactive}
\end{figure}

Figure~\ref{fig: time compare nonactive} shows that EIPS generally achieves superior computational efficiency over ESGS across a range of input sizes $n$ and ratios $\tau_r$. In addition, as discussed in Section~\ref{subsec: practical complexity}, increasing $r$ enhances the relative performance of EIPS over ESGS. While EIPS exhibits slightly longer runtime than ESGS in a few instances when $\tau_r = 1$, such cases are rare in practice. Moreover, when $r$ is only marginally greater than $\T_{(k)}(\vect{a})$, such as $\tau_r = 101/100$, EIPS consistently outperforms ESGS. 

Table~\ref{tab: time compare nonactive} further highlights the performance between EIPS and ESGS across varying values of $k$. A key observation is that for fixed $\tau_r$ and $n$, the relative efficiency of EIPS improves as $k$ increases. This growing advantage is primarily due to ESGS's dependence on partial sorting, which incurs a computational cost of $O(n +k\log n)$. For example, under $\tau_r = 11/10$ and $n = 10^7$,  EIPS is 2.3 times faster than ESGS when $\tau_k = 1/10000$ and this speedup increases to over 4 times when $\tau_k = 3/5$. 

\begin{table}[!t]
\centering
\caption{Mean computation time (standard deviation) in seconds under $r \geq \T_{(k)}(\vect{a})$ averaged over 100 instances with $r = \tau_r \cdot \T_{(k)}(\vect{a})$ and $k = \tau_k \cdot n$.}
\resizebox{\textwidth}{!}{
\begin{threeparttable}
\begin{tabular}{cllll}
    \toprule
    {\multirow{2}{*}{ }} & \multicolumn{4}{c}{Experiment: $\tau_r = 1/1,\ \tau_k = 1/10000$} \\ 
    \cmidrule{2-5} 
    ~ & $n=10^4$ & $n=10^5$ & $n=10^6$ & $n=10^7$  \\ 
    \cmidrule{1-5} 
    EIPS & \textbf{1.74e-5} (2.34e-6) & \textbf{2.67e-4} (3.47e-5) & \textbf{3.84e-3} (2.44e-4) & \textbf{4.04e-2} (1.77e-3) \\
    ESGS & 3.05e-5 (5.83e-6) & 3.19e-4 (5.80e-5) & 3.87e-3 (1.23e-3) & 5.62e-2 (5.63e-2)\\ 
    \cmidrule{1-5} 
    {\multirow{2}{*}{ }} & \multicolumn{4}{c}{Experiment: $\tau_r = 1/1,\ \tau_k = 1/5$} \\ 
    \cmidrule{2-5} 
    ~ & $n=10^4$ & $n=10^5$ & $n=10^6$ & $n=10^7$  \\ 
    \cmidrule{1-5} 
    EIPS & \textbf{5.76e-5} (1.14e-5) & \textbf{6.00e-4} (2.23e-5) & \textbf{6.21e-3} (3.33e-4) & \textbf{6.23e-2} (1.07e-3) \\
    ESGS & 6.10e-5 (8.80e-6) & 7.85e-4 (7.02e-4) & 7.49e-3 (5.58e-4) & 8.16e-2 (7.04e-3) \\ 
    \cmidrule{1-5} 
    {\multirow{2}{*}{ }} & \multicolumn{4}{c}{Experiment: $\tau_r = 1/1,\ \tau_k = 3/5$} \\ 
    \cmidrule{2-5} 
    ~ & $n=10^4$ & $n=10^5$ & $n=10^6$ & $n=10^7$  \\ 
    \cmidrule{1-5} 
    EIPS & \textbf{6.22e-5} (4.90e-6) & \textbf{6.42e-4} (2.29e-5) & \textbf{6.52e-3} (2.01e-4) & \textbf{6.72e-2} (3.80e-3) \\
    ESGS & 1.24e-4 (1.49e-5) & 1.46e-3 (9.40e-4) & 1.54e-2 (2.12e-3) & 1.72e-1 (9.65e-3) \\ 
    \cmidrule{1-5} 
    {\multirow{2}{*}{ }} & \multicolumn{4}{c}{Experiment: $\tau_r = 11/10,\ \tau_k = 1/10000$} \\ 
    \cmidrule{2-5} 
    ~ & $n=10^4$ & $n=10^5$ & $n=10^6$ & $n=10^7$  \\ 
    \cmidrule{1-5} 
    EIPS & \textbf{1.36e-5} (2.21e-6) & \textbf{1.51e-4} (2.07e-5) & \textbf{1.55e-3} (1.62e-4) & \textbf{1.51e-2} (5.60e-4) \\
    ESGS & 3.18e-5 (6.82e-6) & 3.28e-4 (5.16e-5) & 3.87e-3 (1.22e-3) & 5.62e-2 (5.64e-2)    \\ 
    \cmidrule{1-5} 
    {\multirow{2}{*}{ }} & \multicolumn{4}{c}{Experiment: $\tau_r = 11/10,\ \tau_k = 1/5$} \\ 
    \cmidrule{2-5} 
    ~ & $n=10^4$ & $n=10^5$ & $n=10^6$ & $n=10^7$  \\ 
    \cmidrule{1-5} 
    EIPS & \textbf{2.20e-5} (3.04e-6) & \textbf{2.37e-4} (1.96e-5) & \textbf{2.35e-3} (1.72e-4) & \textbf{2.41e-2} (4.84e-3) \\
    ESGS & 6.21e-5 (8.79e-6) & 7.92e-4 (7.01e-4) & 7.53e-3 (5.29e-4) & 8.17e-2 (7.13e-3)\\ 
    \cmidrule{1-5} 
    {\multirow{2}{*}{ }} & \multicolumn{4}{c}{Experiment: $\tau_r = 11/10,\ \tau_k = 3/5$} \\ 
    \cmidrule{2-5} 
    ~ & $n=10^4$ & $n=10^5$ & $n=10^6$ & $n=10^7$  \\ 
    \cmidrule{1-5} 
    EIPS & \textbf{2.91e-5} (3.24e-6) & \textbf{3.09e-4} (2.63e-5) & \textbf{3.27e-3} (3.74e-4) & \textbf{3.23e-2} (7.76e-4) \\
    ESGS & 1.24e-4 (1.51e-5) & 1.47e-3 (9.38e-4) & 1.53e-2 (2.12e-3) & 1.72e-1 (9.58e-3)\\ 
    \bottomrule
\end{tabular}
\begin{tablenotes}
    \footnotesize               
    \item The fastest algorithm is listed in bold. 
\end{tablenotes}
\end{threeparttable}
}
\label{tab: time compare nonactive}
\end{table}

\begin{figure}[!ht]
    \centering
    \includegraphics[width=\linewidth]{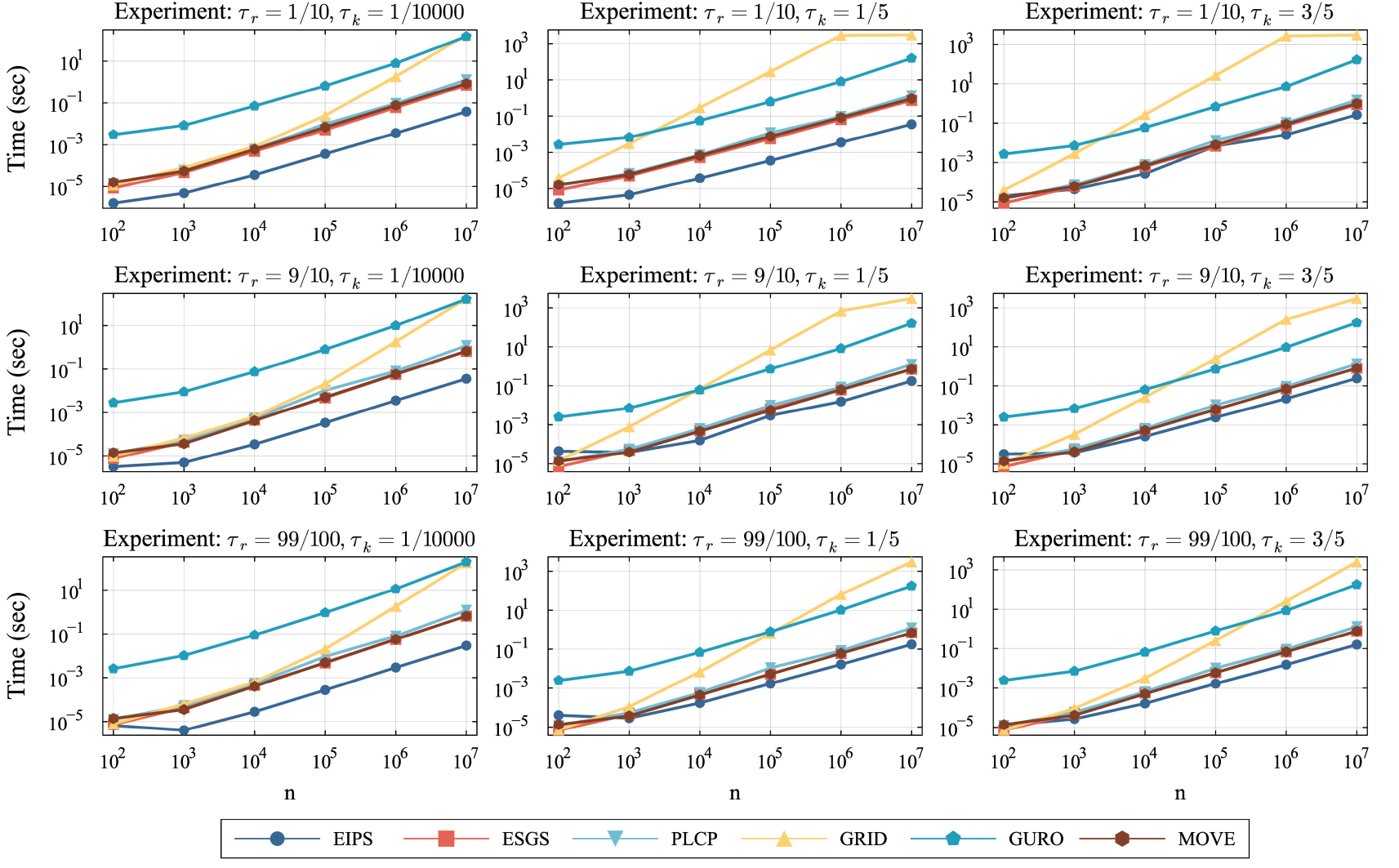}   
    \caption{Mean computation time under $r < \T_{(k)}(\vect{a})$ averaged over 100 instances, except for methods GRID and GRBU with $n \in \{10^6, 10^7\}$ in which a time-limit of $3000$ seconds is imposed across 2 instances.}
    \label{fig: time compare active}
\end{figure}

{\bf Regime $\bf r < \T_{(k)}(\vect{a})$.}
The numerical results of $r < \T_{(k)}(\vect{a})$ are summarized in Figure~\ref{fig: time compare active} and Table~\ref{tab: time compare active}. 

Across all large-scale instances ($n \geq 10^3$), EIPS consistently outperforms the baseline algorithms, with its advantages becoming increasingly pronounced as $n$ increases. As shown in Figure~\ref{fig: time compare active}, the runtime of EIPS exhibits practical $O(n)$ complexity as discussed in Section~\ref{subsec: practical complexity}. In contrast, the baseline algorithms incur higher complexity due to their reliance on sorting: MOVE, ESGS, PLCP, and GURO require $O(n\log n)$ computational cost, while GRID exhibits quadratic scaling due to its $O(k(n-k))$ complexity. Notably, MOVE and ESGS exhibit nearly identical performance across all tested scenarios. For problem size where $n \in \{10^6, 10^7\}$, the computational time of GRID extends into the range of minutes or even hours, whereas EIPS consistently completes in fractions of a second.

Table~\ref{tab: time compare active} provides detailed numerical comparisons across various $\tau_r$ and $\tau_k$ configurations. EIPS demonstrates consistent superiority over all baselines in every tested setting. For instance, under the parameter setting $\tau_r=1/10$ and $\tau_k = 1/10000$ (corresponding to \textsc{Flag} = –1), EIPS is 17 times faster than ESGS and about 4000 times faster than GURO at $n=10^7$. In another configuration $\tau_r=99/100$ and $\tau_k=3/5$, EIPS is roughly 2.5 times faster than ESGS at $n=10^4$, and this advantage increases to a factor of 4 at $n=10^7$. Moreover, we observe that the time required solely for sorting exceeds the time spent on searching the solution procedure in ESGS and MOVE in some cases and is even greater than the total runtime of EIPS. This reinforces the core advantage of EIPS: by eliminating costly preprocessing steps such as sorting, it achieves high computational efficiency and it is particularly well-suited for large-scale problems. 

{\bf Comprehensive performance comparison between EIPS and ESGS, PLCP.} Figure~\ref{fig: time compare heatmap} compares the performance of EIPS relative to ESGS (left) and PLCP (right) across the full range of parameters $k$ and $r$. The results demonstrate that EIPS consistently outperforms both ESGS and PLCP across nearly the entire parameter space. Under $r < \T_{(k)}(\vect{a})$, EIPS is at least 1.5 times faster than ESGS and at least 3 times faster than PLCP. In some configurations, this advantage increases substantially, exceeding 20 times against ESGS and 40 times against PLCP. Under $r = \T_{(k)}(\vect{a})$ and small $k$, the performance difference between EIPS and the baselines narrows, and in a few rare cases, EIPS is marginally slower. However, as $r$ becomes significantly larger than $\T_{(k)}(\vect{a})$, the performance advantage of EIPS emerges again and becomes increasingly prominent. As MOVE shows behavior similar to ESGS, we place the comprehensive comparison between EIPS and MOVE in Appendix~\ref{app: EIPS vs MOVE}.

\begin{table}[!t]
\centering
\caption{Mean computation time (standard deviation) in seconds under $r<\T_{(k)}(\vect{a})$ averaged over 100 instances with $r = \tau_r \cdot \T_{(k)}(\vect{a})$ and $k = \tau_k \cdot n$.}
\resizebox{\textwidth}{!}{
\begin{threeparttable}
\begin{tabular}{cllll}
    \toprule
    {\multirow{2}{*}{ }} & \multicolumn{4}{c}{Experiment: $\tau_r = 1/10,\ \tau_k = 1/10000$} \\ 
    \cmidrule{2-5} 
    ~ & $n=10^4$ & $n=10^5$ & $n=10^6$ & $n=10^7$  \\ 
    \cmidrule{1-5}
    EIPS & \textbf{3.66e-5} (2.73e-6) & \textbf{3.66e-4} (2.34e-5) & \textbf{3.58e-3} (1.62e-4) & \textbf{3.81e-2} (5.75e-3) \\
    ESGS & 4.59e-4 (3.11e-4) & 4.85e-3 (3.73e-4) & 5.65e-2 (1.56e-3) & 6.61e-1 (3.55e-2) \\ 
    MOVE & 5.79e-4 (1.45e-5) & 6.58e-3 (7.70e-4) & 7.34e-2 (3.86e-3) & 7.95e-1 (5.80e-2)\\
    PCLP & 5.92e-4 (2.37e-5) & 9.57e-3 (1.62e-2) & 9.06e-2 (4.31e-2) & 1.23e+0 (4.32e-2) \\
    GRID & 8.02e-4 (3.11e-4) & 2.48e-2 (1.65e-2) & 1.85e+0$^*$ (1.74e-2) & 1.80e+2$^*$ (6.32e-1) \\
    GURO & 7.09e-2 (1.01e-2) & 6.62e-1 (6.04e-2) & 7.89e+0$^*$ (1.20e-1) & 1.51e+2$^*$ (2.50e-1) \\ 
    \cmidrule{1-5}
    {\multirow{2}{*}{ }} & \multicolumn{4}{c}{Experiment: $\tau_r = 9/10,\ \tau_k = 1/5$} \\ 
    \cmidrule{2-5}
    ~ & $n=10^4$ & $n=10^5$ & $n=10^6$ & $n=10^7$  \\ 
    \cmidrule{1-5}
    EIPS & \textbf{1.59e-4} (2.04e-5) & \textbf{3.00e-3} (1.29e-2) & \textbf{1.54e-2} (1.84e-3) & \textbf{1.76e-1} (8.60e-3) \\
    ESGS & 4.35e-4 (3.12e-4) & 5.90e-3 (1.31e-2) & 5.42e-2 (1.43e-3) & 6.26e-1 (1.22e-2) \\ 
    MOVE & 4.07e-4 (1.35e-5) & 4.80e-3 (1.08e-3) & 5.63e-2 (3.28e-3) & 6.40e-1 (9.92e-3) \\
    PCLP & 5.64e-4 (3.04e-4) & 9.05e-3 (1.62e-2) & 7.76e-2 (1.78e-3) & 1.17e+0 (2.56e-2) \\
    GRID & 6.90e-2 (3.07e-3) & 6.81e+0 (9.66e-2) & 6.87e+2$^*$ (1.19e+0) & - (-)\\
    GURO & 6.17e-2 (1.01e-2) & 7.26e-1 (5.99e-2) & 8.20e+0$^*$ (5.03e-1) & 1.58e+2$^*$ (3.34e+0) \\ 
    \cmidrule{1-5}
    {\multirow{2}{*}{ }} & \multicolumn{4}{c}{Experiment: $\tau_r = 99/100,\ \tau_k = 3/5$} \\ 
    \cmidrule{2-5}
    ~ & $n=10^4$ & $n=10^5$ & $n=10^6$ & $n=10^7$  \\ 
    \cmidrule{1-5}
    EIPS & \textbf{1.64e-4} (2.09e-5) & \textbf{1.67e-3} (7.69e-4) & \textbf{1.53e-2} (1.53e-3) & \textbf{1.59e-1} (9.03e-3) \\
    ESGS & 4.36e-4 (3.13e-4) & 4.68e-3 (7.86e-4) & 5.46e-2 (1.51e-3) & 6.25e-1 (1.24e-2) \\ 
    MOVE & 3.81e-4 (2.20e-5) & 4.59e-3 (1.18e-3) & 5.42e-2 (4.10e-3) & 6.14e-1 (1.23e-2) \\
    PCLP & 5.19e-4 (2.37e-5) & 8.87e-3 (1.62e-2) & 7.60e-2 (1.76e-3) & 1.15e+0 (2.64e-2) \\
    GRID & 3.03e-3 (5.47e-4) & 2.54e-1 (1.06e-2) & 2.54e+1$^*$ (2.71e-1) & 2.51e+3$^*$ (3.25e-4) \\
    GURO & 6.63e-2 (1.26e-2) & 8.06e-1 (5.52e-2) & 8.54e+0$^*$ (2.59e-1) & 1.81e+2$^*$ (3.25e-4) \\ 
    \cmidrule{1-5}
    Sorting & 4.07e-4 (3.11e-4) & 4.33e-3 (9.33e-5) & 5.15e-2 (1.33e-3) & 5.95e-1 (1.18e-2) \\
    \bottomrule
\end{tabular}
\begin{tablenotes}
    \footnotesize               
    \item The fastest algorithm is listed in bold, $^*$ indicates statistics are computed over 2 instances, and $-$ indicates the algorithm achieves the time-limits. 
\end{tablenotes}

\end{threeparttable}
}
\label{tab: time compare active}
\end{table}


\begin{figure}
\centering
{
\resizebox*{0.48 \textwidth}{!}{\includegraphics{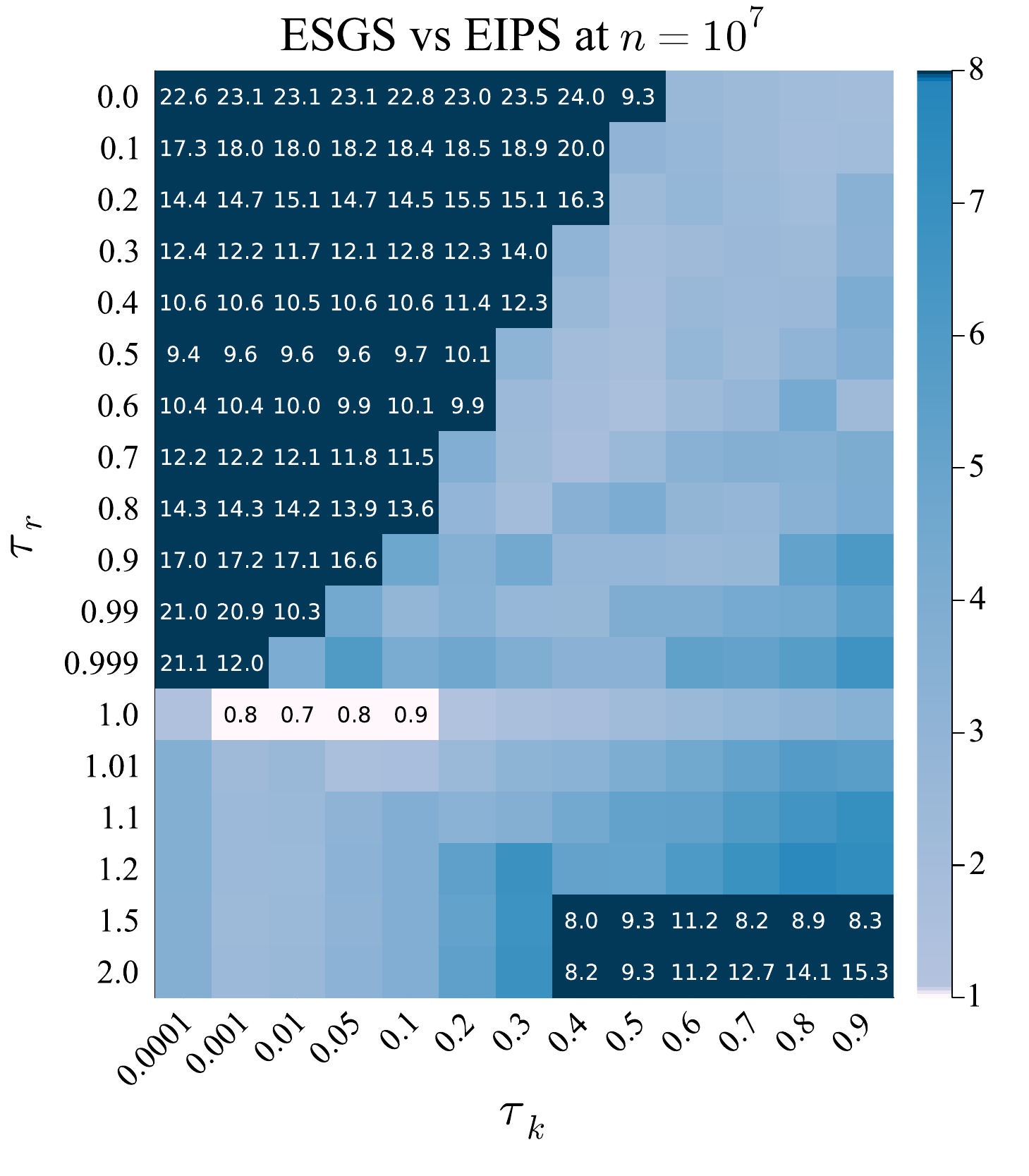}}}\hfill
{
\resizebox*{0.48 \textwidth}{!}{\includegraphics{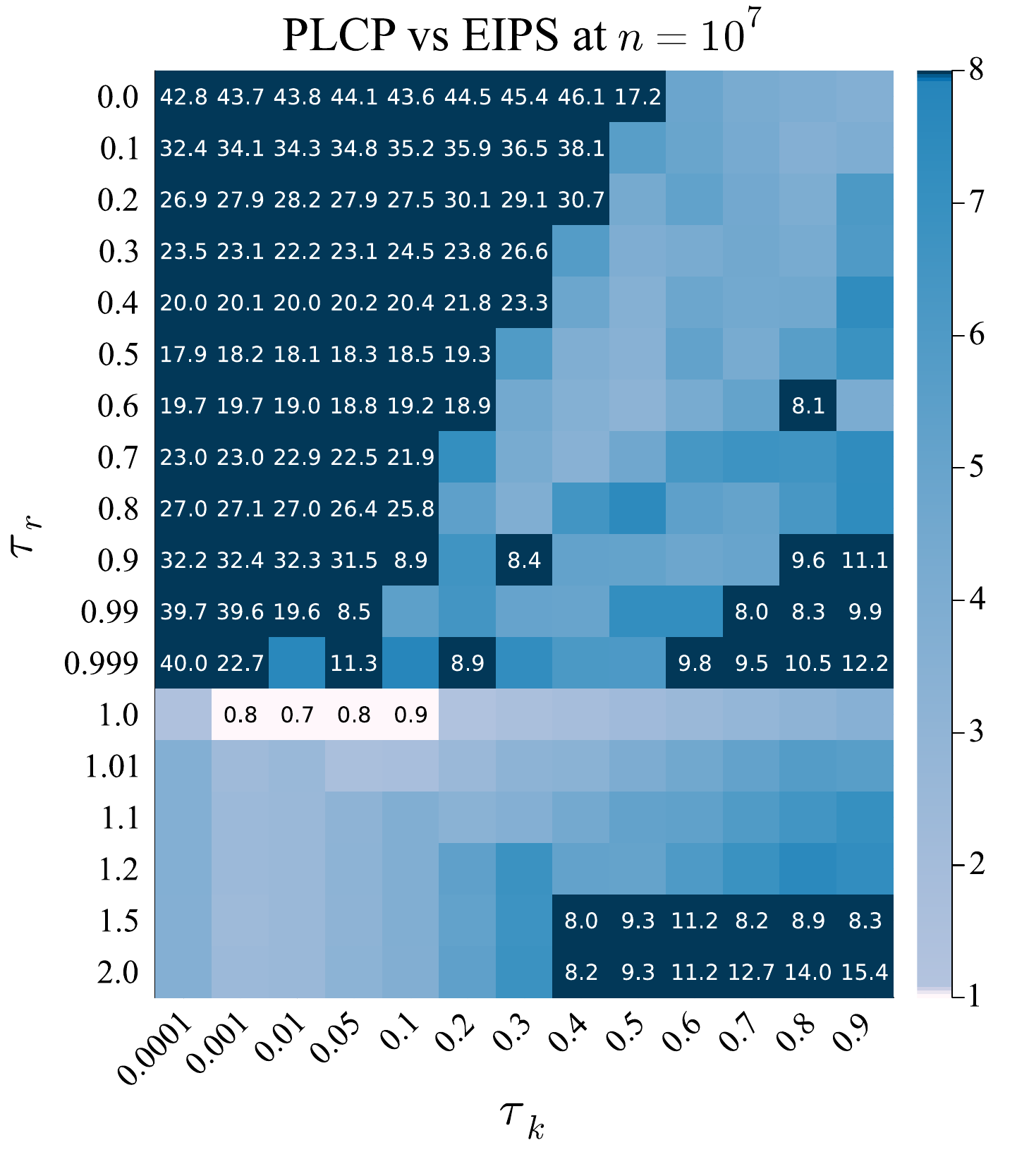}}}
\caption{Mean computation time relative to EIPS at $n=10^7$ averaged over 100 instances. A value of $c>1$ indicates EIPS is $c$ times faster than the other algorithm, while a ratio $c < 1$ indicates that EIPS is slower, taking $1/c$ times as long. For visualization purposes, values of $c>8$ and $c<1$ are capped at 8 and 1 respectively, and these values are shown in the heatmap.} \label{fig: time compare heatmap}
\end{figure}


\section{Conclusion}
We have introduced a novel projection algorithm, EIPS, to compute the Euclidean projection onto the top-$k$-sum constraint and a geometric interpretation of this algorithm. By leveraging relaxed KKT conditions, we have provided a geometric interpretation of the projection problem and have developed an iterative approach that eliminates the need for sorting, thus significantly enhancing efficiency for large-scale problems. Rigorous convergence analysis confirms that EIPS converges globally to the exact solution within a finite number of iterations. Extensive numerical experiments indicate empirical $O(n)$ complexity for EIPS. In addition, the results also demonstrate that EIPS achieves superior computational efficiency over the baseline algorithms, and its advantages become more significant as $n$ increases. Finally, future research directions include exploring the potential of applying EIPS to large-scale optimization problems related to conditional value at risk.

\newpage
\bibliographystyle{plain}
\bibliography{reference}   

\appendix

\section{KKT condition derivation for problem~\eqref{opt}}
\label{app: kkt derivation}
Without loss of generality, we assume that the initial input to the projection problem~\eqref{opt} is sorted in nonincreasing order, i.e., \( \vect{a} = \Vec{\vect{a}} \). Under this assumption, the unique solution \( \vect{x}^{\star} \) will also be sorted in nonincreasing order. Furthermore, the top-$k$-sum constraint can be relaxed to \( \sum_{i=1}^k x_i \leq r \) with \( \min\{x_i\}_{i=1}^{k-1} \geq x_k \) and \( \max\{x_j\}_{j=k+1}^n \leq x_k \). The problem is equivalent to
\begin{equation}
\label{opt-mono}
\begin{aligned}
    \min_{\vect{x} \in \R^n}\quad &\frac{1}{2}\| \vect{x} - \vect{a} \|^2 \\
    \text{subject to}\quad & x_i \geq  x_k,  \quad\forall i \leq k-1, \\
    & x_j \leq x_{k},  \quad\forall j \geq k+1, \\
    & \sum_{i=1}^k x_i \leq r,
\end{aligned}
\end{equation}

The Lagrangian function for this problem is:
\[
\mathcal{L}(x, \lambda, \alpha, \beta) = \frac{1}{2} \sum_{i=1}^n (x_i - a_i)^2 + \lambda \left( \sum_{i=1}^k x_i - r \right)
- \sum_{i=1}^{k-1} \alpha_i (x_i - x_k) + \sum_{j=k+1}^{n} \beta_j (x_j - x_k),
\]
where $\lambda \geq 0$, \( \alpha_i \geq 0 \) for \( i = 1, \dots, k-1 \), and \( \beta_j \geq 0 \) for \( j = k+1, \dots, n \) are Lagrange multipliers. 

The KKT conditions are as follows:
\begin{itemize}
    \item {Primal feasibility:}
\[
\sum_{i=1}^k x_i \leq r; \quad x_i \geq x_k\quad \forall i\leq k-1; \quad x_j\leq x_k\quad \forall j\geq k+1.
\]

\item{Dual feasibility:}
\[
\lambda \geq 0; \quad \alpha_i \geq 0 \quad\, \forall i \leq k-1; \quad \beta_j \geq 0 \quad\, \forall j \geq k+1. 
\]

\item{Stationarity:}
\begin{align}
    &x_i - a_i + \lambda - \alpha_i = 0\quad \forall i\leq k-1; \label{Sta-k-} \\
    &x_k - a_k + \lambda + \sum_{i=1}^{k-1} \alpha_i - \sum_{j=k+1}^{n} \beta_j = 0\label{Sta-k};\\
    &x_j - a_j + \beta_j = 0\quad \forall j\geq k+1. \label{Sta-k+}
\end{align}

\item{Complementary slackness:}
\[
\lambda \left( \sum_{j=1}^k x_j - r \right) = 0; \quad \alpha_i (x_i - x_k) = 0 \quad \forall i \leq k-1; \quad \beta_j (x_k - x_j) = 0 \quad \forall j \geq  k+1.
\]
\end{itemize}

Let $(\vect{x}^{\star}, \lambda^{\star}, \vect{\alpha}^{\star}, \vect{\beta}^{\star})$ satisfies the above KKT conditions, then we know from the KKT conditions:
\begin{itemize}
    \item {\bf Solve for $\alpha_i^{\star}$ for $i\leq k-1$}. If $\alpha_i^{\star}>0$, we have $x_i^{\star}=x_k^{\star}$ and condition~\eqref{Sta-k-} gives $\alpha_i^{\star}=x_k^{\star}+\lambda^{\star}-a_i>0$. If $\alpha_i^{\star}=0$, then we have $0=x_i^{\star}+\lambda^{\star}-a_i\geq x_k^{\star}+\lambda^{\star}-a_i$. Together, we have $\alpha_i^{\star}=\max\{x_k^{\star}+\lambda^{\star}-a_i,0\}$ and $x_i^{\star} = \max\{a_i - \lambda^{\star}, x_k^{\star}\}$ for all $i\leq k-1$. \\
    \item {\bf Solve for $\beta_j^{\star}$ for $j\geq k+1$}. If $\beta_j^{\star}>0$, we have $x_j^{\star}=x_k^{\star}$, and condition~\eqref{Sta-k+} gives $\beta_j^{\star}=a_j-x_k^{\star}>0$. If $\beta_j^{\star}=0$, then we have $0=x_j^{\star}-a_j\leq x_k^{\star}-a_j$. Together, we have $\beta_j^{\star}=\max\{a_j-x_k^{\star},0\}$ and $x_j^{\star} = \min\{a_j, x_k^{\star}\}$ for all $j \geq k+1$.\\ 
    \item Because $a_k\leq a_i$ for $i\leq k-1$, we have $x_k^{\star}+\lambda^{\star}-a_k\geq x_k^{\star}+\lambda^{\star}-a_i$. If $x_k^{\star}+\lambda^{\star}-a_i<0$ for all $i\leq k-1$, we have $\alpha_i^{\star}=0$ for all $i\leq k-1$. In this case, the condition~\eqref{Sta-k} shows that $x_k^{\star}+\lambda^{\star}-a_k\geq 0$. Thus we always have  $x_k^{\star}+\lambda^{\star}-a_k\geq 0$ and we can write it as $x_k^{\star}+\lambda^{\star}-a_k=\max\{x_k^{\star}+\lambda^{\star}-a_k,0\}$.\\
    \item Because $a_k\geq a_j$ for $j\geq k+1$, we have $a_k-x_k^{\star}\geq a_j-x_k^{\star}$. If $a_j-x_k^{\star}<0$ for all $j\geq k+1$, we have $\beta_j^{\star}=0$ for all $j\geq k+1$. In this case, the condition~\eqref{Sta-k} shows that $a_k-x_k^{\star}\geq \lambda^{\star}\geq 0$. Thus, we always have $a_k-x_k^{\star}\geq 0$.
\end{itemize}
Therefore, the condition~\eqref{Sta-k} becomes $x_k^{\star}\leq a_k\leq x_k^{\star}+\lambda^{\star}$ and
\[     \sum_{i=1}^{k}\max\{x_k^{\star} +\lambda^{\star} -a_i,0\} - \sum_{j=k+1}^{n} \max\{a_j-x_k^{\star}, 0\} = 0.\\ \]

In addition, if $\lambda^{\star} = 0$, then we know $\vect{x}^{\star} = \vect{a}$ and $\alpha_i^{\star} = \beta_j^{\star} = 0$ for all $i \leq k-1$, $j \geq k+1$ and $\sum_{i=1}^k x_i^{\star} = \sum_{i=1}^k a_i < r$. On the other hand, if $\lambda^{\star} > 0$, we know $\sum_{i=1}^{k} x_i^{\star} = \sum_{i=1}^{k} \max\{a_i - \lambda^{\star} - x_k^{\star}, 0\} + kx_k^{\star} = r$, where the first equality holds since $x_k^{\star} \geq a_k - \lambda^{\star}$. Combine these two cases, we have
\[ \sum_{i=1}^{n} \max\{a_i-x_k^{\star}-\lambda^{\star},0\} +kx_k^{\star} = \sum_{i=1}^{k} \max\{a_i-x_k^{\star}-\lambda^{\star},0\} +kx_k^{\star} = \min\left\{\sum_{i=1}^k a_i, r\right\},\] 
where the first equality holds since $x_k^{\star} \geq a_k - \lambda^{\star} \geq a_j - \lambda^{\star}$ for all $j \geq k+1$. 

Setting $l^{\star} := x_k^{\star}$ and $u^{\star} := \lambda^{\star} + l^{\star}$ yields the KKT conditions in~\eqref{kkt: kkt conditions} when $\vect{a}$ is in nonincreasing order. Let
$$
k_0 = |\{i: a_i > u^{\star}\}| \quad \text{and} \quad k_1 = |\{i: a_i \geq l^{\star}\}|,
$$
then the solution $\vect{x}^{\star}$ is obtained:
$$
\vect{x}^{\star}_{1:k_0} = \vect{a}_{1:k_0} - \lambda^{\star} \1_{k_0}, \quad \vect{x}^{\star}_{k_0+1 : k_1} = l^{\star}\1_{k_1 - k_0},\quad \vect{x}^{\star}_{k_1+1:n} = \vect{a}_{k_1+1:n}. 
$$

\section{Additional experiment details}
\subsection{Experiment under $\tau_r=1$}
\label{app: tau_r=1}
To further validate the observed complexity, we conducted additional experiments under $\tau_r=1.0$, i.e., $r = \T_{(k)}(\vect{a})$, with results depicted in Figure~\ref{fig: complexity(r=1)}. This figure provides estimated slope for various $\tau_k$ values when $n$ ranges from $10^6$ to $10^8$, clearly demonstrating that Algorithm~\ref{alg: EIPS} maintains linear computational scaling across all tested configurations.

\begin{figure}[!ht]
    \centering
    \includegraphics[width=0.7\linewidth]{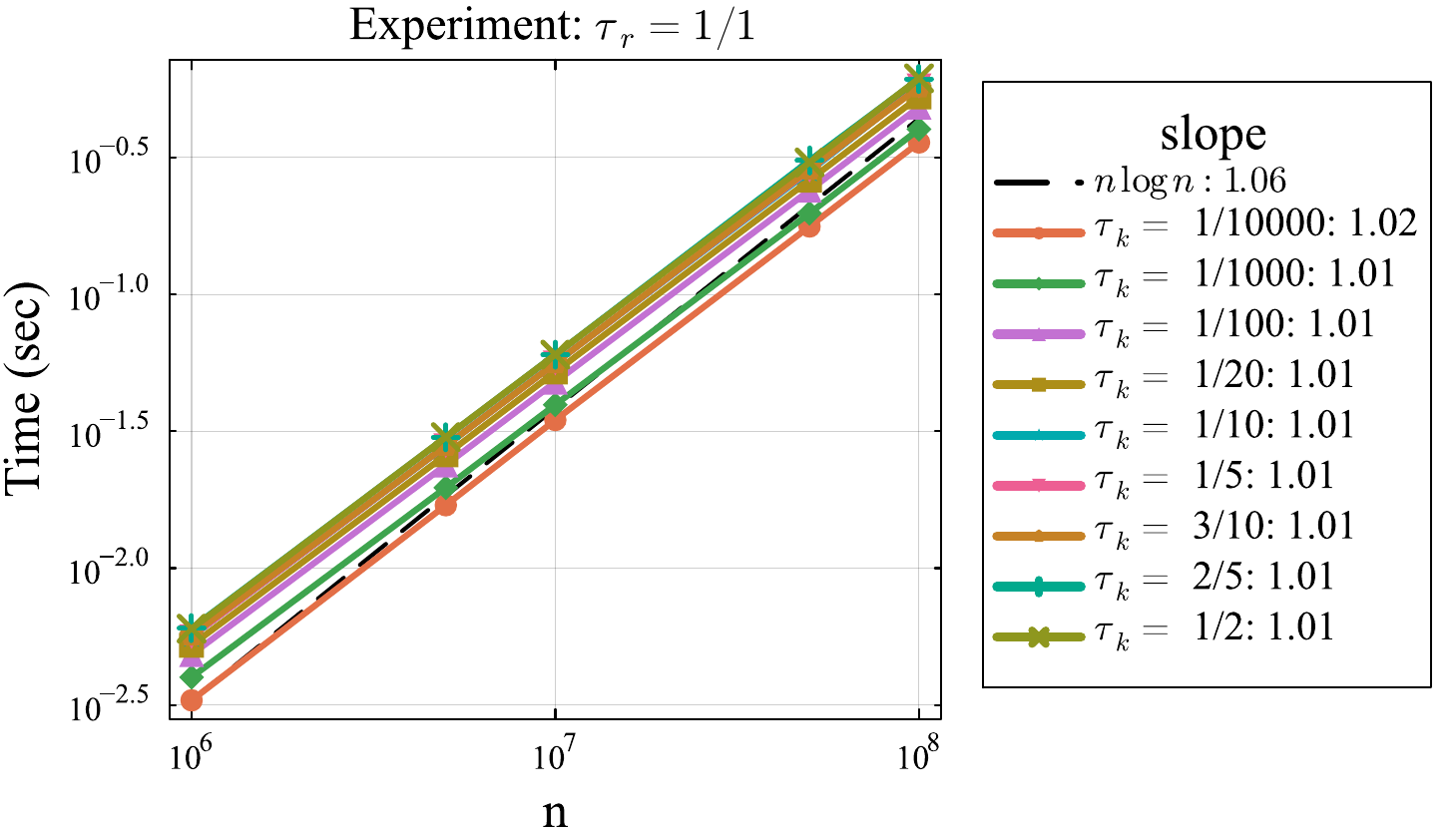}   
    \caption{Mean computation time under $r = \T_{(k)}(\vect{a})$ and $k= \tau_k \cdot n$ averaged over 100 instances. The legend $\emph{slope}$ represents the slope estimated using $\texttt{GLM}$ in Julia, where an approximate value of $1$ indicates $O(n)$ complexity. }
    \label{fig: complexity(r=1)}
\end{figure}

\subsection{Comprehensive performance comparison between EIPS and MOVE}
\label{app: EIPS vs MOVE}

We also include a comparison between EIPS and  MOVE. The results are shown in Figure~\ref{fig: EIPS vs MOVE}. Overall, the performance pattern of MOVE is very similar to that of ESGS. Consequently, the relative behavior between EIPS and MOVE closely matches the comparison between EIPS and ESGS presented in Figure~\ref{fig: time compare heatmap} (left). EIPS remains consistently faster than MOVE across most of the parameter space.

\begin{figure}[!ht]
    \centering
    \includegraphics[width=0.7\linewidth]{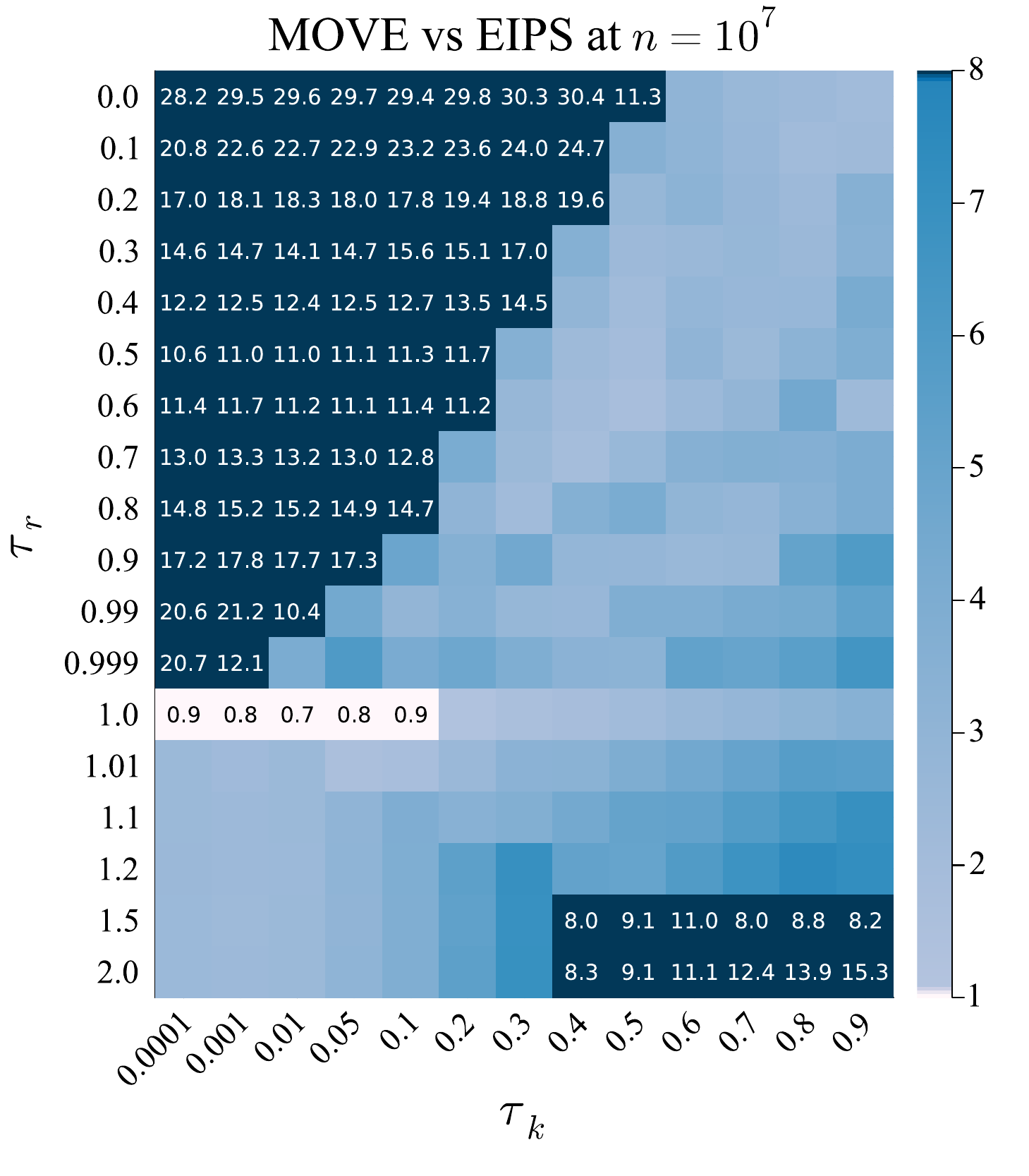}   
    \caption{Mean computation time relative to EIPS at $n=10^7$ averaged over 100 instances. A value of $c>1$ indicates EIPS is $c$ times faster than MOVE, while a ratio $c < 1$ indicates that EIPS is slower, taking $1/c$ times as long. For visualization purposes, values of $c>8$ and $c<1$ are capped at 8 and 1 respectively, and these values are shown in the heatmap.}
    \label{fig: EIPS vs MOVE}
\end{figure}

\end{document}